%% file: SharpInterfaceLimitNSAC.tex
\documentclass[a4paper,twoside]{article}
\usepackage{makeidx}
\usepackage{amssymb,amsfonts,amsmath,amsbsy,theorem,amscd,dina42e}
\usepackage{epsfig,psfrag,graphicx,xcolor}
\definecolor{CMUrot}{RGB}{128,18,18}
\definecolor{Gold}{RGB}{238,180,34}

\newcommand{\ol}[1]{\overline{#1}}

\numberwithin{equation}{section}
\newcommand{\C}{\ensuremath{\mathbb{C}}}
\newcommand{\R}{\ensuremath{\mathbb{R}}}

\newcommand{\Rn}{\ensuremath{\mathbb{R}^2}}

\newcommand{\Om}{\ensuremath{\Omega}}
\newcommand{\N}{\ensuremath{\mathbb{N}}}

\newcommand{\dist}{\operatorname{dist}}
\newcommand{\sdist}{\operatorname{sdist}}

\newcommand{\sd}{{\rm d}}

\newcommand{\eps}{\ensuremath{\varepsilon}}
\newcommand{\weight}[1]{\langle #1\rangle}

\newcommand{\Div}{\operatorname{div}}

\newcommand{\T}{\ensuremath{\mathbb{T}}}
\newcommand{\Tn}{\Omega}

\newcommand{\no}{\mathbf{n}}
\newcommand{\g}{\mathbf{g}}

\newcommand{\tn}[1]{\mathbf{#1}}
\newcommand{\ve}{\mathbf{v}}
\newcommand{\we}{\mathbf{w}}
\newcommand{\bfvarphi}{\boldsymbol{\varphi}}
\newcommand{\btau}{{\boldsymbol{\tau}}}
\newtheorem{thm}{THEOREM}[section]
\newtheorem{cor}[thm]{Corollary}
\newtheorem{lem}[thm]{Lemma}
\newtheorem{defn}[thm]{Definition}
\newtheorem{theorem}[thm]{Theorem}
\newtheorem{prop}[thm]{Proposition}

\newtheorem{claim*}{Claim}

\theorembodyfont{\rmfamily}
\newtheorem{rem}[thm]{Remark}

\newenvironment{proof*}[1]{{\bf Proof
#1:}}{\hspace*{\fill}\rule{1.2ex}{1.2ex}\\ }
\newenvironment{proof}{{\bf
Proof:\,}}{\hspace*{\fill}\rule{1.2ex}{1.2ex}\\ }

\newcommand{\First}[1]{{\color{red}#1}}
\newcommand{\Second}[1]{{#1}}
\newcommand{\Third}[1]{{\color{blue}#1}}

\newcommand{\p}{\partial}
\newcommand{\G}{\Gamma}
\def\({\left(}
\def\){\right)}
\renewcommand{\O}{\Omega}
\newcommand{\Lgrad}{L^{\nabla}}
\newcommand{\Ldelta}{L^{\Delta}}
\newcommand{\Lt}{L^{t}}

\newcommand{\divtau}{\Div_\btau}
\newcommand{\tc}{\hat{c}}
\newcommand{\tr}{\hat{r}}

\newcommand{\order}{N}
\newcommand{\zg}{\zeta\circ d_\Gamma}
\begin{document}
\begin{titlepage}
\title{Sharp Interface Limit for a Stokes/Allen-Cahn System}
\author{Helmut Abels  and Yuning Liu}
\end{titlepage}

\maketitle
\abstract{We consider the sharp interface limit of a coupled Stokes/Allen-Cahn system, when a parameter $\eps>0$ that is proportional to the thickness of the diffuse interface tends to zero, in a two dimensional bounded domain. For sufficiently small times we prove convergence of the solutions of the Stokes/Allen-Cahn system to solutions of a sharp interface model, where the interface evolution is given by the mean curvature equation with an additional convection term coupled to a two-phase Stokes system with  an additional contribution to the stress tensor, which describes the capillary stress. To this end we construct a suitable approximation of the solution of the Stokes/Allen-Cahn system, using three levels of the terms in the formally matched asymptotic calculations, and estimate the difference with the aid of a suitable refinement of a spectral estimate of the linearized
Allen-Cahn operator. Moreover, a careful treatment of the coupling terms is needed.}

{\small\noindent
}

\setcounter{tocdepth}{2}
\tableofcontents

\input{SharpInterfaceLimitNSACIntro}

\input{SharpInterfaceLimitNSACOverview}

\input{SharpInterfaceLimitNSACPrelim}

\input{SharpInterfaceLimitNSACApproxStokes}
\input{SharpInterfaceLimitNSACApproxAC}

\input{SharpInterfaceLimitNSACDifference}

\appendix
\section{Appendix}

\input{SharpInterfaceLimitNSACAppendix}


\def\ocirc#1{\ifmmode\setbox0=\hbox{$#1$}\dimen0=\ht0 \advance\dimen0
  by1pt\rlap{\hbox to\wd0{\hss\raise\dimen0
  \hbox{\hskip.2em$\scriptscriptstyle\circ$}\hss}}#1\else {\accent"17 #1}\fi}

\bigskip

\noindent
{\it
  (H. Abels) Fakult\"at f\"ur Mathematik,
  Universit\"at Regensburg,
  93040 Regensburg,
  Germany}\\
{\it E-mail address: {\sf helmut.abels@mathematik.uni-regensburg.de} }\\[1ex]
{\it (Y. Liu) NYU-ECNU Institute of Mathematical Sciences at NYU Shanghai, Shanghai, China}\\{\it E-mail address: {\sf yl67@nyu.edu}}
\end{document}

%% file: SharpInterfaceLimitNSACIntro.tex
\section{Introduction, Main Result, and Overview}
Two-phase flows of macroscopically immiscible fluids play an important role in real world applications e.g.~in chemistry and engineering sciences. They lead to many interesting fundamental questions concerning modeling, numerical simulations and their mathematical analysis. There are two basic model classes: The so-called sharp and diffuse interface models. In sharp interface models the interface separating two fluids is described as a lower dimensional surface. In numerical simulations and the mathematical analysis the interface is usually  either parametrized explicitly or described with the aid of the level set or characteristic function of one fluid domain. This leads to fundamental problems, when the interface develops singularities e.g.\ due to droplet collision or pinch-off. In diffuse interface models (also called phase field models) a partial mixing of the two fluids on a small length scale proportional to a parameter $\eps>0$ is taken into account. To this end an order parameter, which will be denoted by $c_\eps$, is introduced, which is close to one of two distinct values (e.g.~$\pm 1$) in the bulk phases of the fluids and which varies smoothly between these two values in an interfacial region, which has -- at least heuristically/in sufficiently smooth situations -- a thickness proportional to $\eps$. A fundamental diffuse interface model for a two-phase flow of two macroscopically immiscible, viscous Newtonian fluids with same densities is given by the so-called \emph{model H}, which yields the Navier-Stokes/Cahn-Hilliard system
\begin{alignat}{2}\label{eq:NSCH1}
 \rho \partial_t \ve_\eps+\rho \ve_\eps\cdot \nabla \ve_\eps -\Div (2\nu(c_\eps)D\ve_\eps) +\nabla p_\eps & = -\eps \Div (\nabla c_\eps \otimes \nabla c_\eps)&\quad & \text{in}\ \Omega\times(0,T_1),\\\label{eq:NSCH2}
  \Div \ve_\eps& = 0&\quad & \text{in}\ \Omega\times(0,T_1),\\\label{eq:NSCH3}
 \partial_t c_\eps +\ve_\eps\cdot \nabla c_\eps & =m_\eps\Delta \mu_\eps&\quad & \text{in}\ \Omega\times(0,T_1),\\\label{eq:NSCH4}
 \mu_\eps &= -\eps \Delta c_\eps + \frac1{\eps} f'(c_\eps)&\quad & \text{in}\ \Omega\times(0,T_1)
\end{alignat}
together  with suitable boundary and initial conditions for a suitable  double well potential $f$, e.g. $f(s)=\frac 18(s^2-1)^2$. Here $\ve_\eps, p_\eps$ are the velocity and the pressure of the fluid mixture, $\mu_\eps$ is a chemical potential and $c_\eps$ is related to the concentration difference of the fluids. Moreover,  $D\ve_\eps= \frac12 (\nabla \ve_\eps+\nabla \ve_\eps^T)$, $\rho>0$ is a constant, $\nu\colon \R\to (0,\infty)$ a suitable function describing the viscosity, $\nabla=\nabla_x$, $\Div=\Div_x$, $\Delta=\Delta_x$ are always taken with respect to $x\in\Omega$.

Nowadays there are many results on existence of smooth solutions for sharp interface models of two-phase flows of viscous incompressible fluids for short times and sufficiently smooth initial data and for large times and initial data close to a stable equilibrium, cf.\ e.g.\ \cite{DenisovaTwoPhase,KoehnePruessWilkeTwoPhase,PruessSimonettTwoPhaseFlow,StrongNSMS,PruessShibataShimizuTwoPhaseFlow1}.
Moreover, there are some results on existence of generalized/weak solutions, cf.\ \cite{PlotnikovTwoPhase,MultiFluidFlow,GeneralTwoPhaseFlow,ReviewGeneralTwoPhaseFlow,NSMS,HoffmannStarovoitovStefanConvection}.
But in most cases a satisfactory theory on weak solutions is unknown.
On the other hand diffuse interface models for two-phase flows, such as \eqref{eq:NSCH1}-\eqref{eq:NSCH4} and generalizations of it, were studied intensively during the last decade. There are many results on existence and uniqueness of weak and strong solutions as well as long time behavior, which are comparable to the known results for the classical incompressible Navier-Stokes system in two and three space dimensions and phase field models without fluid mechanics, cf.\ e.g.\  \cite{StarovoitovModelH,BoyerModelH,ModelH,GalGrasselliModelH2D,GalGrasselliTrajAttr}.
We also refer to \cite[Section on ``Weak Solutions and Diffuse Interface Models for Incompressible Two-phase Flows'']{Handbook}

Although there are many analytic results on sharp and diffuse interface models for two-phase flows in fluid mechanics, there are only few rigorous results on convergence of solutions of diffuse interface models to sharp interface models as $\eps\to 0$.  Most results so far are based on the method of formally matched asymptotic expansions, where the validity of certain power series expansions close to the interface is assumed, cf. \cite{StarovoitovSharpInterfaceLimes,LowengrubQuasiIncompressible, AbelsGarckeGruen2}  for the model \eqref{eq:NSCH1}-\eqref{eq:NSCH4}.
First results on convergence for large times to so-called varifold solutions of the sharp interface models were obtained in \cite[Appendix]{NSMS} and \cite{ChenSharpInterfaceLimit} for the model H and a generalization of it for fluids of different densities. A disadvantage of these results is that the notion of varifold solutions, which are comparable to measure-valued solutions, is rather weak and no convergence rates can be shown. We note that results on non-convergence for certain scalings of a mobility coefficient were obtain in \cite{ChenSharpInterfaceLimit} and \cite{ShibataProceedings} or \cite[Chapter 5]{PromotionStefan}. A sharp interface limes  for a Navier-Stokes/phase field system to a Navier-Stokes/Stefan system was proved by Starovoitov and Hoffmann~\cite{PhaseTransitionHoffmannStarovoitov} on the level of weak solutions. But the technique relies rather specifically on structural properties of the systems.

A similar sharp interface limit also arises in the theory of liquid crystals \cite{MR3376794}. In the low temperature regime, the Landau-De Gennes theory predicts the co-existence of an isotropic phase and a nematic phase. Rescaling the corresponding hydrodynamic  system (also referred to as Beris-Edward system) near the isotropic-nematic interface will lead to a limit system that is nematic and is governed by the Ericksen-Leslie system on one side and purely isotropic on the other side. However, a rigorous justification remains open. We believe that insights from the following analysis might also be helpful to solve this problem and other more complicated sharp interface limits in fluid mechanics.

It is the purpose of this contribution to establish a first rigorous convergence result with convergence rates in strong norms for the sharp interface limit $\eps\to 0$ in the case of a two-phase flow in fluid mechanics, which is comparable to results known for single phase field models like the Allen-Cahn, the Cahn-Hilliard,  or the phase field model equation, cf. De~Mottoni and Schatzman~\cite{DeMottoniSchatzman}, Alikakos et al.~\cite{AlikakosLimitCH}, Caginalp and Chen~\cite{CaginalpChen}, respectively.

More precisely,
on a bounded domain $\Omega\subseteq \R^2$
we consider the asymptotic limit $\eps\to 0$ of the following system:
\begin{alignat}{2}\label{eq:NSAC1}
  -\Delta \ve_\eps +\nabla p_\eps & = -\eps \Div (\nabla c_\eps \otimes \nabla c_\eps)&\quad & \text{in}\ \Omega\times(0,T_1),\\\label{eq:NSAC2}
  \Div \ve_\eps& = 0&\quad & \text{in}\ \Omega\times(0,T_1),\\\label{eq:NSAC3}
 \partial_t c_\eps +\ve_\eps\cdot \nabla c_\eps & =\Delta c_\eps - \frac1{\eps^2} f'(c_\eps)&\quad & \text{in}\ \Omega\times(0,T_1),\\
\ve_\eps|_{\partial\Omega} &= 0 \quad c_\eps|_{\partial\Omega} = -1&&\text{on }\partial\Omega\times (0,T_1), \\\label{eq:NSAC4}
 c_\eps |_{t=0}& = c_{0,\eps}&\quad& \text{in }\Omega
\end{alignat}
for a suitable double well potential $f$ and
for suitable ``well-prepared'' initial data $ c_{0,\eps}$ specified below.

Let us note that every sufficiently smooth solution of \eqref{eq:NSAC1}-\eqref{eq:NSAC4} satisfies the \emph{energy identity}
\begin{alignat}{1}\label{eq:Energy}
  E_\eps(c_\eps(t))+ \int_0^t \int_\Omega \left(|\nabla \ve_\eps|^2 +\frac1\eps|\mu_\eps|^2 \right)\, \sd x\, \sd \tau= E_\eps(c_{0,\eps})
\end{alignat}
for all $t\in (0,T_1)$, where $\mu_\eps= -\eps \Delta c_\eps + \frac1{\eps} f'(c_\eps)$ and
\begin{equation*}
  E_\eps(c_\eps(t))=\int_\Omega \eps \frac{|\nabla c_\eps(x,t)|^2}2\, \sd x+ \int_\Omega \frac{f(c_\eps(x,t))}\eps\, \sd x,
\end{equation*}
which provides some limited control as $\eps\to 0$.

The sharp interface limit of \eqref{eq:NSAC1}-\eqref{eq:NSAC4} is the system
\begin{alignat}{2}
\label{eq:Limit1}
  -\Delta \ve +\nabla p &= 0 &\qquad &\text{in }\Omega^\pm (t), t\in [0,T_0],\\\label{eq:Limit2}
    \Div \ve &= 0 &\qquad &\text{in }\Omega^\pm (t), t\in [0,T_0],\\\label{eq:Limit3}
    [2D\ve -p \tn{I}]\no_{\Gamma_t} &= -\sigma H_{\Gamma_t}\no_{\Gamma_t} && \text{on }\Gamma_t, t\in [0,T_0],\\
    [\ve] &=0 && \text{on }\Gamma_t, t\in [0,T_0],\\
    \ve|_{\partial\Omega} &=0&&\text{on }\partial\Omega\times [0,T_0],
\\\label{eq:Limit4}
    V_{\Gamma_t} -\no_{\Gamma_t}\cdot \ve|_{\Gamma_t} &= H_{\Gamma_t} && \text{on }\Gamma_t, t\in [0,T_0].
\end{alignat}
Here $\Omega$ is the disjoint union of $\Omega^+(t), \Omega^-(t)$, and $\Gamma_t$ for every $t\in[0,T_0]$, $\Omega^\pm(t)$ are smooth domains, $\Gamma_t=\partial\Omega^+(t)$,  $\no_{\Gamma_t}$ is the interior normal of $\Gamma_t$ with respect to  $\Omega^+(t)$. Moreover,
\begin{equation*}
  [u](p,t)= \lim_{h\to 0+} \left[u(p+\no_{\Gamma_t}(p)h)- u(p-\no_{\Gamma_t}(p)h)\right]
\end{equation*}
is the jump of a function $u\colon \Omega\times [0,T_0]\to \R^2$ at $\Gamma_t$ in direction of $\no_{\Gamma_t}$, $H_{\Gamma_t}$ and $V_{\Gamma_t}$ are the curvature and the normal velocity of $\Gamma_t$,  both with respect to $\no_{\Gamma_t}$. Furthermore, $D\ve=\frac 12(\nabla\ve+(\nabla\ve)^T)$ and $\sigma= \int_{\R}\theta_0'(\rho)^2\sd \rho$, where $\theta_0$ is the so-called optimal profile that is the unique solution of
\begin{alignat}{1}\label{eq:OptProfile1}
  -\theta_0''(\rho)+f'(\theta_0(\rho))&=0\qquad \text{for all }\rho\in\R,\\\label{eq:OptProfile2}
  \lim_{\rho\to\pm \infty}\theta_0(\rho)&=\pm 1,\qquad \theta_0(0)=0.
\end{alignat}
If the material derivative $\partial_t \ve_\eps+\ve_\eps \cdot \nabla \ve_\eps$ is added to the right-hand side of \eqref{eq:NSAC1} (i.e., the Navier-Stokes equations are considered), the system \eqref{eq:NSAC1}-\eqref{eq:NSAC4} was already suggested by Liu and Shen in \cite{LiuShenModelH} as an alternative approximation of a classical sharp interface model for a two-phase flow of viscous, incompressible, Newtonian fluids, which has advantages for numerical simulations since the Allen-Cahn equation is of second order and not of fourth order as the Cahn-Hilliard equation. On the other hand, for solutions of \eqref{eq:NSAC1}-\eqref{eq:NSAC4} the total mass $\int_\Omega c_\eps(x,t)\sd x$ is in general not preserved in time, in contrast to solutions of \eqref{eq:NSCH1}-\eqref{eq:NSCH4}, which is a disadvantage if the model is used to approximate a two-phase flow without phase transitions. However, \eqref{eq:NSAC1}-\eqref{eq:NSAC4} can be considered as a simplified model for a two-phase flow with phase transitions. Such models can yield systems of Navier-Stokes/Allen-Cahn type, cf. e.g. Blesgen~\cite{BlesgenModel}. Finally, let us mention that in \cite{LiuShenModelH} a rigorous result on the sharp interface limes of \eqref{eq:NSAC1}-\eqref{eq:NSAC4} was announced, which was not published so far to the best of the author's knowledge.

The limit system \eqref{eq:Limit1}-\eqref{eq:Limit4} was also studied by Liu, Sato, and Tonegawa in \cite{LiuSatoTonegawa2} if the Stokes equation on the right-hand side is replaced by a modified Navier-Stokes equation for a shear thickening non-Newtonian fluid of power-law type.
They constructed weak solutions for this system using a Galerkin approximation by a corresponding Navier-Stokes/Allen-Cahn system. In the proof they
pass to the limit  in the Galerkin approximation and pass to the limit $\eps\to 0$ simultaneously. Although the authors do not perform a separate sharp interface limit in the latter non-Newtonian Navier-Stokes/Allen-Cahn system, the result is close to it. The analysis depends heavily on the fact that a shear thickening fluid is used. Moreover, uniqueness of the limit is unknown and no convergence rates are given. Finally, we note that in \cite{TakasaoTonegawa} Takasao and Tonegawa study existence and regularity of global in time solutions of the mean curvature flow with a given convection term. To this end an approximation with a convective Allen-Cahn equation is used as well.

Throughout the paper we assume that
 $(\ve,p,\Gamma)$ is a smooth solution of the limit equation \eqref{eq:Limit1}-\eqref{eq:Limit4} for some $T_0>0$, where $(\Gamma_t)_{t\in[0,T_0]}$ is a family of smoothly evolving compact, non-self-intersecting, closed curves in $\Om$.
  More precisely, we assume that
 \begin{equation*}
   \Gamma:=\bigcup_{t\in [0,T_0]}\Gamma_t\times\{t\}
 \end{equation*}
 is a smooth two-dimensional submanifold of $\Tn\times \R$ (with boundary), and
 $\ve|_{\Omega^\pm}\in C^{\infty}(\ol{\Omega^\pm})^2$, $p|_{\Omega^\pm}\in C^{\infty}(\ol{\Omega^\pm})$, where
 \begin{equation*}
   \Omega^\pm = \bigcup_{t\in [0,T_0]} \Omega^\pm(t)\times \{t\}.
 \end{equation*}
In particular, we assume that $\Gamma_t \subseteq \Omega$ for every $t\in [0,T_0]$, which excludes contact angle problems.
Moreover, for $T_1\geq T_0$ let $(\ve_\eps,p_\eps, c_\eps)$ be the (classical) solution of \eqref{eq:NSAC1}-\eqref{eq:NSAC4} with smooth initial values $c_{0,\eps}\colon\Omega\to \R$, which will be specified in the main result below. Existence of classical solutions can be shown by standard methods. We could work with less regular initial data, but for simplicity of the presentation we assume smoothness.

We note that existence of local strong solutions of \eqref{eq:Limit1}-\eqref{eq:Limit4} can e.g.\ be obtained by adapting the strategy in \cite{StrongNSMS}, where a coupled Navier-Stokes/Mullins-Sekerka system was treated. This was carried out by Moser in \cite{Moser} in the case where the Stokes system \eqref{eq:Limit1}-\eqref{eq:Limit2} is replaced by the instationary Navier-Stokes system. By standard arguments from the regularity theory of parabolic equations and the Stokes system, one can prove that the solution is indeed smooth for smooth initial values.





%% file: SharpInterfaceLimitNSACOverview.tex

As in \cite{AlikakosLimitCH} and \cite{ChenHilhorstLogak} we assume in the following that $f\colon \R\to \R$ is smooth and satisfies the assumptions
\begin{equation*}
  f'(\pm 1)=0, \quad f''(\pm 1)>0,\quad f(s)=f(-s)>0 \quad \text{for all }s\in (-1,1).
\end{equation*}
We note that $f'$ in the present paper corresponds to $f$ in \cite{AlikakosLimitCH,ChenSpectrumAC} and to $-f$ in \cite{ChenHilhorstLogak}. Then there is a unique solution $\theta_0 \colon\mathbb{R}\to \R$ of \eqref{eq:OptProfile1}-\eqref{eq:OptProfile2},
which is monotone and will play a central role in the following.
Moreover, for every $m\in\N_0$, there is some $C_m>0$ such that
\begin{equation}\label{yuning:decayopti}
  |\p_\rho^m(\theta_0(\rho)\mp 1)|\leq C_m e^{-\alpha|\rho|}\quad\text{for all } \rho\in\R  \text{ with }\rho \gtrless 0,
\end{equation}
where $\alpha=\min(\sqrt{f''(-1)}, \sqrt{f''(1)})$. In the case $0<\alpha <\min(\sqrt{f''(-1)}, \sqrt{f''(1)})$ a detailed proof  can be found in \cite[Lemma~2.6.1]{PromotionStefan}.
One can choose e.g.\ $f(s)= \frac18(1-s^2)^2$. Then $\theta_0(s)=\tanh(\tfrac{s}2)$ for all $s\in\R$ and $\alpha=1$, cf.\ e.g.\ \cite{StabilityCHSimple}. For simplicity we will assume that $f$ is even. This implies that $\theta_0$ is odd and $\theta'_0$ is even.

For the statement of our main result, we need tubular neighborhoods of $\Gamma_t$.
For $\delta>0$ and $t\in[0,T_0]$ we defined
\begin{equation*}
  \Gamma_t(\delta):=\{y\in \Omega: \dist(y,\Gamma_t)<\delta\},\quad \Gamma(\delta)=\bigcup_{t\in [0,T_0]}\Gamma_t(\delta)\times \{t\}.
\end{equation*}
Moreover, we define the signed distance function
\begin{equation*}
  d_{\G}(x,t):=\sdist (\Gamma_t,x)=
  \begin{cases}
    \dist(\Omega^-(t),x) &\text{if } x\not \in \Omega^-(t)\\
    -\dist(\Omega^+(t),x) &\text{if } x \in \Omega^-(t)
  \end{cases}
\end{equation*}
for all $x\in\Omega, t\in[0,T_0]$.
 Since $\Gamma$ is smooth and compact, there is some $\delta>0$ sufficiently small,  such that $d_\Gamma\colon \Gamma(3\delta)\to \R$ is smooth, cf.\ Section~\ref{subsec:Coordinates} below.

Our main result is:
\begin{thm}\label{thm:main}
  Let $N=2$, $(\ve,\Gamma)$ be a smooth solution of \eqref{eq:Limit1}-\eqref{eq:Limit4} for some $T_0\in (0,\infty)$ and let
\begin{equation*}
  \begin{split}
    c_{A,0}^0(x)&=\zeta(d_{\Gamma_0}(x))\theta_0\left(\frac{d_{\Gamma_0}(x)}\eps\right)+(1-\zeta(d_{\Gamma_0}(x)))\left(
  \chi_{\Omega^+(0)}(x)-\chi_{\Omega^-(0)}(x)\right) \quad \text{for all }x\in \Om,
  \end{split}
\end{equation*}
where $d_{\Gamma_0}=d_\Gamma|_{t=0}$ is the signed distance function to $\Gamma_0$ and $\zeta\in C^\infty(\R)$ such that
\begin{equation}\label{yuning:1.34}
  \zeta(s)=1,~\text{if}~|s|\leq\delta; ~\zeta(s)=0,~\text{if}~|s|\geq 2\delta;~ 0\leq  -s\zeta'(s) \leq 4~\text{if}~ \delta\leq |s|\leq 2\delta.
\end{equation}
Moreover, let $c_{0,\eps}\colon \Omega\to\R$, $0<\eps\leq 1$, be smooth such that
\begin{equation*}
  \|c_{0,\eps}-c^0_{A,0}\|_{L^2(\Omega)}\leq C\eps^{\order+\frac12}\qquad \text{for all }\eps\in (0,1]
\end{equation*}
and some $C>0$, $\sup_{0<\eps\leq 1}\|c_{0,\eps}\|_{L^\infty(\Omega)}<\infty$ and $(\ve_\eps,c_\eps)$ be the corresponding solutions of \eqref{eq:NSAC1}-\eqref{eq:NSAC4}. Then there are some $\eps_0 \in (0,1]$, $R>0$, $T\in (0,T_0]$, and $c_A\colon \Om\times [0,T_0]\to \R$, $\ve_A\colon \Om\times [0,T_0]\to \R^2$ (depending on $\eps$) such that
\begin{subequations}\label{assumptions'}
  \begin{align}
 \sup_{0\leq t\leq T} \|c_\eps(t) -c_A(t)\|_{L^2(\Omega)}+\|\nabla (c_\eps -c_A)\|_{L^2(\Omega\times (0,T)\setminus\Gamma(\delta))}  &\leq R\eps^{\order+\frac12}\\
\|\nabla_\btau(c_\eps -c_A)\|_{L^2(\Omega\times(0,T)\cap \Gamma(2\delta))}+ \eps \|\partial_\no(c_\eps -c_A)\|_{L^2(\Omega\times(0,T)\cap \Gamma(2\delta))} &\leq R\eps^{\order+\frac12}
\end{align}
\end{subequations}
and for any $q\in [1,2)$
\begin{equation}
  \label{eq:convVelocityb}
  \|\ve_\eps -\ve_A\|_{L^2(0,T;L^q(\Om))} \leq C(q,R)\eps^2
\end{equation}
 hold true for all $\eps \in (0,\eps_0]$ and some $C(q,R)>0$. Moreover,
\begin{equation*}
  \lim_{\eps\to 0} c_A= \pm 1 \qquad \text{uniformly on compact subsets of } \Omega^\pm.
\end{equation*}
and
\begin{equation*}
  \ve_A= \ve + O(\eps) \qquad \text{in }L^\infty(\Om\times (0,T))\text{ as }\eps\to 0.
\end{equation*}
\end{thm}
More precise information on $c_A$ can be found in Section~\ref{eq:Approx}.
In particular, the result implies
\begin{equation*}
  c_\eps\to \pm 1 \qquad \text{in } L^2_{loc}(\Omega^\pm).
\end{equation*}
Here $\order=2$ is the basic convergence order (w.r.t. the $L^\infty(\Omega)$-norm).
Although this order is fixed, we will write $\order+\frac12$, $\order-\frac12$ etc.\ instead of $\frac52$, $\frac32$ etc. since in this way the relations between the different orders become more transparent.
\begin{rem}\label{rem:LInftyboundedness}
   It is easy to show that the solutions $(\ve_\eps,c_\eps)$ of \eqref{eq:NSAC1}-\eqref{eq:NSAC4} satisfy
  \begin{equation}\label{yuning:1.65}
    |c_\eps(x,t)|\leq  \max(\sup_{0<\eps\leq 1}\|c_{0,\eps}\|_{L^\infty(\Omega)},1) \qquad \text{for all }x\in\Om,t\in [0,T_0].
  \end{equation}
To this end one can e.g.~argue by contradiction and apply the same arguments as in the proof of the weak maximum principle.
\end{rem}

For the proof of the main result, we will follow the same basic strategy, which was already successfully used in \cite{DeMottoniSchatzman} for the Allen-Cahn equation, in \cite{AlikakosLimitCH} for the Cahn-Hilliard equation, and in \cite{ChenHilhorstLogak} for the mass-preserving Allen-Cahn equation. In many details we will follow the constructions in \cite{ChenHilhorstLogak}. Following this strategy the proof consists of two parts. In the first part a suitable approximate solution, which will be denoted by $(\ve_A,c_A)$ in the following, for \eqref{eq:NSAC1}-\eqref{eq:NSAC4} upto an error term of a certain  order in $\eps$ is constructed. To this end finitely many terms of an expansion in $\eps>0$, using the method of formally matched asymptotics, are used. In the second step the error of the approximate $(\ve_A, c_A)$ and the exact solutions $(\ve_\eps,c_\eps)$ are estimated with the aid of a suitable estimate for the linearized Allen-Cahn operator $\mathcal{L}_\eps$, defined by
\begin{equation}\label{eq:Leps}
  \mathcal{L}_\eps u = -\Delta u +\frac1{\eps^2} f''(c_A)u,\quad \text{for all }u\in H^2(\Omega).
\end{equation}
  However,  in order to adapt this strategy to the present system, several new difficulties, which are mainly related to the coupling of the Allen-Cahn and the Stokes system, have to be overcome.
More precisely, in order to estimate the difference $u:=c_\eps-c_A$, a suitable estimate of  the convection term $\ve_\eps\cdot \nabla c_\eps$ is required. To this end it will be essential how this term is approximated in the equation of $c_A$. More precisely, we will construct $c_A$ such that we have the following result:
\begin{thm}\label{thm:ApproxSolutions}
  Let the assumption of Theorem~\ref{thm:main} be satisfied and $R\geq 1$. Then for every $\eps\in (0,1)$ there are 
  \begin{equation*}
  \ve_A,\we_1,\we_2\colon \Om\times [0,T_0]\to \R^2,\quad c_A\colon \Om\times [0,T_0]\to \R,\quad r_A\colon \Om\times [0,T_0]\to \R
  \end{equation*}
 (depending on $\eps \in (0,1]$) such that $\ve_\eps = \ve_A+ \eps^2\we_1+\eps^2\we_2$ and
  \begin{alignat}{2}\label{eq:cA}
     \partial_t c_A+ (\ve_A +\eps^2\we_2) \cdot \nabla c_A + \eps^2\we_1|_{\Gamma}\cdot \nabla c_A&= \Delta c_A-\frac{f'(c_A)}{\eps^2} + r_{A} &\quad &\text{in } \Om\times [0,T_0]
  \end{alignat}
\Second{as well as $c_A|_{\partial\Omega}=-1$, $\ve_A|_{\partial\Omega}=0$.}
    Moreover, there are some $\eps_0>0$, $T_1>0$ and $M_R\colon (0,1]\times (0,T_0]\to (0,\infty)$, which is increasing with respect to both variables, such that $M_R(\eps,T)\to_{(\eps,T)\to 0} 0$ and, if
\begin{subequations}\label{assumptions}
  \begin{align}
 \sup_{0\leq t\leq T_\eps} \|c_\eps(t) -c_A(t)\|_{L^2(\Omega)} + \|\nabla (c_\eps -c_A)\|_{L^2(\Omega\times(0,T_\eps)\setminus \Gamma(\delta))} &\leq R\eps^{\order+\frac12},\\
 \|\nabla_\btau(c_\eps -c_A)\|_{L^2(\Omega\times(0,T_\eps)\cap \Gamma(2\delta))} +\eps\|\partial_\no(c_\eps -c_A)\|_{L^2(\Omega\times(0,T_\eps)\cap \Gamma(2\delta))} &\leq R\eps^{\order+\frac12}
\end{align}
\end{subequations}
hold true for some $T_\eps \in (0,T_0]$, $\eps_0\in (0,1]$, and all $\eps \in (0,\eps_0]$, then
  \begin{alignat}{1}\label{eq:rEstim}
\int_0^T \left|\int_{\Om} r_A(x,t)(c_\eps(x,t)-c_A(x,t))\sd x \right| \sd t&\leq M_R(\eps,T)\eps^{2(\order+\frac12)},\quad
  \end{alignat}
  for all $T\in (0,\min(T_\eps,T_1))$, $\eps \in (0,\eps_0]$. 
\end{thm}
Here
 $\we_1$ will be the leading part of the error $\we=\frac{\ve_\eps-\ve_A}{\eps^2}$ and $\we_1|_{\Gamma}(x,t)= \we_1(P_{\Gamma_t}(x),t)$ for $x\in \Gamma_t(2\delta)$, where $P_{\Gamma_t}$ denotes the orthogonal projection onto $\Gamma_t$, cf.~Section~\ref{subsec:Coordinates} below. The tangential gradient $\nabla_\btau$ and $\partial_\no$ will be defined precisely in the same subsection. 
In \eqref{eq:cA} $\eps^2\we_2 \cdot \nabla c_A$ can also be omitted since it is of the same order as  $r_A$. But the presence of the term $\eps^2\we_1|_{\Gamma}\cdot \nabla c_A$ is essential for the error estimates, cf. Section~\ref{sec:DifferenceEstimate} and Lemma~\ref{eq:ConvTermEstim} below for the details.

We will prove Theorem~\ref{thm:main} with the aid of the latter theorem by considering the equation for $u=c_\eps-c_A$ and using suitable estimates for $\mathcal{L}_\eps$, which refine the results of \cite{ChenSpectrumAC} in tangential directions, as well as careful estimates of all remainder terms. In this proof a continuation argument is used to show that \eqref{assumptions} is valid for $T_\eps\geq T_1>0$ if $T_1$ is sufficiently small.

\smallskip

\noindent {\bf Remark:}  Let us comment on the orders of $\eps$ on the right-hand side of \eqref{assumptions}. Because of \eqref{yuning:decayopti} and a simple change of variables, we have
  \begin{equation*}
    \|\theta_0(\tfrac{\cdot}\eps)-(\chi_{[0,\infty)}-\chi_{(-\infty,0)})\|_{L^2(\R)} = M\eps^{\frac12}.
  \end{equation*}
  Hence the power $\eps^{\frac12}$ appears naturally, when estimating differences in $L^2$ in normal direction. Moreover, applying the one-dimensional interpolation inequality
  \begin{equation*}
    \|u\|_{L^\infty(-2\delta,2\delta)}\leq C_\delta\|u\|_{L^2(-2\delta,2\delta)}^{\frac12}\|u\|_{H^1(-2\delta,2\delta)}^{\frac12}\quad \text{for all }u\in H^1(-2\delta,2\delta),
  \end{equation*}
  to
\eqref{assumptions} yield a control of $u$ of the order $\eps^{\order}$ in the $L^\infty$-norm in normal direction (and $L^2$-norms in the other directions).

As mentioned above the form of the convection terms in \eqref{eq:cA} will be essential for the remainder estimates.
In order to obtain this we will construct $c_A$ such that
\begin{equation*}
  c_A(x,t)= \underbrace{\theta_0 \left(\frac{d_\G(x,t)}\eps - h_\eps(S(x,t),t)\right)}_{=c_{A,0}} + O(\eps^2),\quad \text{where }h_\eps = h_1+\eps h_{2,\eps},
\end{equation*}
in  $\Gamma(\delta)$ for some suitable $h_1,h_{2,\eps}\colon \T^1\times [0,T_0]\to \R$, cf. Section~\ref{eq:Approx} for the details.. Here $S$ is defined in Section~\ref{subsec:Coordinates} below. For every $t\in [0,T_0]$, $S(\cdot,t)|_{\Gamma_t}\colon \Gamma_t\to \T^1$ is the pull-back of a suitable parametrization of $\Gamma_t$.
In order to include the term $\we_1|_{\Gamma}\cdot \nabla c_A$ in \eqref{eq:cA}, we will use a suitable choice of $h_{2,\eps}$, which will depend on $\we_1$, cf.~\eqref{yuning:h2equ'} below. Here $\we_1$ depends on $\eps$ and we have only certain norms of $\we_1$ under control as $\eps\to 0$. This is in sharp contrast to the term $h_1$, which is defined by an equation independent of $\eps$, cf.\ \eqref{yuning:h1equ'} below. The latter equation basically depends only on the (smooth) limit solution $(\Gamma,\ve,p)$ as it is the case in the expansions in \cite{DeMottoniSchatzman,AlikakosLimitCH,ChenHilhorstLogak}. The careful treatment of $h_{2,\eps}$ (as well as a similar higher order term $\hat{c}_{3,\eps}$ in the expansion of $c_A$) is one of the essential novelties in this contribution compared to \cite{DeMottoniSchatzman,AlikakosLimitCH,ChenHilhorstLogak} and a key to the proof of our main result. Furthermore,   treatment of the error $\ve_\eps-\ve_A$ using careful estimates for (very) weak solutions of the Stokes system and the errors in the capillary term on the right-hand side of \eqref{eq:NSAC1} are another essential ingredient. Finally let us note that  another novelty of the present contribution in comparison to \cite{DeMottoniSchatzman,AlikakosLimitCH,ChenHilhorstLogak} is that, in the expansion of $c_A$ we will only use three terms of orders $O(1)$, $O(\eps^2)$ and $O(\eps^3)$, respectively, which reduces the number of levels in the asymptotic expansion significantly.

The structure of this article is as follows. In Section~\ref{sec:Prelim} we will discuss several preliminary results concerning suitable coordinates close to $\Gamma$, evolution equations for $h_1,h_{2,\eps}$, results on ODEs needed in the asymptotic expansions, and some kind of spectral estimate for $\mathcal{L}_\eps$ uniformly in $\eps\in (0,\eps_0]$.
Then $\ve_A$ is constructed in Section~\ref{sec:ApproxVelocity} assuming that the leading part of $c_{A,0}$ of $c_A$ (which depends on $h_1,h_{2,\eps}$) is known. Moreover, $\ve_A\cdot \nabla c_{A,0}$ is expanded and  using the knowledge of the latter expansion, $c_A$ is constructed in Section~\ref{eq:Approx}  and Theorem~\ref{thm:ApproxSolutions} is proved. Finally, the main result is proven in Section~\ref{sec:DifferenceEstimate}. Some lengthy but straight forward calculations related to the matched asymptotic expansions are given in details in the Appendix.

\smallskip

\noindent
{\bf Acknowledgements:} This work was supported by the SPP 1506 "Transport Processes
at Fluidic Interfaces" of the German Science Foundation (DFG) through the grant AB285/4-2.  Y. Liu was supported by Chenguang Plan 15CG66 of Shanghai
Municipal Education Commission \&\ Shanghai Educational Development Foundation. Both supports are  gratefully acknowledged. Moreover, the authors are greatful to Dr.\ Daniel Lengeler for several fruitful discussions, which contributed to this work, and help with the proof of Theorem~\ref{thm:Spectral}.  Finally, we thank Andreas Sch\"ottl and Maximilian Moser for their careful proofreading.




%% file: SharpInterfaceLimitNSACPrelim.tex
\section{Preliminaries}\label{sec:Prelim}
\subsection{Coordinates}\label{subsec:Coordinates}

We will parametrize $(\Gamma_t)_{t\in[0,T_0]}$ with the aid of a family of smooth diffeomorphisms $X_0\colon \T^1\times [0,T_0]\to \Om$ such that $\partial_s X_0(s,t)\neq 0$ for all $s\in\T^1$, $t\in [0,T_0]$. Moreover, let
 \begin{equation*}
\btau(s,t)= \frac{\partial_s X_0(s,t)}{|\partial_s X_0(s,t)|}\quad \text{and}\quad \no(s,t)=
\begin{pmatrix}
  0 & -1\\
1 & 0
\end{pmatrix}
\btau(s,t)\qquad \text{for all }(s,t)\in \T^1\times [0,T_0]
 \end{equation*}
be the normalized tangent and normal vectors on $\Gamma_t$ at $X_0(s,t)$.

We choose the orientation of $\Gamma_t$ (induced by $X_0(\cdot,t)$) such that $\no(s,t)$ is the exterior normal with respect to $\Omega^-(t)$. Moreover, we denote
\begin{equation}\label{yuning:1.58}
\no_{\Gamma_t}(x):= \no (s,t)~\text{for all}~ x=X_0(s,t)\in \Gamma_t.
\end{equation}
  Furthermore, $V_{\Gamma_t}$ and $H_{\Gamma_t}$ should be the normal velocity and (mean) curvature of $\Gamma_t$ (with respect to $\no_{\Gamma_t}$) and we define
 \begin{equation}\label{yuning:1.57}
   V(s,t)= V_{\Gamma_t}(X_0(s,t)),\quad H(s,t)= H_{\Gamma_t}(X_0(s,t))\quad \text{for all }s\in\T^1, t\in [0,T_0].
 \end{equation}
Hence $H_{\Gamma_t}\leq 0$ if $\Omega^-(t)$ is convex.
Moreover, by definition,
\begin{equation*}
  V_{\Gamma_t}(X_0(s,t))=V(s,t)= \partial_t X_0(s,t)\cdot \no(s,t)\qquad \text{for all }(s,t)\in \T^1\times [0,T_0].
\end{equation*}

In the following we will need a tubular neighborhood of $\Gamma_t$: For $\delta>0$ sufficiently small,  the orthogonal projection $P_{\Gamma_t}(x)$ of all
\begin{equation*}
x\in \Gamma_t(3\delta) =\{y\in \Omega: \dist(y,\Gamma_t)<3\delta\}
\end{equation*}
is well-defined and smooth. Moreover, we choose $\delta$ so small that $\dist(\partial\Omega,\Gamma_t)>3\delta$ for every $t\in [0,T_0]$. Every $x\in\Gamma_t(3\delta)$ has a unique representation
\begin{equation*}
x=P_{\Gamma_t}(x)+r\no_{\Gamma_t}(P_{\Gamma_t}(x))
\end{equation*}
 where $r=\sdist(\Gamma_t,x)$. Here
\begin{equation*}
  d_{\G}(x,t):=\sdist (\Gamma_t,x)=
  \begin{cases}
    \dist(\Omega^-(t),x) &\text{if } x\not \in \Omega^-(t),\\
    -\dist(\Omega^+(t),x) &\text{if } x \in \Omega^-(t).
  \end{cases}
\end{equation*}
For the following we define for $\delta'\in (0,3\delta]$
\begin{equation*}
  \Gamma(\delta') =\bigcup_{t\in [0,T_0]} \Gamma_t(\delta') \times\{t\}.
\end{equation*}
 Throughout this contribution we will often use
\begin{equation*}
  \int_{\Gamma_t(\delta')} f(x)\,\sd x = \int_{-\delta'}^{\delta'}\int_{\Gamma_t} f(p+r\no_{\Gamma_t}(p))J(r,p,t)\sd \sigma(p)\sd r
\end{equation*}
for any $\delta'\in (0,3\delta]$, where $J\colon (-3\delta,3\delta)\times \Gamma \to (0,\infty)$ is a smooth function depending on $\Gamma$.

We introduce new coordinates in  $\Gamma(3\delta)$ which we denote by
\begin{equation*}
  X\colon  (-3\delta, 3\delta)\times \T^1 \times [0,T_0]\mapsto \Gamma(3\delta)~\text{by}~  X(r,s,t):= X_0(s,t)+r\no(s,t),
\end{equation*}
where
\begin{equation}\label{yuning:1.42}
  r=\sdist(\Gamma_t,x), \qquad s= X_{0}^{-1}(P_{\Gamma_t}(x),t)=: S(x,t).
\end{equation}
Differentiating the identity
\begin{equation*}
  d_\Gamma(X_0(s,t)+r\no(s,t),t)=r,
\end{equation*}
one obtains
  \begin{equation}\label{yuning:1.26}
    \nabla d_{\G}(x,t)=\no_{\Gamma_t} (P_{\Gamma_t}(x)),~  \partial_t d_{\G}(x,t)=-V_{\Gamma_t} (P_{\Gamma_t}(x)),~\Delta d_\Gamma(q,t)=-H_{\Gamma_t}(q)
\end{equation}
for all $(x,t)\in \Gamma(3\delta)$, $(q,t)\in\Gamma$, resp., cf. Chen et al.~\cite[Section~4.1]{ChenHilhorstLogak}.
 Moreover, we define
\begin{equation}\label{yuning:1.50}
\partial_{\btau} u(x,t):= \btau(S(x,t),t)u(x,t),\quad   \nabla_\btau u(x,t):= \partial_{\btau} u(x,t)\btau(S(x,t),t)\quad
 \end{equation}
for all $(x,t)\in \Gamma(3\delta)$. 

In the following we associate a function $\phi(x,t)$ to $\tilde{\phi}(r,s,t)$ such that
 \begin{equation*}
   \phi(x,t)=\tilde{\phi}(d_{\G}(x,t),S(x,t),t)\quad\text{or}\quad\phi(X_0(s,t)+r\no(s,t),t)=\tilde{\phi}(r,s,t).
 \end{equation*}
 Then it follows from chain rule together with \eqref{yuning:1.26} that
\begin{equation}\label{Prelim:1.13}
  \begin{split}
    \partial_t \phi(x,t) &= -V_{\Gamma_t} (P_{\Gamma_t}(x)) \partial_r\tilde{\phi}(r,s,t) + \partial_{t}^\Gamma \tilde{\phi}(r,s,t) \\
  \nabla \phi(x,t) &= \no_{\Gamma_t} (P_{\Gamma_t}(x)) \partial_r\tilde{\phi}(r,s,t) + \nabla^ \Gamma  \tilde{\phi}(r,s,t) \\
 \Delta \phi(x,t) &= \partial_r^2\tilde{\phi}(r,s,t) + \Delta d_{\G_t}(x) \partial_r\tilde{\phi}(r,s,t) +  \Delta^{\Gamma} \tilde{\phi}(r,s,t),
  \end{split}
\end{equation}
where $r,s$ are as in \eqref{yuning:1.42} and we used the notation
\begin{equation}\label{Prelim:1.12}
  \begin{split}
    \partial_{t}^\Gamma \tilde{\phi}(r,s,t) &= \partial_t \tilde{\phi}(r,s,t) + \partial_t S(x,t)\partial_s \tilde{\phi}(r,s,t),\\
\nabla^{\Gamma} \tilde{\phi}(r,s,t) &= (\nabla S)(x,t)\partial_s \tilde{\phi}(r,s,t),\\
\Delta^{\Gamma} \tilde{\phi}(r,s,t) &= |(\nabla S)(x,t)|^2\partial_s^2 \tilde{\phi}(r,s,t)+(\Delta S)(x,t)\partial_s \tilde{\phi}(r,s,t),
  \end{split}
\end{equation}
 cf. \cite[Section~4.1]{ChenHilhorstLogak} for more details.
In \eqref{Prelim:1.12} $x$ is understood via $x=\no(s,t)r+X_0(s,t)$.
It can be seen from \eqref{Prelim:1.12} that, if $g\colon \T^1\times [0,T_0]\to \R$ depends only on $(s,t)$, then $\nabla^\G g$ will be a function of $(r,s,t)$ instead of merely depending on the surface coordinate $s$ and $t$:
\begin{equation}\label{yuning:1.46}
  \nabla^\G g(r,s,t)=(\nabla S)(x,t)\partial_s g(s,t), \qquad \text{where }x=X(r,s,t).
\end{equation}
We note that, since $\Gamma$ is given and smooth, we have
$
   |\nabla^\G g|\leq C|\p_s g|.
$
 Moreover, differentiating
  \begin{equation*}
    s= S(X_0(x,t)+ r\no(s,t),t)\qquad \text{for all }s\in \T^1, r\in(-3\delta,3\delta)
  \end{equation*}
  with respect to $r$
  we obtain $\nabla S(x,t)\cdot \no(s,t)=0$ for all $(x,t)\in \Gamma(3\delta)$, $s=S(x,t)$.
Motivated by this, we  define for every $h\colon \T^1\times [0,T_0]\to \R$
\begin{equation}\label{yuning:1.27}
\begin{split}
  (\nabla_\G h)(s,t):=(\nabla^\Gamma h)(0,s,t),\\
  (\Delta_\G h)(s,t):=(\Delta^\G h)(0,s,t),\\
  (D_t h)(s,t):=(\p_t^\G h)(0,s,t),
\end{split}
\end{equation}
which are the restrictions to $r=0$ of the operators above,
and we define the differences
\begin{equation}\label{Prelim:1.11}
\begin{split}
  (\Lgrad h)(r,s,t):=(\nabla^\Gamma h)(r,s,t)-(\nabla_\Gamma h)(s,t),\\
   (\Ldelta h)(r,s,t):=(\Delta^\Gamma h)(r,s,t)-(\Delta_\Gamma h)(s,t),\\
  (\Lt h)(r,s,t):=(\p_t^\G h)(r,s,t)-(D_t h)(s,t).
 \end{split}
\end{equation}
So the coefficients of the latter operators vanish for $r=0$, which corresponds to $x\in\Gamma_t$.

Throughout this contribution we will frequently use that, if $a\colon \Gamma(3\delta)\to \R$ is smooth in normal direction and vanishes on $\Gamma$, then $\tilde{a}\colon \Gamma(3\delta)\to \R$ with
  \begin{equation*}
    \tilde{a}(x,t)=
    \begin{cases}
      \frac{a(x,t)}{d_\Gamma(x,t)} &\text{if }(x,t)\in \Gamma(3\delta)\setminus \Gamma,\\
      \partial_{\no} a(x,t) &\text{if }(x,t)\in \Gamma
    \end{cases}
  \end{equation*}
  is smooth in normal direction as well. Moreover, regularity in tangential directions is preserved. In particular $\tilde{a}$ is smooth if $a$ is smooth. These statements can be easily proved with the aid of a Taylor expansion with respect to $d_\Gamma$.

Finally, we denote
\begin{alignat*}{1}
 (X_0^\ast u)(s,t)&:= u(X_0(s,t),t) \qquad \text{for all }s\in\T^1,t\in[0,T_0],\\
  (X_0^{\ast,-1} v)(p,t) & := v(X_0^{-1}(p,t),t) \qquad \text{for all }(p,t)\in\Gamma
\end{alignat*}
 if $u\colon \Gamma\to \R^N$  and $v\colon \T^{1}\times [0,T_0]\to \R^N$  for some $N\in\N$.

\subsection{Tangential Differential Operators}
In view of \eqref{yuning:1.50}, we can express the tangential derivative in the following way
\begin{equation}\label{yuning:1.49}
  \nabla_\btau=(I-\no(S(\cdot),\cdot)\otimes\no(S(\cdot),\cdot))\nabla.
\end{equation}
 Then we have 
 \begin{alignat}{1}\nonumber
      [\partial_\no, \nabla_\btau] g 
  &:=\p_\no(  (I-\no\otimes \no) \nabla g)  - (I-\no\otimes \no) \nabla(\p_\no g)\\\nonumber
  &=  (I-\no\otimes \no) \p_\no\nabla g   - (I-\no\otimes \no) \nabla(\no\cdot\nabla g)\\\label{yuning:1.22}
   &=\sum_{j=1}^2((I-\no\otimes \no)\nabla \no_j)\partial_{x_j} g = \btau (\partial_{\btau} \no\cdot \nabla g).
  \end{alignat}
 This shows that the commutator  $[\p_\no,\nabla_\btau]$ is a tangential differential operator.
  For the sake of integrating by parts for functions defined near the interface, we need the  formula
  \begin{equation*}
 \begin{split}
    &\Div (I-\no(S(\cdot),\cdot)\otimes\no(S(\cdot),\cdot))\\
    =&- \underbrace{\Div (\no(S(\cdot),\cdot))}_{=:-\kappa}\no(S(\cdot),\cdot)- \no(S(\cdot))\cdot \nabla \no(S(\cdot),\cdot)= \kappa \no(S(\cdot),\cdot).
 \end{split}
\end{equation*}
This together with  Gau\ss' Theorem implies the following lemma:
\begin{lem}
Let $t\in [0,T_0]$.
For any  $u\in H^1_0(\Gamma_t(2\delta))$, $\ve\in H^1_0(\Gamma_t(2\delta))^2$ we have
   \begin{equation}\label{yuning:1.23}
   \int_{\Gamma_t(2\delta)} u \Div_\btau  \ve\sd x= - \int_{\Gamma_t(2\delta)}  \nabla_\btau u\cdot \ve\sd x - \int_{\Gamma_t(2\delta)} \kappa \no\cdot \ve u\sd x
 \end{equation}
 where $\nabla_\btau$ is defined by \eqref{yuning:1.49} and
  $\kappa=-\Div \no(S(\cdot),\cdot)$.
\end{lem}

\subsection{The Stretched Variable}
In the sequel,  for given functions $h_1,h_2\colon \T^1\times [0,T_0]\to \R$, we shall define the stretched variable $\rho$ by
\begin{equation}\label{yuningliu:stretched}
  {\rho(x,t)=\tfrac{d_\G(x,t)}{\eps}-h_1(S(x,t),t)-\eps h_2(S(x,t),t)}.
\end{equation}
Moreover, we denote $h_\eps:=h_1+\eps h_2$.
As in \cite[Section~4.2]{ChenHilhorstLogak} we consider the Taylor expansion of $\Delta d_\Gamma$ in the normal direction and obtain
    \begin{equation}\label{yuning:ex1}
      \begin{split}
        \Delta d_{\Gamma}(x,t)&=-H_{\Gamma_t}(s)-\eps(\rho+h_\eps(s,t))\kappa_1(s,t)\\
        &+ \eps^2 \kappa_2(s,t)(\rho+h_\eps(s,t))^2+\eps^3 \kappa_{3,\eps}(\rho,s,t),
      \end{split}
    \end{equation}
where $s$ is understood via \eqref{yuning:1.42} and $$\kappa_1(s,t)=-\nabla d_{\Gamma}(X_0(s,t))\cdot\nabla\Delta d_{\Gamma}(X_0(s,t))=H(s,t)^2,$$ $\kappa_2(s,t)$ is smooth
and $\kappa_{3,\eps}$ is a smooth function satisfying
\begin{equation}\label{yuning:1.44}
  |\kappa_{3,\eps}(\rho,s,t)|\leq C|\rho+h_\eps(s,t)|^3\qquad \text{for all }\rho\in \R, s\in \mathbb{T}^1, t\in [0,T_0], \eps \in  (0,1).
\end{equation}
The following lemma is due to the chain rule and \eqref{Prelim:1.13}, cf. \cite[Section~4.2]{ChenHilhorstLogak}:
\begin{lem}
Let ${\hat{w}}\colon \R\times \T^1\times [0,T_0]\to \R$ be sufficiently smooth and let
\begin{equation*}
w(x,t)={\hat{w}}\(\tfrac{d_\G(x,t)}\eps-h_\eps(S(x,t),t),S(x,t),t\) \quad \text{for all }(x,t)\in \Gamma(2\delta).
\end{equation*}
Then for each $\eps>0$
  \begin{equation}\label{yuningliu:formula1}
  \begin{split}
    \p_t w(x,t)=&-\(\tfrac{V_{\Gamma_t} (P_{\Gamma_t}(x))}\eps +\p_t^\Gamma h_\eps(r,s,t)\)\p_\rho {\hat{w}}(\rho,s,t) +\p_t^\Gamma {\hat{w}}(r,\rho,s,t)\\
    \nabla w(x,t)=& \(\tfrac{\no_{\Gamma_t} (P_{\Gamma_t}(x))}  \eps -\nabla^\Gamma h_\eps(r,s,t)\)\p_\rho{\hat{w}}(\rho,s,t) +\nabla^\Gamma{\hat{w}}(r,\rho,s,t)\\
    \Delta w(x,t)=& (\eps^{-2}+|\nabla^\Gamma h_\eps(r,s,t)|^2) \p^2_\rho{\hat{w}}(\rho,s,t)\\
    &+\(\eps^{-1}\Delta d_\Gamma (x,t) -\Delta^\Gamma h_\eps(r,s,t)\)\p_\rho{\hat{w}}(\rho,s,t)\\
    &-2\nabla^\Gamma h_\eps(r,s,t)\cdot\nabla^\Gamma \p_\rho{\hat{w}}(r,\rho,s,t)+\Delta^\Gamma {\hat{w}}(r,\rho,s,t),
  \end{split}
\end{equation}
where $\rho$ is as in \eqref{yuningliu:stretched} and $(r,s)$ is understood via \eqref{yuning:1.42}.
\end{lem}

\subsection{A Result from ODE-Theory}

The following result on solvability of the linearized equation for the optimal profile $\theta_0$ will be essential for the asymptotic expansion. The proof can be found in \cite[Lemma 3]{ChenHilhorstLogak}.
\begin{prop}\label{yuning:solveode}
  Assume that $g\colon \R\times\T^1\times [0,T_0]\rightarrow\R$ and $g^\pm\colon \T^1\times [0,T_0]\rightarrow\R $ are smooth 
 and for some $i\in\N_0$, $a>0$ satisfying
  \begin{equation*}
    \sup_{(s,t)\in \T^1\times [0,T_0]}\left|\p_\rho^k\partial_s^l \p_t^m[g(\rho,s,t)-g^\pm(s,t)]\right|\leq C_{k,l,m}(1+|\rho|)^ie^{-a|\rho|}\quad \text{for all } \rho\gtrless 0.
  \end{equation*}
  for all $k,l,m\in \mathbb{N}_0$ 
  and some $C_{k,l,m}>0$. Then for given $(s,t)\in \T^1\times [0,T_0]$ the ODE
  \begin{equation*}
    -\p^2_\rho u+f''(\theta_0)u=g(\cdot ,s,t)~\text{in}~\mathbb{R},\quad u(0,s,t)=0
  \end{equation*}
  has a unique bounded solution $u(\cdot,s,t)$ if and only if
  \begin{equation}\label{lem:CompCondODE}
    ~\int_\mathbb{R}g(\rho,s,t)\theta'_0(\rho) \sd \rho=0. 
  \end{equation}
  If the solution exists for all $(s,t)\in \T^1\times [0,T_0]$, then for all $(k,l,m)\in \mathbb{N}^3$ 
  there is some $C_{k,l,m}$ such that
   \begin{equation}\label{yuning:matching1}
    \sup_{(s,t)\in \T^1\times [0,T_0]}\left|\p_\rho^k\partial_s^l \p_t^m\left(u(\rho,s,t)-\frac{g^\pm(s,t)}{f''(\pm 1)}\right)\right|\leq C_{k,l,m}(1+|\rho|)^ie^{-a|\rho|}\quad\text{for all }\rho\gtrless 0.
  \end{equation}
\end{prop}

\subsection{Remainder Terms}

In the sequel we shall use for fixed $t\in[0,T_0]$ and $1\leq p< \infty$
\begin{alignat*}{1}
  L^{p,\infty}(\Gamma_t(2\delta))&:=\left\{ f\colon \Gamma_t(2\delta)\to \R \text{ measurable }:\|f\|_{L^{p,\infty}(\Gamma_t(2\delta))}<\infty \right\},\quad \text{where }\\
  \|f\|_{L^{p,\infty}(\Gamma_t(2\delta))}&:=\left(\int_{\T^1}\operatorname{ess\, sup}_{|r|\leq 2\delta } |f(X_0(s,t)+r\no(s,t))|^p \sd s\right)^{\frac1p}.
\end{alignat*}
Moreover, the standard $L^p$-Sobolev space of order $m\in\N_0$ on an open set $U\subseteq\R^N$ will be denoted by $W^m_p(U)$ and $L^p(U)$ is the usual Lebesgue space with respect to the Lebesgue measure. Furthermore, $H^s(U)$ denotes the $L^2$-Sobolev space of order $s\in\R$ and $H^s_0(U)$ is the closure of $C_0^\infty(U)$ in $H^s(U)$. The $X$-valued variants are denoted by $W^m_p(U;X)$, $L^p(U;X)$, and $H^s(U;X)$, respectively.
We note that
\begin{equation} \label{yuning:2.01}
  H^1(\Gamma_t(2\delta))\hookrightarrow L^{4,\infty}(\Gamma_t(2\delta)),
\end{equation}
which follows from the interpolation inequality
\begin{equation*}
  \|f\|_{L^\infty(-2\delta,2\delta)}\leq C\|f\|_{L^2(-2\delta,2\delta)}^{\frac12}\|f\|_{H^1(-2\delta,2\delta)}^{\frac12}\quad \text{for }f\in H^1(-2\delta,2\delta).
\end{equation*}

The following lemma provides a first useful general estimate of typical remainder terms.
\begin{lem}\label{yuning:1.21}
  Let  $g\colon\T^1\times [0,T_0]\to \R$ be continuous. There is some $C>0$, independent of $g$, such that
  \begin{equation*}
    \left\|\eta(\tfrac {d_{\G}(\cdot,t)}\eps-g(S(\cdot,t),t))u\psi \right\|_{L^1( \G_t(2\delta))}\leq C \eps^{\frac 12}\|\eta\|_{L^2(\R)} \|u\|_{L^2(\G_t(2\delta))}\|\psi\|_{L^{2,\infty}(\G_t(2\delta))}
  \end{equation*}
  holds for any $u \in L^2(\G_t(2\delta)), \psi\in L^{2,\infty}(\G_t(2\delta))$, continuous $\eta\in L^2(\R)$, and $\eps\in (0,1)$.
\end{lem}
\begin{proof}
For fixed $t\in [0,T_0]$ we perform the change of variables $(x,t)\mapsto (s,r,t)$ and get
  \begin{equation*}
    \begin{split}
      &\int_{\G_t(2\delta)}\left|\eta(\tfrac {d_{\G}(x,t)}\eps-g(S(x,t),t))u(x)\psi(x) \right| \sd x
      \\
      &=\int_{\T^1}\int_{-2\delta}^{2\delta}\left|\eta(\tfrac r\eps-g(s,t)) u(X(r,s,t))\psi(X(r,s,t))\right|J_t(r,s) \sd r\sd s.
    \end{split}
  \end{equation*}
where $J_t(r,s)$ denotes the square root of the Gram determinant. Since $|J_t(r,s)|\leq C$,
    \begin{equation*}
    \begin{split}
      &  \left\|\eta(\tfrac {d_{\G}(\cdot,t)}\eps-g(S(\cdot,t),t))u\psi \right\|_{L^1( \G_t(2\delta))}\\
      &\leq C\int_{\T^1}\|\eta(\tfrac\cdot\eps-g(s,t))\|_{L^2(\R)}
       \sup_{|r|\leq 2\delta}|\psi(X(r,s,t)) |\sqrt{\int_{|r|\leq 2\delta}|u(X(r,s,t))|^2J_t(r,s)\sd r}\sd s\\
      &\leq C\eps^{\frac 12}\|\eta\|_{L^2(\R)}       \sqrt{\int_{\T^1}
       \sup_{|r|\leq 2\delta}|\psi(X(r,s,t)) |^2\sd s}\sqrt{\int_{\T^1}\int_{|r|\leq2\delta}|u(X(r,s,t))|^2J_t(r,s)\sd r\sd s},
    \end{split}
  \end{equation*}
which implies the statement of the lemma.
\end{proof}

For a systematic treatment of the remainder terms, we introduce:
\begin{defn}\label{yuning:1.15}
  For any $k\in \R$ and $\alpha>0$,  $\mathcal{R}_{k,\alpha}$ denotes the vector space of all families of continuous functions $\tr_\eps\colon \R\times \Gamma(2\delta) \to \R$, $\eps\in (0,1)$, which are continuously differentiable with respect to $\no_{\Gamma_t}$ for all $t\in [0,T_0]$ such that
  \begin{equation}\label{eq:EstimRkalpha}
    |\partial_{\no_{\Gamma_t}}^j \tr_\eps(\rho,x,t)|\leq Ce^{-\alpha |\rho|}\eps^k\qquad \text{for all }\rho\in \R,(x,t)\in\Gamma(2\delta), j=0,1, \eps \in (0,1)
  \end{equation}
for some $C>0$ independent of $\rho\in \R,(x,t)\in\Gamma(2\delta)$, $\eps\in (0,1)$. Moreover,   $\mathcal{R}_{k,\alpha}$ is equipped  with the norm
\begin{equation*}
  \|(\tr_\eps)_{\eps\in (0,1)}\|_{\mathcal{R}_{k,\alpha}}= \sup_{\eps \in (0,1), (x,t)\in\Gamma(2\delta),\rho \in\R,j=0,1} |\partial_{\no_{\Gamma_t}}^j \tr_\eps(\rho,x,t)|e^{\alpha |\rho|}\eps^{-k}.
\end{equation*}
Finally, $\mathcal{R}_{k,\alpha}^0$ is the subspace of all $(\tr_\eps)_{\eps\in (0,1)}\in \mathcal{R}_{k,\alpha}$ such that
\begin{equation}\label{eq:RemainderVanish}
  \tr_\eps(\rho,x,t)= 0 \qquad \text{for all }\rho \in\R, x\in\Gamma_t, t\in [0,T_0].
\end{equation}
\end{defn}
\begin{lem}\label{lem:Remainder}
  Let $k\in\R$, $\alpha>0$, $1\leq p \leq \infty$, $h_\eps \colon \T^1\times [0,T_0] \to \R$ such that
  \begin{equation*}
   M:= \sup_{{0<\eps <\eps_0},(s,t)\in\T^1\times [0,T_\eps]} |h_\eps(s,t)|<\infty
  \end{equation*}
  for some $T_\eps\in (0,T_0]$, {$\eps_0\in (0,1)$},
   $(\tr_\eps)_{\eps\in (0,1)}\in \mathcal{R}_{k,\alpha}$ and
  \begin{equation*}
    r_\eps (x,t):= \tr_\eps\left( \frac{d_{\G}(x,t)}\eps - h_\eps(S(x,t),t), x,t\right)\qquad \text{for all }(x,t)\in\Gamma(2\delta).
  \end{equation*}
 Then there is a constant $C>0$,  independent of ${M}, T_\eps,0<\eps\leq \eps_0,$ and $\eps_0\in (0,1)$, such that
  \begin{equation}\label{eq:RemainderEstim1}
    \| \sup_{(x,t)\in\Gamma(2\delta)}|\tr_\eps(\tfrac{\cdot}\eps,x,t)|\|_{L^p(\R)} \leq C\eps^{k+\frac1{p}}\qquad {\text{for all }0<\eps<\Third{\eps_0}}.
  \end{equation}
  Moreover, if even $(\tr_\eps)_{\eps\in (0,1)}\in \mathcal{R}_{k,\alpha}^0$, then there is a constant $C>0$, independent of ${M}, T_\eps,0<\eps\leq \eps_0,$ and $\eps_0\in (0,1)$, such that
  \begin{equation}\label{eq:RemainderEstim2}
{    \sup_{0\leq t\leq T_\eps, s\in\T^1 }\|r_\eps(X(\cdot,s,t),t)\|_{L^p(-2\delta,2\delta)}} \leq C (M+1) \eps^{k+\frac1p+1}\quad \text{for all }\eps\in (0, \eps_0 ).
  \end{equation}
%
\end{lem}
\begin{proof}
  First let $(\tr_\eps)_{\eps\in (0,1)}\in \mathcal{R}_{k,\alpha}$ and $1\leq p<\infty$. Then
  \begin{alignat*}{1}
     \| \sup_{(x,t)\in\Gamma(2\delta)}|\tr_\eps(\cdot,x,t)|\|_{L^p(\R)}^p & 
\leq C\eps^{kp} \int_{-\infty}^\infty e^{-\alpha p|r|/\eps} \sd r= C\eps^{kp+1} \int_{-\infty}^\infty e^{-\alpha p|z|} \sd z = C' \eps^{kp+1}
  \end{alignat*}
uniformly in $\eps\in (0,1)$. This implies \eqref{eq:RemainderEstim1}. If $p=\infty$, then \eqref{eq:RemainderEstim1} follows immediately from \eqref{eq:EstimRkalpha}.

 Now, if even $(\tr_\eps)_{\eps\in (0,1)}\in \mathcal{R}_{k,\alpha}^0$, then \eqref{eq:RemainderVanish} implies
\begin{equation*}
  {r_\eps (X(r,s,t),t)}= \int_0^r \partial_{\no_{\Gamma_t}} \tr_\eps\left(\tfrac{r}\eps-h_\eps (s,t),X(z,s,t),t\right) \sd z\quad \text{for all }r\in (-2\delta,2\delta).
\end{equation*}
Hence
\begin{alignat*}{1}
  |r_\eps (X(r,s,t),t)| &\leq C\eps^k |r| e^{-\alpha|r/\eps-h_\eps(s,t)|}\\
&\leq C\eps^{k+1} |r/\eps-h_\eps(s,t)| e^{-\alpha|r/\eps- h_\eps(s,t)|}+ CM \eps^{k+1}e^{-\alpha|r/\eps- h_\eps(s,t)|}.
  \end{alignat*}
Using this estimate together with $\int_{-\infty}^\infty |z|^{jp}e^{-\alpha|z|p}\sd z<\infty$ for all $j\in\N_0$, one proves \eqref{eq:RemainderEstim2} in a similar way as before.
\end{proof}
\begin{cor}\label{cor:repsEstim2}
  Let $r_\eps$, $h_\eps$, $M$ and $T_\eps,0<\eps\leq \eps_0$ be as in Lemma~\ref{lem:Remainder},  $(\tr_\eps)_{0<\eps<1}\in \mathcal{R}_{k,\alpha}$ for some $\alpha>0$, $k\in\R$ and let $j=1$ if even $(\tr_\eps)_{0<\eps<1}\in \mathcal{R}_{k,\alpha}^0$ and $j=0$ else.
  Then there is some $C>0$, independent of $T_\eps,0<\eps\leq \eps_0$, $\eps_0\in (0,1)$  such that
  \begin{equation*}
    \left\|{a(P_{\Gamma_t}(\cdot))r_\eps \varphi} \right\|_{L^1(\G_t(2\delta))} \leq C(1+M)^j\eps^{1+k+j}\|\varphi\|_{H^1(\Omega)}\|a\|_{L^2(\Gamma_t)},
  \end{equation*}
  \begin{equation*}
    \left\| {a(P_{\Gamma_t}(\cdot))r_\eps} \right\|_{L^2(\G_t( 2\delta))} \leq C (1+M)^j\eps^{\frac 12+k+j} \|a\|_{L^2(\Gamma_t)}
  \end{equation*}
  uniformly for all $\varphi\in H^1(\Omega)$, $a\in L^2(\Gamma_t)$, $t\in [0,T_\eps]$, and $\eps\in (0,\eps_0]$.
\end{cor}
\begin{proof}
  Using a change of variables $(x,t)\mapsto (r,s,t)$ and then Lemma \ref{lem:Remainder}, we obtain
  \begin{align*}
    &  \left\|{a(P_{\Gamma_t}(\cdot))r_\eps\varphi} \right\|_{L^1(\G_t(2\delta))} \\
    &=\int_{-2\delta}^{2\delta}\int_{\T^1} |a(X_0(s,t))| \left|\tr_\eps \left(\tfrac{r}\eps - h_\eps(s,t),X(r,s,t),t \right)\right||\varphi(X(r,s,t))|J_t(r,s)\sd s\sd r \\
    &\leq C\int_{\T^1} |a(X_0(s,t))| \left\|\sup_{(x,t)\in\Gamma(2\delta)}|\tr_\eps (\tfrac{\cdot}\eps, x,t) |\right\|_{L^1(\R)} \sup_{r\in (-2\delta,2\delta)}|\varphi(X(r,s,t))|  \sd s\\
    &\leq C(1+M)^j \eps^{1+j+k} \left(\int_{\T^1}\sup_{|r|\leq 2\delta }|\varphi(X(r,s,t))|^2\sd s\right)^{\frac12}\|a\|_{L^2(\Gamma_t)}\\
    &\leq C(1+M)^j\eps^{1+j+k} \|\varphi\|_{H^1(\Omega)}\|a\|_{L^2(\Gamma_t)}
  \end{align*}
for all $a\in L^2(\Gamma_t)$, $\eps\in (0,1]$, $t\in [0,T_\eps]$, and $\varphi\in H^1(\Omega)$, which proves the first estimate.

Similarly,
we obtain
{  \begin{align*}
     \left\| a(P_{\Gamma_t}(\cdot))r_\eps \right\|_{L^2(\G_t( 2\delta))} &= \left(\int_{-2\delta}^{2\delta}\int_{\T^1} |a(X_0(s,t))|^2 \left|r_\eps (X(r,s,t),t))\right|^2J_t(r,s)\sd s\sd r\right)^{\frac12} \\
    &\leq C\left(\int_{\T^1} \sup_{0\leq t\leq T_\eps, s\in \T^1}\left\|r_\eps (X(\cdot,s,t)) \right\|_{L^2(-2\delta,2\delta)}^2 |a(X_0(s,t))|^2  \sd s\right)^{\frac12}\\
    &\leq C(1+M)^j \eps^{1/2+j+k}  \|a\|_{L^2(\Gamma_t)}
  \end{align*}}
for all $a\in L^2(\Gamma_t)$, $\eps\in (0,1)$, and $t\in [0,T_\eps]$.
\end{proof}
\begin{rem}
  In the following we will apply the results of this subsection to remainder terms $(\tr_\eps)_{\eps\in (0,\eps_0)}$ for some $\eps_0\in (0,1]$. This case can easily be reduced to the former case by considering  $(\check{r}_\eps)_{\eps\in (0,1)}$ with
  \begin{equation*}
    \check{r}_\eps =
    \begin{cases}
      \tr_\eps &\text{if }\eps \in (0,\eps_0),\\
      0 &\text{else}.
    \end{cases}
  \end{equation*}
  In this way we can apply all definitions and statements above with $\eps\in (0,\eps_0)$ instead of $\eps \in (0,1)$.
\end{rem}


\subsection{Parabolic Equations on Evolving Hypersurfaces}\label{subsec:ParabolicEq}

Throughout this subsection $0<T<\infty$ is arbitrary, but fixed. In later applications we will choose  $T=T_0$.
We shall denote the function space
\begin{equation}\label{yuning:1.31}
  X_T:=L^2(0,T;H^{5/2}(\T^1))\cap H^1(0,T;H^{1/2}(\T^1)),
\end{equation}
equipped with the norm
\begin{equation*}
  \|u\|_{X_T} =\|u\|_{L^2(0,T;H^{5/2}(\T^1))}+\|u\|_{H^1(0,T;H^{1/2}(\T^1))}+ \|u|_{t=0}\|_{H^{3/2}(\T^1)}.
\end{equation*}
We note that
\begin{equation}\label{eq:EmbeddingE1}
X_T\hookrightarrow BUC([0,T]; H^{3/2}(\T^1))\cap L^4(0,T; H^2(\T^1))
\end{equation}
{and the operator norm of the embedding is uniformly bounded in $T$.}
 Likewise, we define
 \begin{equation}\label{yuning:1.33}
  Y_T:=L^2(\R\times\T^1\times(0,T)),
\end{equation}

\begin{thm}\label{thm:ParabolicEqOnSurface}
  Let $w\colon \T^1\times [0,T] \to \Rn$ and $a\colon \T^1\times [0,T]\to \R$ be smooth.
  For every $g\in L^2(0,T;H^{\frac 12}(\T^1))$ and $h_0\in H^{ \frac32 }(\T^1)$ there is a unique solution $h\in X_T$ of
  \begin{alignat}{2}\label{eq:h1}
    D_t h+ w\cdot \nabla_\G h -\Delta_\Gamma h + a h &= g&\qquad& \text{on }\T^1\times [0,T],\\\label{eq:h2}
    h|_{t=0} &=h_0 && \text{on } \T^1.
  \end{alignat}
\end{thm}
\begin{proof}
  According to \eqref{yuning:1.27}, equation \eqref{eq:h1}  is equivalent to
  \begin{equation*}
    \partial_t h- \mathcal{L}  h =g\qquad  \text{on }\T^1\times [0,T]
  \end{equation*}
  for some uniformly elliptic operator $ \mathcal{L} $ on $\T^1$ with smooth coefficients depending on $(s,t)\in \T^1\times [0,T]$. More precisely for every fixed  $t \in[0,T]$  the operator is $\Lambda(\theta)$-elliptic in the sense of \cite[Definition~3.3]{EscherSeiler} for every $0<\theta <\pi$.

  Now the result follows from known results for parabolic equations. E.g.~one can argue as follows:   Let $ \mathcal{L} (t_0)$, with $t_0\in [0,T]$, denote the same operator, where $t$ in the coefficients is replaced by a fixed $t_0\in [0,T]$. Moreover, let
\begin{equation*}
  \mathcal{A}(t_0)\colon \mathcal{D}(\mathcal{A}(t_0)):=H^{\frac52}(\T^1)\subseteq H^{\frac12}(\T^1)\to H^{\frac12}(\T^1)
\end{equation*}
be its realization on $H^{\frac12}(\T^1)$. Then there is some $\lambda_0\geq 0$ such that $\lambda_0+\mathcal{A}(t_0)$
possesses a bounded $H^\infty$-calculus due to \cite[Theorem~4.10]{EscherSeiler}. This implies that for every $\mathcal{A}(t_0)$ has maximal $L^p$-regularity on every finite time interval and for every $1<p<\infty$ due to \cite[Theorem~3.2]{DoreVenni}. Now the theorem follows from \cite[Theorem~2.7]{ArendtMaxReg}.
\end{proof}

In order to couple \eqref{eq:h1} to the two-phase Stokes system, we need

\begin{lem}\label{lem:Stokes}
 For every $t\in [0,T]$, $\mathbf{a}\in H^{\frac12}(\Gamma_t)^2$ and $\mathbf{f}\in L^2(\Omega)^2$ there is a unique solution $(\ve,p) \in H^1_0(\Omega)^2\times L^2(\Omega)$ with $\ve|_{\Omega^\pm(t)}\in H^2(\Omega^\pm(t))^2$, $p|_{\Omega^\pm(t)}\in H^1(\Omega^\pm(t))$, with $\int_\Omega p\sd x=0$   of
  \begin{alignat}{2}\label{eq:TwoPhaseStokes1'}
    -\Delta \ve +\nabla p &= \mathbf{f} &\qquad&\text{in }\Omega^\pm (t),\\\label{eq:TwoPhaseStokes2'}
    \Div \ve &= 0 &&\text{in }\Omega^\pm (t),\\\label{eq:TwoPhaseStokes3'}
    [2D \ve-p\mathbf{I}]\no_{\Gamma_t}&= \mathbf{a}  && \text{on }\Gamma_t.
   \end{alignat}
   Moreover, there is a constant $C>0$ independent of $t\in [0,T]$ such that
   \begin{equation}\label{eq:TwoPhaseEstim}
     \|\ve\|_{H^1(\Om)}+ \sum_{\pm}\left( \|\ve\|_{H^{2}(\Omega^\pm (t))} + \|p\|_{H^1(\Omega^\pm(t))} \right)\leq C\left(\|\mathbf{f}\|_{L^2(\Om)}+ \|\mathbf{a}\|_{H^{1/2}(\Gamma_t)}\right).
   \end{equation}
\end{lem}
\begin{proof}
  First of all, \cite[Theorem 1.1]{ShibataShimizuTwoPhase} implies the existence of a unique solution $(\ve,p)$ and \eqref{eq:TwoPhaseEstim} for some $C>0$, which might depend on $t\in [0,T]$.
Hence it only remains to prove that $C>0$ can be chosen independently of $t\in [0,T]$. Standard perturbation arguments imply that for every $t_0\in [0,T]$ there is some $\eps>0$ such that \eqref{eq:TwoPhaseEstim} holds true for some $C>0$ and any $t\in [0,T]\cap (t_0-\eps,t_0+\eps)$. (Alternatively, one can verify that the constants and cut-off functions in the proof of \cite[Theorem 1.1]{ShibataShimizuTwoPhase} can be chosen independently of $t\in [0,T]\cap (t_0-\eps,t_0+\eps)$ if $\eps>0$ is sufficiently small.) Since $[0,T]$ is compact, there is some $C>0$ such that \eqref{eq:TwoPhaseEstim} holds true for any $t\in[0,T]$.
\end{proof}
\begin{cor}\label{cor:ParabolicEqOnSurfaceStokes}
  Let $w\colon \T^1\times [0,T] \to \Rn$, $a\colon \T^1\times [0,T]\to \R$ and $b\colon \Gamma\to \R^2$ be smooth.
  For every $g\in L^2(0,T;H^{\frac 12}(\T^1))$ and $h_0\in H^{ \frac32 }(\T^1)$ there is a unique solution $h\in X_T$ of
  \begin{alignat}{2}\label{eq:h1'}
    D_t h+ w\cdot \nabla_\G h -\Delta_\Gamma h + a h &= X_0^\ast (\ve_{\no})+g&\qquad& \text{on }\T^1\times [0,T],\\\label{eq:h2'}
    h|_{t=0} &=h_0 && \text{on } \T^1,
  \end{alignat}
  where  for every $t\in [0,T]$, $\ve=\ve(\cdot,t) \in H^1_0(\Omega)^2$ with $\ve|_{\Omega^\pm(t)}\in H^2(\Omega^\pm(t))^2$ and $p=p(\cdot,t)\in H^1(\Om\setminus\Gamma_t)\cap  L^2(\Omega)$ with $\int_\Omega p\sd x=0$  are determined by
  \begin{alignat}{2}\label{eq:TwoPhaseStokes1}
    -\Delta \ve +\nabla p &= 0 &&\text{    in }\Omega^\pm (t), t\in [0,T],\\\label{eq:TwoPhaseStokes2}
    \Div \ve &= 0 &&\text{    in }\Omega^\pm (t), t\in [0,T],\\\label{eq:TwoPhaseStokes3}
    [2D \ve-p\mathbf{I}]\no_{\Gamma_t}&= b X_0^{\ast,-1}(h)-\sigma X_0^{\ast,-1}(\Delta_\Gamma h)\no_{\Gamma_t}  && \text{    on }\Gamma_t, t\in [0,T].
  \end{alignat}
 Here $\sigma=\int_{\R} \theta'_0(\rho)^2\, \sd \rho$ as before and $\ve_\no=\no\cdot \ve$. 
Moreover, if $g$ and $h_0$ are smooth, then $h$ is smooth and $\ve|_{\Omega^\pm}, p|_{\Omega^\pm}$ are smooth in $\overline{\Omega^\pm}$.
\end{cor}
\begin{proof}
For given $h\in H^{\frac52}(\T^1)$ and fixed $t\in [0,T]$, let $\ve\in H^2(\Om\setminus \Gamma_t)^2\cap H^1_0(\Om)^2$, $p\in H^1(\Om\setminus \Gamma_t)\cap L^2(\Om)$ with $\int_\Omega p\, dx=0$ be determined as solution of \eqref{eq:TwoPhaseStokes1}-\eqref{eq:TwoPhaseStokes3}.
 Then testing \eqref{eq:TwoPhaseStokes1} by $\ve$, integration by parts, and using \eqref{eq:TwoPhaseStokes3} yields
 \begin{equation*}
   \|\ve\|_{H^1(\Om)}\leq C\|(h,\Delta_\Gamma h)\|_{H^{-\frac12}(\T^1)}\leq C' \|h\|_{H^{\frac32}(\T^1)}
 \end{equation*}
for some $C,C'$ independent of $h$ and $t\in [0,T]$.

 Now we define $\mathcal{B}(t)\colon H^{\frac52}(\T^1)\to H^{\frac12}(\T^1)$  by $\mathcal{B}(t) h:=  X_0^\ast (\ve_{\no})$.
Then the previous estimate implies
\begin{equation*}
  \|\mathcal{B}(t) h\|_{H^{\frac12}(\T^1)}\leq C\|h\|_{H^{\frac32}(\T^1)}
\end{equation*}
for some $C$ independent of $h$ and $t\in [0,T]$. Hence we can extend $\mathcal{B}(t)$ to  $\mathcal{B}(t)\colon H^{\frac32}(\T^1)\to H^{\frac12}(\T^1)$ by continuity. Moreover, $H^{\frac32}(\T^1)$ is relatively closer to $H^{\frac12}(\T^1)$ compared with $H^{\frac52}(\T^1)$ in the sense of \cite{ArendtMaxReg} because of the compactness of the embedding $H^{\frac52}(\T^1)\hookrightarrow H^{\frac32}(\T^1)$, cf. \cite[Example 2.9. (d)]{ArendtMaxReg}.
Hence the existence of a unique solution $h\in X_T$ as in the statement of the corollary follows from \cite[Theorem 2.11]{ArendtMaxReg} and the previous theorem.

Finally, smoothness of $h$ and $\ve|_{\Omega^\pm}$ in  $\overline{\Omega^\pm}$ can be shown by standard localization techniques. To this end one transforms $\Omega^\pm$ and $\Gamma$ locally (in space and time) to the situation of a flat interface and applies e.g.\ difference quotients in tangential directions or the time direction using the a~priori estimates in $L^2$-Sobolev spaces obtained before. This yields one order of higher regularity of $h$, $\ve|_{\Omega^\pm}$ and $p|_{\Omega^\pm}$ in tangential and time directions. Then one order higher regularity in normal directions follows from \eqref{eq:TwoPhaseStokes1} and \eqref{eq:TwoPhaseStokes2}. Repeating this arguments yields the claimed smoothness.
\end{proof}

\begin{thm}\label{lem:InftyEstimLinPart}
  Let  $g\colon \R\times\T^1\times (0,T)  \to \R$ be a smooth function with $g(\cdot,s,t)\in L^2(\R)$ and
    \begin{equation}\label{eq:gtheta}
    \int_\R g(\rho,s,t)\theta_0'(\rho)\sd \rho =0 \qquad \text{for all}\ (s,t)\in \T^1\times [0,T].
  \end{equation}
  Then there is a unique smooth solution $c=c(\rho,s,t )$ solving
  \begin{equation}\label{eq:c1}
    \begin{split}
      \eps^2(D_t c  -\Delta_\Gamma c)-\partial_{\rho}^2 c+f''(\theta_0)c &= g\quad \text{on } \R\times\T^1\times [0,T] ,\\
    c|_{t=0} &=0 \quad\text{on } \R\times\T^1.
    \end{split}
  \end{equation}
  Moreover,
  \begin{equation}\label{eq:OrthogonalitySol}
        \int_\R c(\rho,s,t)\theta_0'(\rho)\sd \rho =0 \qquad \text{for all}\ (s,t)\in \T^1\times [0,T]
  \end{equation}
  and there is a constant $C$ independent of $g,c$, $T'\in (0,T]$ and $\eps\in (0,1]$ such that
  \begin{alignat}{1}\nonumber
      \eps\|  c\|_{L^\infty(0,T';L^2(\T^1\times \R))}+ \| (  c,\p_\rho c,\eps \partial_s c) \|_{L^2(\R\times\T^1\times(0,T'))}  &\leq C\|g\|_{L^2(\R\times\T^1\times(0,T'))}, \\\nonumber
  \eps\| \p_\rho c\|_{L^\infty(0,T';L^2(\T^1\times \R))}+     \| \p_\rho^2 c \|_{L^2(\R\times\T^1\times(0,T'))}  &\leq C\|(g,\p_\rho g)\|_{L^2(\R\times\T^1\times(0,T'))}, \\\label{yuning:1.12}
       \eps\| \p_s c\|_{L^\infty(0,T';L^2(\T^1\times \R))}+ \| (  \partial_s c, \partial_s \p_\rho c,  \eps \partial_s^2 c ) \|_{L^2(\R\times\T^1\times(0,T'))}    &\leq C\|\p_s g \|_{L^2(\R\times\T^1\times(0,T'))}.
  \end{alignat}
Furthermore, if additionally $\|(1+|\rho|)^kg\|_{L^2(\R\times\T^1\times(0,T'))}$ is finite for some $k\in\N$, then there is some $C_k$ independent of $c,g,\eps, T'$ such that
  \begin{alignat}{1}\nonumber
      &\eps\|  \rho^k c\|_{L^\infty(0,T';L^2(\T^1\times \R))}\\\label{yuning:1.12Weigthed}
      &+ \| ( \rho^k c,\p_\rho (\rho^kc),\eps \rho^k \partial_s c) \|_{L^2(\R\times\T^1\times(0,T'))}  \leq C_k\|(1+|\rho|)^k g\|_{L^2(\R\times\T^1\times(0,T'))}.
  \end{alignat}
 Finally, for any  $1<p\leq \infty$ there is a constant $C_p$ independent of $g,c$, $T'\in (0,T]$ and $\eps\in (0,1]$ such that

  \begin{equation}\label{yuning:1.11}
    \begin{split}
      & \eps^{\frac1p}\sup_{(s,t)\in \T^1\times [0,T']} \|c(\cdot, s,t)\|_{H^1(\R)}\leq C_p \|\sup_{s\in \T^1}\|g(\cdot, s,t)\|_{L^2(\R)}\|_{L^{2p}(0,T')}.
    \end{split}
  \end{equation}
\end{thm}
\begin{proof}
The proof  of existence is similar to the proof of Theorem~\ref{thm:ParabolicEqOnSurface}.
If we denote  $$\mathcal{L} \phi:= -\partial_\rho^2 \phi +f''(\theta_0(\rho))\phi\quad \text{for all }\phi=\phi(\rho)\in H^2(\R),$$  then
equation \eqref{eq:c1}  is equivalent to
  \begin{equation}\label{yuning:1.19}
    \eps^2(\partial_t - \mathcal{A} )c +\mathcal{L}c= g\qquad  \text{on }\T^1\times [0,T]
  \end{equation}
  for some uniformly elliptic operator $\mathcal{A}$ on $\T^1$ with smooth coefficients depending on $(s,t)\in \T^1\times [0,T]$.
    Then $\mathcal{L}$ is a non-negative self-adjoint operator on $L^2(\R)$ with $\ker \mathcal{L} = \operatorname{span} \{\theta_0'\}$,
which is invertible on the orthogonal complement of $\operatorname{span} \{\theta_0'\}$, cf.\ e.g.\ \cite[(1.8)]{DeMottoniSchatzman}. By \cite[Theorem~3.11 and Theorem~4.10]{EscherSeiler} for every $\theta\in (0,\pi)$ there are some $\lambda_0,\lambda_1\geq 0$ such that $\lambda_0+ \mathcal{A}(t_0)$ and $\lambda_1 + \mathcal{L}$ possess a bounded $H^\infty$-calculus on $L^2(\T^1)$, $L^2(\R)$, respectively, and a sector $\C\setminus \Lambda(\theta)$, where $\Lambda(\theta)= \{z\in\C: \arg (z)\in [\theta,2\pi-\theta]\}$. The same is true if $L^2(\T^1)$ and $L^2(\R)$ are replaced $L^2(\T^1\times \R)$ since $\mathcal{A}(t_0)$ does not act on $\rho \in \R$ and $\mathcal{L}$ does not act on $s\in\T^1$. Hence \cite[Theorem~6.3]{KaltonWeis} implies that $\lambda_0+\lambda_1+ \mathcal{A}(t_0)+ \mathcal{L}$ is $\mathcal{R}$-sectorial on $\Lambda(2\theta)$ for every $t_0\in [0,T]$. By choosing $\theta <\frac{\pi}4$, we obtain that $\mathcal{A}(t_0)+ \mathcal{L}$ has maximal $L^p$-regularity on every finite time interval and for every $1<p<\infty$ due to \cite[Theorem~3.2]{DoreVenni}. Now the existence of a unique solution
\begin{alignat*}{1}
c\in H^1(0,T;L^2(\T^1\times\R))\cap L^2(0,T;H^2(\T^1\times\R))  
\end{alignat*}
 follows again from \cite[Theorem~2.7]{ArendtMaxReg}. Furthermore, standard results on parabolic equations imply smoothness.

In order to prove \eqref{eq:OrthogonalitySol}, we multiply  the equation \eqref{yuning:1.19} with $\theta'_0(\rho)$, use \eqref{eq:gtheta}, and the fact that $\mathcal{L}$ is self-adjoint. This yields
\begin{alignat*}{2}
  \eps^2 (\partial_t -\mathcal{A}) \int_{\R}c(\rho,s,t)\theta_0'(\rho)\sd \rho &=0&\quad& \text{on } \T^1\times [0,T],\\
  \left.\int_{\R}c(\rho,s,t)\theta_0'(\rho)\sd \rho\right|_{t=0} &=0 &\quad& \text{on } \T^1.
\end{alignat*}
Hence, \eqref{eq:OrthogonalitySol} follows from the unique solvability of the latter system. Consequently, we also obtain
 \begin{equation}\label{yuning:1.13}
    \int_{\R}c(\rho,s,t) (\mathcal{L}c)(\rho,s,t)\sd \rho\geq C \|c(\cdot,s,t)\|_{H^1(\R)}^2
 \end{equation}
since $\mathcal{L}$ is positive on the orthogonal complement of $\operatorname{span} \{\theta_0'\}$.
Now the proof of \eqref{yuning:1.12} follows in a straight forward manner by testing the equation with $c$, differentiating with respect to $\rho$, $s$ and testing with $\partial_\rho c$,  $\p_s c$, respectively, and  integration by parts. The details are omitted here. {Moreover, in order to prove \eqref{yuning:1.12Weigthed} one uses that
\begin{equation*}
  \eps^2 (D_t -\Delta_\Gamma)(\rho^kc)- \partial_\rho^2 (\rho^k c)+ f''(\theta_0)\rho^k c= -2k \rho^{k-1}\partial_\rho c - k(k-1) \rho^{k-2}c + \rho^kg.
\end{equation*}
Now testing with $\rho^k c$ and using the interpolation inequalities
\begin{equation*}
  \|\rho^{k-1} \partial_\rho c\|_{L^2}\leq \|\rho^k\partial_\rho c\|^{\frac{k-1}k}_{L^2}\|\partial_\rho c\|_{L^2}^{\frac1k},\qquad
  \|\rho^{k-2} c\|_{L^2}\leq \|\rho^kc\|^{\frac{k-2}k}_{L^2}\|c\|_{L^2}^{\frac2k},
\end{equation*}
Young's inequality, and \eqref{yuning:1.12} one derives \eqref{yuning:1.12Weigthed} in the same way as before.}

 Finally, we prove \eqref{yuning:1.11}. We multiply the differential equation for $c$ in \eqref{eq:c1} with $\mathcal{L} c$, integrate  with respect to $\rho$, and obtain
\begin{alignat*}{1}
  \eps^2 (\partial_t {-\mathcal{A}}) \int_{\R}c(\rho,s,t) (\mathcal{L}c)(\rho,s,t)\sd \rho + \int_\R (\mathcal{L}c)^2(\rho,s,t)\sd \rho  = \int_\R g(\rho,s,t) (\mathcal{L}c)(\rho,s,t)\sd \rho\\
  \leq \|g(\cdot,s,t)\|_{L^2(\R)}\|\mathcal{L}c(\cdot,s,t)\|_{L^2(\R)}.
\end{alignat*}
Using the fact that \[\|(\mathcal{L}c)(\cdot,s,t)\|_{L^2(\R)}^2\geq 2\kappa \|(\mathcal{L}^{\frac12}c)(\cdot,s,t)\|_{L^2(\R)}^2= 2\kappa \int_{\R} c(\cdot,s,t) (\mathcal{L}c)(\cdot,s,t) \sd \rho\] for some $\kappa >0$ (since $c(\cdot,s,t)$ is orthogonal to $\theta_0'$)  and Young's inequality we deduce that
\begin{alignat*}{1}
  \eps^2 (\partial_t{-\mathcal{A}}+\frac{\kappa}{\eps^2}) \int_{\R}c(\rho,s,t) (\mathcal{L}c)(\rho,s,t)\sd \rho    &\leq \sup_{ s\in \T^1}\|g(\cdot,s,t)\|_{L^2(\R)}^2=:M(t) .
\end{alignat*}
Since the function $w(t):= \frac{1}{\eps^2}\int_0^t e^{-\tfrac{\kappa}{\eps^2} (t-s)} M(s)\sd s$ satisfies $\eps^2(w'+\tfrac{\kappa}{\eps^2}w)=M$, it solves the equation
\begin{alignat*}{2}
  \eps^2 (\partial_t {-\mathcal{A}}+\tfrac{\kappa}{\eps^2}) w &= M&\quad& \text{on } \T^1\times [0,T_0],\\
  w|_{t=0} &=0&\quad& \text{on } \T^1.
\end{alignat*}
Thus, by the comparison principle and H\"{o}lder's inequality, we have
\begin{equation*}
  \int_{\R}c(\rho,s,t) (\mathcal{L}c)(\rho,s,t)\sd \rho \leq w(t)\leq  C_p \frac{\|M\|_{L^p(0,T')}}{\eps^{2/p}}\quad  \text{for all }t\in [0,T'] ,
\end{equation*}
which combined with \eqref{yuning:1.13} implies \eqref{yuning:1.11}.
\end{proof}

\subsection{Spectral Estimate}

In this subsection we assume that $\Omega\subseteq \R^d$ is a bounded domain and $\Gamma_t\subseteq \Om$, $t\in [0,T_0]$, $T_0>0$, are given smoothly evolving closed and compact $C^\infty$-hypersurfaces, dividing $\Om$ in disjoint domains $\Omega^+(t)$ and $\Omega^-(t)$ as before, and
\begin{alignat*}{2}
    c_A(x)&= c_{A,0}(x)+ \eps^2c_{A,2+}(x), &\quad& \text{for all } x\in \Om,\\
    c_{A,0}(x)&=\zeta\circ d_\Gamma \theta_0(\rho)+(1-\zeta\circ d_\Gamma)\left(
  \chi_{\Omega^+(t)}-\chi_{\Omega^-(t)} \right)&\quad& \text{for all } x\in \Om
\end{alignat*}
where $\zeta$ satisfies \eqref{yuning:1.34}.
Moreover, we assume that $\dist(\Gamma_t,\partial\Omega)> 2\delta$ for all $t\in [0,T_0]$.
In this subsection, for given  continuous functions $(\tilde{h}_\eps)_{0<\eps< 1} \colon \G\to \R$ with $\Gamma:=\bigcup_{t\in[0,T_0]} \Gamma_t \times {\{t\}}$, we define the stretched variable $\rho$ by
\begin{equation*}
  \rho = \frac{d_\Gamma(x,t)}\eps - \tilde{h}_\eps (P_{\Gamma_t}(x),t).
\end{equation*}
 Finally, we assume that
\begin{equation}\label{eq:BoundhcA}
\sup_{\eps \in (0,1)}\left(\sup_{(p,t)\in \Gamma} |\tilde{h}_\eps (p,t)|+ \sup_{x\in\Om, t\in[0,T_0]} |c_{A,2+} (x,t)|\right)\leq M
\end{equation}
for some $M>0$.  We remark that later we will apply the results of this subsection to $\tilde{h}_\eps (p,t)= h_\eps(X_0^{-1}(p,t),t)$ for some $h_\eps\colon \T^1\times [0,T_0]\to \R$. But for the following proof the present formulation is more convenient. Moreover, the result holds true in any dimension $d\geq 2$.

The following spectral estimate will be a key ingredient for the proof of convergence.
\begin{thm}\label{thm:Spectral}
Let $c_A$ be as above and \eqref{eq:BoundhcA} be satisfied for some $M>0$. Then there are some $C,\eps_0>0$, independent of $\tilde{h}_\eps, c_A$, such that for every $\psi\in H^1(\Om)$, $t\in[0,T_0]$, and $\eps\in (0,\eps_0]$ we have
  \begin{equation*}
    \int_{\Om}\left(|\nabla\psi(x)|^2+ \eps^{-2}f''(c_A(x,t))\psi^2(x)\right)\sd x\geq -C\int_\Om \psi^2 \sd x + \int_{\Om\setminus \Gamma_t(\delta)} |\nabla \psi|^2\sd x +  \int_{\Gamma_t({\delta})} |\nabla_\btau \psi|^2\sd x.
  \end{equation*}
\end{thm}
\begin{proof}
In the sequel, we shall fix $t\in [0,T_0]$.
For sufficiently small $\eps_0\in (0,1]$ and $\eps\in (0,\eps_0]$, consider the set
\begin{equation*}
\Gamma_\eps:=\{x\in\Gamma_t(2\delta)\,|\,d_\Gamma(x,t)\in I_\eps^p, p=P_{\Gamma_t}(x) \}
\end{equation*}
 where $I_\eps^p$ is defined via
\begin{equation*}
I_{\eps}^p:=(-\tfrac{3\delta}4+\eps \tilde{h}_\eps(p,t),\tfrac{3\delta}4+\eps \tilde{h}_\eps(p,t)).
\end{equation*}
As in \cite{ChenSpectrumAC} we assume for notational simplicity that $\delta=1$. Then
it is evident that $x\in \Gamma_\eps$ is equivalent to $\rho\in I_\eps:=( -\tfrac 3{4\eps},\frac 3{4\eps})$.  If we denote $\tilde{\psi}(r,p)=\psi(x),\tilde{c}_A(r,p,t)=c_A(x,t)$, then
\begin{equation}\label{problem1}
\begin{split}
    &\int_{\Om}\left(|\nabla\psi(x)|^2 + \eps^{-2}f''(c_A(x,t))\psi^2(x)\right)\sd x\\
 &\ge \int_{\Gamma_t}\int_{I_\eps^p}\left(|\p_r \tilde{\psi}(r,p)|^2 + \eps^{-2}f''(\tilde{c}_A(r,p,t))\tilde{\psi}^2(r,p)\right)J_t(r,p)\sd r\sd p\\
&\quad + \int_{\Om\setminus {\Gamma_\eps}} |\nabla \psi|^2\sd x +  \int_{{\Gamma_\eps}} |\nabla_\btau \psi|^2\sd x,
\end{split}
\end{equation}
  since $f''(c_A(x,t))\ge 0$ for all $x\in {\Om\setminus\Gamma_t(\frac12)\supseteq \Om \setminus \Gamma_\eps}$, $t\in [0,T_0]$, and $0<\eps\leq\eps_0$ for sufficiently small $\eps_0$. {Moreover, since $|\nabla\psi(x)|^2\geq |\nabla_\btau \psi(x)|^2$ for all $x\in \Gamma_t(1)\setminus \Gamma_\eps$, and $\Gamma_t(1)\supseteq \Gamma_\eps$ for sufficiently small $\eps$, we have
    \begin{equation*}
      \int_{\Om\setminus {\Gamma_\eps}} |\nabla \psi|^2\sd x +  \int_{{\Gamma_\eps}} |\nabla_\btau \psi|^2\sd x\geq
\int_{\Om\setminus \Gamma_t(1)} |\nabla \psi|^2\sd x +  \int_{\Gamma_t(1)} |\nabla_\btau \psi|^2\sd x.
    \end{equation*}}
 Let us fix $(p,t)\in\Gamma$ and perform a change of variables in $r$:
 \begin{align*}
 F_\eps^p: I_\eps\rightarrow I_\eps^p,\ z\mapsto r=\eps (z + \tilde h_\eps(p,t)).
 \end{align*}
 This gives
 \begin{align*}
&\int_{I_\eps^p}\left(|\p_r\tilde\psi(r,p)|^2 + \eps^{-2}f''(\tilde{c}_A(r,p,t))\tilde\psi^2(r,p)\right) J_t(r,p)\sd r\\
& = \int_{I_\eps}\left(\eps|\p_r\tilde\psi(F_\eps^p(z),p)|^2+ \eps^{-1}f''(\tilde{c}_A(F_\eps^p(z),p,t))\,\tilde\psi^2(F_\eps^p(z),p)\right) J_t(F_\eps^p(z),p)\sd z\\
& =\eps^{-2} \int_{I_\eps}\left(|\p_z \Psi(z,p)|^2 + f''(\tilde{c}_A(F_\eps^p(z),p,t))\,\Psi^2(z,p)\right) \tilde J_t(z,p)\sd z,
  \end{align*}
where $\Psi(z,p):=\sqrt\eps\tilde\psi(F_\eps^p(z),p)$ and $\tilde J_t(z,p):=J_t(F_\eps^p(z),p)$. Furthermore, with $\hat\Psi:=\tilde J_t^{\frac12} \Psi$, we compute
   \begin{align*}
&\int_{I_\eps^p}\left(|\p_r\tilde\psi(r,p)|^2 + \eps^{-2}f''(\tilde{c}_A(r,p,t))\tilde\psi^2(r,p)\right) J_t(r,p)\sd r\\
& =\eps^{-2}\int_{I_\eps}(|\p_z\hat\Psi(z,p)|^2 + f''(\theta_0(z)+\eps^2\hat{c}_{A,2+}(z,p,t))\,\hat\Psi^2(z,p))\sd z\\
&\quad -\eps^{-2}\int_{I_\eps} 2\partial_z\Psi(z,p)\tilde J_t^{\frac12}(z,p)\Psi(z,p)\partial_z\tilde J_t^{\frac12}(z,p) + \Psi^2(z,p)(\partial_z\tilde J_t^{\frac12}(z,p))^2\sd z\\
& =\eps^{-2}\int_{I_\eps}(|\p_z\hat\Psi(z,p)|^2 + f''(\theta_0(z)+\eps^2\hat{c}_{A,2+}(z,p,t))\,\hat\Psi^2(z,p))\sd z\\
&\quad -\eps^{-2}\int_{I_\eps}\frac12 \partial_z\Psi^2(z,p)\partial_z\tilde J_t(z,p) + \frac14\Psi^2(z,p)(\partial_z\tilde J_t(z,p))^2/\tilde J_t(z,p)\sd z\\
& =\eps^{-2}\int_{I_\eps}(|\p_z\hat\Psi(z,p)|^2 + f''(\theta_0(z)+\eps^2\hat{c}_{A,2+}(z,p,t))\,\hat\Psi^2(z,p))\sd z\\
&\quad +\eps^{-2}\int_{I_\eps}\frac14 \hat\Psi^2(z,p)\left[2\partial_z^2\tilde J_t(z,p)/\tilde J_t(z,p) -  (\partial_z\tilde J_t(z,p))^2/\tilde J_t^2(z,p)\right]\sd z\\
&\quad -\eps^{-2}\frac12 \hat\Psi^2(z,p)\partial_z\tilde J_t(z,p)/\tilde J_t(z,p)\Big|_{z=-{3/4\eps}}^{z={3/4\eps}},
\end{align*}
where we used integration by parts in the last step. We conclude that
 \begin{align*}
&\int_{I^p_\eps}\left(|\p_r\tilde\psi(r,p)|^2 + \eps^{-2}f''(\tilde{c}_{A}(r,p,t))\tilde{\psi}^2(r,p)\right) J_t(r,p)\sd r\\
& =\eps^{-2}\int_{I_\eps}(|\p_z\hat\Psi(z,p)|^2 + f''(\theta_0(z))\,\hat\Psi^2(z,p))\sd z\\
&\quad + \int_{I_\eps}\tilde q(z,p,t)\hat\Psi^2(z,p)\sd z - \eps^{-1}\frac12 \hat\Psi^2(z,p)(\partial_r J_t)(F_\eps^p(z),p)/J_t(F_\eps^p(z),p)\big|_{z=-{3/4\eps}}^{z={3/4\eps}},
  \end{align*}
where
 \begin{align*}
\tilde q(z,p,t)&:=\eps^{-2}\big(f''(\theta_0(z)+\eps^2\hat{c}_{A,2+}(z,p,t))-f''(\theta_0(z))\big)\\
&\quad+\frac14(2(\partial_r^2J_t)(F_\eps^p(z),p)/J_t(F_\eps^p(z),p) - ((\partial_r J_t)(F_\eps^p(z),p))^2/J_t^2(F_\eps^p(z),p)).
 \end{align*}
Since $\Gamma$ is smooth and compact and due to \eqref{eq:BoundhcA}, we have
 \begin{align*}
 |\tilde q(z,p,t)|\le \sup_{|\xi|\leq 2M+2}|f'''(\xi)|\sup_{(x,t)\in \Omega}|c_{A,2+}(x,t)| + C\le C'
\end{align*}
for all $z\in\R$, $(p,t)\in \Gamma$, where $M$ is the uniform bound in \eqref{eq:BoundhcA}.
Now, we can proceed as in \cite[Proof of Theorem 2.3]{ChenSpectrumAC} to show that
\begin{align*}
&\int_{\Gamma_t}\int_{I_\eps^p}\left(|\p_r\tilde\psi(r,p)|^2 + \eps^{-2}f''(\tilde{c}_A(r,p,t))\tilde\psi^2(r,p)\right) J_t(r,p)\sd r\sd p\\
& \ge -\frac{1}{C}\int_{\Gamma_t}\int_{I_\eps}\Psi^2(z,p)\, \tilde J_t(z,p)\sd z\sd p\geq -\frac1C\int_{\Gamma_t(2\delta)}\tilde\psi^2(r,p)\sd r\sd p
\geq - \frac1C \int_\Omega \psi^2(x)\sd x
  \end{align*}
for some $C>0$ independent of $\psi, t\in[0,T_0]$, and $\eps\in (0,\eps_0]$. Altogether we obtain the claimed estimate.
\end{proof}


%% file: SharpInterfaceLimitNSACApproxStokes.tex
\section{Approximation of the Stokes System}\label{sec:ApproxVelocity}

\subsection{The Leading Part of the Velocity}

The aim of this section is to construct an approximation of  the following system
\begin{equation} \label{train:stokes}
  \left\{
 \begin{array}{rl}
   -\Delta \tilde{\ve}_A +\nabla \tilde{p}_A &= -\eps \Div (\nabla c_{A,0}\otimes \nabla c_{A,0}),\\
 \Div \tilde{\ve}_A &=0,\\
 \Second{\tilde{\ve}_A|_{\partial\Omega}} &\Second{=0}
 \end{array}
  \right.
\end{equation}
with the aid of the method of formally matched asymptotics,
where
\begin{equation}\label{yuning:ca1}
  c_{A,0}(x,t)=\zg c_0^{in}+(1-\zg )\(c_{+}^{out}
  \chi_+ +c_{-}^{out}\chi_- \)~\text{with}~c_0^{in}(\rho):=\theta_0(\rho)
\end{equation}
and $c_{\pm}^{out}=\pm 1$, $\chi_\pm=\chi_{\O^\pm(t)}(x)$,
\begin{equation}\label{yuning:1.40}
  \begin{split}
    \rho(x,t)  &:= \frac{d_{\G}(x,t)}\eps -h_\eps (S(x,t),t),\\
  h_\eps(s,t)&:=h_1(s,t)+\eps h_2(s,t),
  \end{split}
\end{equation}
  and  $\zeta$ is the cutoff function defined by \eqref{yuning:1.34}. Let  $X_T$ be the function space defined by  \eqref{yuning:1.31}. Then we assume that $h_1,h_2=h_{2,\eps}$ satisfy
    \begin{equation}\label{eq:h2epsBound}
    h_1\in C^\infty(\T^1\times [0,T_0]),\qquad  \sup_{0<\eps\leq \eps_0}\|h_{2,\eps}\|_{X_{T_\eps}}\leq M,
    \end{equation}
         for some $\eps_0\in (0,1)$, $M\geq 1$, $T_\eps\in (0,T_0]$ for all $\eps \in (0,\eps_0]$.
Note that $h_1$ is independent of $\eps$ but $h_2=h_{2,\eps}$ does.

In the following we construct approximate solutions $(\ve_A, p_A)$ to   $(\tilde{\ve}_A,\tilde{p}_A)$  as
\begin{equation}\label{yuning:chap2glue1}
  \begin{split}
     \ve_A(x,t)&:=\zg \ve_A^{in}(\rho,x,t)+(1-\zg )\left(\ve_A^{+}(x,t)
  \chi_+ +\ve_A^-(x,t)\chi_-\right),\\
  p_A(x,t)&:=\zg p_A^{in}(\rho,x,t)+(1-\zg )
  \left(p_A^{+}(x,t)
  \chi_+ +p_A^{-}(x,t)\chi_- \right),\\
  \end{split}
\end{equation}
where $\rho$ is understood via \eqref{yuning:1.40}
 and $\ve_A^\pm, p_A^\pm, \ve_A^{in}, p_A^{in}$  are defined by
 \begin{equation}\label{yuning:outer1}
  \begin{split}
     \ve_A^\pm(x,t)&:=\ve^\pm_0( x,t)+\eps \ve_1^\pm(x,t)+\eps^2\ve^\pm_2(x,t),\\
p_A^\pm(x,t)&:=p^\pm_0(x,t)+\eps p^\pm_1(x,t),
  \end{split}
\end{equation}
and
\begin{equation}\label{yuning:innerd}
  \begin{split}
     \ve_A^{in}(\rho,x,t)&:=\ve_0(\rho,x,t)+\eps \ve_1(\rho,x,t)+\eps^2\ve_2(\rho,x,t),\\
p_A^{in}(\rho,x,t)&:=\eps^{-1}p_{-1}(\rho,x,t)+p_0(\rho,x,t)+\eps p_1(\rho,x,t).
  \end{split}
\end{equation}
After determining $\ve_j^\pm \colon \Omega^\pm(t)\to\R^2$ and $p_j^\pm\colon \Omega^\pm(t) \to \R$ in the sequel, these functions will be extended to smooth functions (denoted again by $\ve_j^\pm$ and $p_j^\pm$) on $\Om\times [0,T_0]$ such that $\Div \ve_j^\pm=0$ in $\Om\times [0,T_0]$ \Second{and $\ve_j^\pm|_{\partial\Omega}=0$}. These extensions can e.g. be obtained by using the extension operator in \cite[Chapter VI, \S 3]{Stein:SingInt} and correcting the divergence of the extension to obtain $\Div \ve_j^\pm=0$ with the aid of the Bogovskii operator, cf.~e.g.~\cite[Chapter III, Theorem~3.2]{Galdi}.

Here $\ve^\pm_j, p^\pm_j, \ve_j$, and $p_j$ are defined as follows: $\ve_0^\pm:= \ve|_{\Omega^\pm(t)}, p_0^\pm := p|_{\Omega^\pm(t)}$, where $(\ve,p)$ is the smooth solution of \eqref{eq:Limit1}-\eqref{eq:Limit4}. Hence $(\ve_0^\pm,p^\pm_0)$ solve
  \begin{subequations}\label{yuning:firstlimit}
  \begin{align}
    -\Delta \ve_0^\pm +\nabla p_0^\pm &= 0 &\qquad &\text{in }\Omega^\pm (t), t\in (0,T_0),\\
    \Div \ve_0^\pm &= 0 &\qquad &\text{in }\Omega^\pm (t), t\in (0,T_0),\\
    [2D  \ve_0^{{\pm}}-p_0^{{\pm}}\mathbf{I}]\no_{\Gamma_t} &= -\sigma H\no_{\Gamma_t} && \text{on }\Gamma_t, t\in (0,T_0),\\
     [\ve_0^{{\pm}}]&=0 && \text{on }\Gamma_t, t\in (0,T_0),\\
\Second{     \ve_0^{-}|_{\partial\Omega}} &\Second{=0} &&\Second{\text{on }\partial\Omega\times (0,T_0),}
  \end{align}
\end{subequations}
where we employed the notation
  \begin{equation}\label{yuning:1.59}
    [f^\pm]= f^+|_{\Gamma_t}-f^-|_{\Gamma_t} \qquad \text{on }\Gamma_t.
  \end{equation}

In order to determine $\ve_j, p_j$ satisfying suitable matching conditions for all $(x,t)\in \Gamma(3\delta)$, we use the ansatz
\begin{equation}\label{yuning:1.41}
  \begin{split}
    \ve_j(\rho,x,t)&= \tilde{\ve}_j(\rho,x,t) + \eta(\rho)d_{\G }(x,t) \hat{\ve}_j(x,t), \qquad j=0,1,2,\\
  p_j(\rho,x,t)&= \tilde{p}_j(\rho,x,t) + \eta(\rho)d_{\G }(x,t) \hat{p}_j(x,t), \qquad j=-1,0,1,
  \end{split}
\end{equation}
where $\eta(\rho):= -1+\frac{2}\sigma\int_{-\infty}^\rho \theta_0'(s)^2 \sd s$ for all $\rho\in\R$. Then $\eta\colon \R\to \R$ is a smooth, non-decreasing, and odd function such that
\begin{equation}  \label{eq:eta}
  |\eta(\rho)\mp 1|\leq Ce^{-\alpha|\rho|} \quad \text{if }\rho \gtrless 0
\end{equation}
for some $\alpha>0$.
Now we define for all $(\rho,x,t)\in \R\times \Gamma(3\delta)$
\begin{subequations}\label{yuning:summery1}
  \begin{align}
     \tilde{p}_{-1}(\rho,x,t)&=  p_{-1}(\rho,x,t)=-(\theta_0'(\rho))^2,~\hat{p}_{-1}\equiv 0 ,~ p^\pm_{-1}\equiv 0,\\
    \tilde{\ve}_0(\rho,x,t) &= \tfrac12 (\ve_0^+(x,t) +\ve_0^-(x,t)),\quad   \hat{\ve}_0(x,t) = \tfrac1{2d_{\G}} (\ve_0^+(x,t) -\ve_0^-(x,t)),\\
  \tilde{p}_0(\rho,x,t) &= \tfrac 12(p_0^+(x,t)+p_0^-(x,t)) -\tfrac \sigma 2\Delta d_\G(x,t) \eta(\rho),\\
    \hat{p}_0(x,t) &= \tfrac1{2d_\G}\(p_0^+(x,t)-p_0^-(x,t)  +\sigma \Delta d_{\G}(x,t)\),
  \end{align}
\end{subequations}
where $(\ve_1^\pm,p_1^\pm)$ is the solution of the linear two-phase Stokes system
\begin{subequations}\label{yuning:secondlimit}
  \begin{align}
    -\Delta \ve_1^\pm +\nabla p_1^\pm &= 0 &&\text{in }\Omega^\pm (t), t\in (0,T_0),\\
    \Div \ve_1^\pm &= 0 &&\text{in }\Omega^\pm (t), t\in (0,T_0),\\\nonumber
      [2D \ve_1^{\pm}-p_1^{\pm}\mathbf{I}]\no_{\Gamma_t}&=2 X_0^{\ast,-1}(h_1)(\no_{\Gamma_t}\hat{p}_0-2\p_\no\hat{\ve}_0)\\
&\quad -\sigma X_0^{\ast,-1}(\Delta_\Gamma h_1)\no_{\Gamma_t}  && \text{on }\Gamma_t, t\in (0,T_0),\\
    [\ve_1^{\pm}]&=0 && \text{on }\Gamma_t, t\in (0,T_0),\\
  \Second{   \ve_1^-}&\Second{=0} &&\Second{\text{on }\p\Omega}.
  \end{align}
\end{subequations}

Here and in the following,  for a function $f$ that vanishes on $\Gamma$, $\frac1{d_\Gamma}f(x,t)$ is understood as $$\lim_{d_\Gamma\to 0} \frac1{d_\Gamma}f(x,t)=\partial_{\no}f(x,t),\quad \forall (x,t)\in \G.$$
Moreover, we define
\begin{equation}\label{eq:DefnV1}
  \tilde{\ve}_1(\rho,x,t)=\tfrac 12(\ve^+_1(x,t)+\ve^-_1(x,t)),\quad
  \hat{\ve}_1(x,t)=\tfrac 1{2d_{\G}}(\ve^+_1(x,t)-\ve^-_1(x,t)).
\end{equation}
Furthermore $(\ve_2,p_1)$ and $\ve^\pm_2$ are determined by
\begin{subequations}\label{yuning:summery3}
  \begin{align}
    \ve_{2,\no}^-&\equiv0~\text{on}~\Om\times [0,T_0],\\
 \ve_{2,\no}^+(x,t)&=-\p_\no\hat{\ve}_{0,\no}(x,t)
  \int_{\R}(z+h_1(S(x,t),t))^2\eta'(z)\sd z,\\
   \tilde{\ve}_{2,\no}(\rho,x,t)&=\ve^-_{2,\no}(x,t)
   +d_{\G}\hat{\ve}_{2,\no}(x,t)
  -\tfrac{\hat{\ve}_{0,\no}(x,t)}{d_{\G}}\int_{-\infty}^\rho(z+h_1(S(x,t),t))^2\eta'(z)\sd z,\\
  \hat{\ve}_{2,\no}(x,t)&=\tfrac 1{2d_{\G}}\(\tfrac{\hat{\ve}_{0,\no}(x,t)}{d_{\G}}\int_{\mathbb{R}}
  (z+h_1(S(x,t),t))^2\eta'(z)\sd z+\ve_{2,\no}^+(x,t)-\ve_{2,\no}
  ^-(x,t)\)
  \end{align}
\end{subequations}
for all $(x,t)\in \Gamma(3\delta)$, $\rho \in\R$,
and 
\begin{subequations}\label{yuning:summery2}
  \begin{align}\label{yuning:summery2a}
   &\tilde{\ve}_{2,\btau}(\rho,x,t)=- \p_\no\hat{\ve}_{0,\btau}(P_{\G_t}(x),t)
   \int_{-\infty}^\rho\int_{-\infty}^y\((z^2-h_1^2 )\eta''(z)+4z\eta'(z)\)\sd z\sd y,\\
   &\hat{\ve}_{2,\btau}(x,t)=0,\quad \ve_{2,\btau}(\rho,x,t)=\tilde{\ve}_{2,\btau}(\rho,x,t),\quad \ve_{2,\btau}^\pm(x,t)=\lim_{\rho \to \pm\infty} \tilde{\ve}_{2,\btau}(\rho,x,t),
  \end{align}
\end{subequations}
where $h_1$ is understood as  a function of $(x,t)$ via  $h_1=h_1(S(x,t),t)$.  Moreover, $\ve_2^\pm$ is extended smoothly to $\Om\times [0,T_0]$ with $\ve_{2,\btau}^-|_{\partial\Omega}=0$.
Finally we define the pressure terms by
\begin{subequations}\label{yuning:1.61}
\begin{align}
&\tilde{p}_1 (\rho,x,t) =\tfrac12\left( p_1^+(x,t)+p_1^-(x,t)-\int_{\mathbb{R}}  a_1( z,x,t)\sd  z\right) + \int_{-\infty}^\rho a_1( z,x,t)\sd  z,\\
   &\hat{p}_1(x,t) = \tfrac1{2d_{\G}}\left( p_1^+(x,t)-p_1^-(x,t) - \int_{\mathbb{R}}  a_1( z,x,t)\sd  z\right),\\
   &p_1(\rho,x,t) = \tfrac{1+\eta(\rho)}{2}\left( p_1^+(x,t)-p_1^-(x,t) - \int_{\mathbb{R}}  a_1( z,x,t)\sd  z\right)+ \int_{-\infty}^\rho a_1( z,x,t)\sd  z,
\end{align}
\end{subequations}
where   $a_1$ is defined for all $(\rho,x,t)\in \R\times\Gamma(3\delta)$ by
\begin{equation} \label{yuning:defa1'}
  \begin{split}
     &a_1(\rho,x,t)\\
     =&
\p_\rho^2 \tilde{\ve}_{2,\no}(\rho,x,t)-\eta''(\rho)(h_2\hat{\ve}_{0,\no}
(x,t)+(\rho+h_1)
\hat{\ve}_{1,\no}(x,t))-\eta'(\rho)\hat{p}_0(\rho+h_1) \\&-2\eta'(\rho)(\nabla^\Gamma h_1\cdot\nabla_x(d_{\G} \hat{\ve}_0(x,t)))\cdot \no_{\Gamma_t} \\
     &+2\eta'(\rho)\no_{\Gamma_t}\cdot\nabla_x(d_{\G}\hat{\ve}_{1,\no}(x,t)) +(h_1 +\rho)\hat{\ve}_{0,\no}(x,t)\eta'(\rho)\Delta_x d_{\G} \\
     &{+\tfrac{\hat{\ve}_{0,\no}(x,t)}{d_{\G}}\p_\rho
     \((\rho+h_1)^2\eta'(\rho)\)
     -\hat{\ve}_{0,\no}(x,t)h_2\eta''(\rho)+2(\rho+h_1)
     \p_\no\hat{\ve}_{0,\no}(x,t)\eta'(\rho)}\\
     &-\p_\rho(\theta_0'(\rho))^2|\nabla^\Gamma h_1 |^2+(\theta_0'(\rho))^2\Delta^\Gamma h_1.
  \end{split}
\end{equation}
In the above formula, $h_2$ should be interpreted as a function of $(x,t)$ via $h_2=h_2(S(x,t),t)$. The motivation for defining $a_1$ can be seen  at \eqref{yuning:defa1} in the appendix.

\begin{rem}
  From the construction above one observes that $\ve_0$ and $p_0$ only depend on the solutions $(\ve_0^\pm, p_0^\pm)$ of the limit sharp interface problem \eqref{yuning:firstlimit} or equivalently \eqref{eq:Limit1}-\eqref{eq:Limit4}. Moreover, $\ve_1$ and $\ve_2$   depend only on the choice of $h_1$ and $p_1$  depends on   $ h_1,h_{2,\eps} $.
\end{rem}
\begin{lem}\label{yuning:lip4}
Under the assumption \eqref{eq:h2epsBound}, the functions $(\ve_0^\pm,p^\pm_0)$, $(\ve_1^\pm,p^\pm_1)$, $(\ve_{2,\btau}^\pm,\ve_{2,\no}^\pm)$ and
\begin{equation*}
\{\tilde{p}_{-1},\hat{p}_{-1},\tilde{p}_{0},\hat{p}_{0}, \tilde{\ve}_0,\hat{\ve}_0,\tilde{\ve}_1,\hat{\ve}_1, \tilde{\ve}_2,\hat{\ve}_2\}
\end{equation*}
can be defined through   \eqref{yuning:firstlimit}, \eqref{yuning:summery1}, \eqref{yuning:secondlimit}, and \eqref{yuning:summery2} and they are all smooth.  The functions $\tilde{p}_1, \hat{p}_1$ are  continuous with respect to $(\rho,x,t)\in \R\times \Gamma(3\delta)$ and  continuously differentiable with respect to $\rho\in \R, x\in \Gamma_t(3\delta)$ for almost every $t\in (0,T_\eps)$.
 \end{lem}
 \begin{proof}
 We first show that these definitions do not leads to circular reasoning: First of all $\Gamma$ and $(\ve_0^\pm,p_0^\pm)$ are given. Hence all the functions in \eqref{yuning:summery1} as well as $\tilde{\ve}_{2,\btau}$ are well defined. Moreover, $(\ve_1^\pm, p_1^\pm)$ are determined by solving \eqref{yuning:secondlimit} for given $h_1$, which can be done because of Lemma~\ref{lem:Stokes} and Corollary \ref{cor:ParabolicEqOnSurfaceStokes}. Hence $\tilde{\ve}_1,\hat{\ve}_1$ and thus $\ve_1$ are well defined as well.
 With all the previous information, the first three formulae in \eqref{yuning:summery2} make sense and thus also  \eqref{yuning:defa1'}. Finally $\tilde{p}_1,\hat{p}_1$ and $p_1$ can be defined via \eqref{yuning:summery2}. 

In order to show that these functions are continuously differentiable/smooth, we need the following facts: if $f(x)\in C^\ell(\overline{\Omega^\pm})$ and vanishes on $\Gamma_t$, then we can redefine $\frac{f(x)}{d_\Gamma}$ on $\Gamma_t$ by $\lim_{d_\Gamma\to 0}\frac{f(x)-f(P_{\G_t}(x) )}{d_\Gamma}$ and $\frac{f(x)}{d_\Gamma}\in C^{\ell-1}(\overline{\Omega^\pm})$. 
       One can verify that all the functions defined in \eqref{yuning:summery1} are smooth, using \eqref{yuning:firstlimit}.          $\tfrac{\hat{\ve}_{0,\no}}{d_\Gamma}$ is smooth because $\hat{\ve}_{0,\no}$ vanishes on the interface due to the divergence-free condition of $\ve_0^\pm$ and $[\ve_0^\pm]:=\ve_0^+\mid_{\G_t}-\ve_0^-\mid_{\G_t}=0$.

 Now we consider  $\hat{\ve}_{2}$ and $\tilde{\ve}_{2}$. We substitute the formula for $\ve^+_{2,\no}$ in \eqref{yuning:summery3}  into that of $\hat{\ve}_{2,\no}$ and deduce that,  for all $x\in \Om\backslash\Gamma_t, t\in [0,T_0]$
\begin{equation*}
  \hat{\ve}_{2,\no}(x,t)=\tfrac 1{2d_\Gamma}\left[\(\tfrac{\hat{\ve}_{0,\no}(x,t)}{d_{\G}}
  -\p_\no\hat{\ve}_{0,\no}(P_{\G_t}(x),t)\)
  \int_{\mathbb{R}}(z+h_1)^2\eta'(z)d z+\ve_{2,\no}^+(x)
   -\ve_{2,\no}^+(P_{\G_t}(x),t) \right]
\end{equation*}
and the expressions in the rectangle bracket vanishes on $\Gamma_t$,
which implies  that $\hat{\ve}_{2,\no}$ as well as $\tilde{\ve}_{2,\no}$ are smooth. On the other hand, it can be verified that the double integral defining \eqref{yuning:summery2a} is well-defined as the integral $\int_{-\infty}^yr\eta'(r)dr$ vanishes when $y\leq -1$  because of \eqref{eq:eta}.
 So $\tilde{\ve}_{2,\btau}$ is smooth in the tangent direction and is constant on the normal direction.

 The conclusion for   $\hat{p}_1$ follows from \eqref{yuning:1.60} together with  the following formula
  \begin{equation*}
    \begin{split}
      \hat{p}_1(x,t) &= \tfrac1{2d_\Gamma}\Big(p_1^+(x,t)-p_1^-(x,t)-\int_{\R}a_1(\rho,x,t)\sd \rho \\
      & \underbrace{-p_1^+(P_{\G_t}(x),t))+p_1^-(P_{\G_t}(x),t)
      +\int_{\R}a_1(\rho,P_{\G_t}(x),t)\sd \rho}_{=0}\Big),
    \end{split}
  \end{equation*}
   where $a_1$ is smooth in normal direction and  continuous in all directions since it   depends on smooth quantities and the push-forward of $h_{2,\eps}\in X_{T_\eps}\hookrightarrow BUC([0,T_\eps]; H^{\frac32}(\T^1))\cap L^4(0,T_\eps;C^1(\T^1))$, $H^{\frac32}(\T^1)\hookrightarrow C^0(\T^1)$.
 \end{proof}

 Now we construct the outer expansion in \eqref{yuning:outer1}.
The following statement can be readily verified:
\begin{lem}\label{lem:ExpDecay}
There exists some $\alpha>0$ and $C>0$ such that for all $\rho\in\R$ and $m,\ell=0,1,2$, $n=0,1$,
  \begin{equation}\label{yuning:matching4}
    \begin{split}
    \sup_{(x,t)\in \Gamma(3\delta)}|  D_x^m\partial_t^n\partial_\rho^\ell(\ve_i(\pm\rho,x,t)-\ve^\pm_i(x,t))|&\leq Ce^{-\alpha |\rho|}\text{ for all }0\leq i\leq 2,\\
     \sup_{(x,t)\in  \Gamma(3\delta) }| \partial_\rho^\ell(p_i(\pm\rho,x,t)-p^\pm_i(x,t))|&\leq Ce^{-\alpha |\rho| } \text{ for all }-1\leq i\leq 1.
    \end{split}
  \end{equation}
\end{lem}

The following lemma is crucial in the inner expansion and its proof is given  in Section~\ref{subsec:ProofOfLemma} below.
\begin{lem}\label{yuning:app1}
If we define $(\ve_A^{in}, p_A^{in})$ through \eqref{yuning:innerd},
then
\begin{align}\nonumber
  -\Delta \ve_A^{in} +\nabla p_A^{in}= &-\eps \Div (\nabla c^{in}_0\otimes \nabla c^{in}_0)+ r_\eps\left(\tfrac{d_\Gamma}\eps - h_\eps,x,t \right) + \tilde{r}_\eps\left(\tfrac{d_\Gamma}\eps - h_\eps,x,t \right)\nonumber\\\label{eq:ApproxStokesEq1}
&+\sum_{ i' \leq 2;0\leq i,j,j'\leq 1} \eps^2   R^{ i' j' i j }_\eps (x,t)  (\partial_s^{j'} h_{2})^j (\partial_s^{i'} h_2 )^i\nonumber\\
    &+\sum_{0\leq i,i',j,j',k,k'\leq 1}  \eps^2  \tilde{R}^{ i' j' k' i j k }_\eps (x,t)  (\partial_s^{j'} h_{2})^j (\partial_s^{i'} h_{2} )^i(\partial_s^{k'} h_{2} )^k\\
\label{eq:ApproxStokesEq2}
\Div \ve_A^{in} =&\,  \sum_{0\leq i,i'\leq 1}g^{{i'}i }_\eps\left(\tfrac{d_\Gamma}\eps - h_\eps,x,t \right)(\partial_s^{i'} h_2)^i +\eps^2 \tilde{g}_\eps (x,t)
\end{align}
in $\Gamma_t(3\delta)$
where $  h_\eps$ is defined through \eqref{eq:h2epsBound} and
 \begin{equation}\label{yuning:class1}
      (r_\eps)_{0<\eps <1}\in \mathcal{R}_{0,\alpha}^0,\quad (\tilde{r}_\eps)_{0<\eps <1}\in \mathcal{R}_{1,\alpha},\quad
  (g^{ i' i}_\eps)_{0<\eps <1}\in \mathcal{R}^0_{1,\alpha}
 \end{equation}
  for some $\alpha>0$  and $
   R^{i' j' i j }_\eps, \tilde{R}^{i' j' k' i j k }_\eps, \tilde{g}_\eps
$  are uniformly bounded with respect to $\eps\in (0,1], (x,t)\in \Gamma(3\delta)$.
\end{lem}

The main result of this section is:
\begin{theorem}\label{thm:ApproxVA}
Under the assumptions  \eqref{eq:h2epsBound}, there is some $C(M)>0$, independent of $(T_\eps)_{0<\eps\leq \eps_0}$ and $\eps_0\in (0,1]$, such that
  \begin{equation*}
\|\ve_A-\tilde{\ve}_A\|_{L^2(0,T;H^1(\Omega))}\leq C(M)(T^{\frac14}+ \eps^{\frac12})\eps^2
  \end{equation*}
for all $0<T\leq T_\eps$ and $\eps \in (0,\eps_0]$.
\end{theorem}
\begin{proof}
We will use that
\begin{equation}\label{eq:dGammaEstim}
  \left|\frac{d_\Gamma(x,t)}\eps - h_\eps (S(x,t),t) \right|\geq \frac{\delta}{2\eps}\quad \text{for all }(x,t)\in \Gamma(3\delta)\setminus \Gamma(\delta),  t\leq T_\eps
\end{equation}
if $0<\eps\leq \min(\eps_1,\eps_0)$ for sufficiently small $\eps_1\in (0,1)$ in dependence of $M$. Moreover, we will choose $\eps_1$ such that $\eps_1M\leq 1$.
  First of all, because of \eqref{yuning:matching4} and \eqref{eq:EmbeddingE1}, we have for $1\leq i' \leq 2, 0\leq  j' \leq 1$,
  \begin{equation}\label{yuning:chap2glue2}
  \begin{split}
    &\|\p_x^{i'} (\zg) \p_x^{j'}( \ve_A^{in}-\ve_A^{+}
  \chi_+-\ve_A^{-}\chi_-)\|_{L^\infty(0,T_\eps; L^2(\Om))}\leq C e^{-\frac{\alpha\delta}{2\eps}},\\
  &\|\p_x^{i'} (\zg) (  p_A^{in}-p_A^{+}
  \chi_+-p_A^{-}\chi_-)\|_{L^\infty(\Om\times (0,T_\eps))}\leq C e^{-\frac{\alpha\delta}{2\eps}}.
  \end{split}
\end{equation}
So it follows from \eqref{yuning:chap2glue1} that
  \begin{align*}
    \Delta {\ve}_A =& (\zg )\Delta \ve_A^{in}+(1-\zg )\left(\Delta \ve_A^{+}
  \chi_++\Delta\ve_A^{-}\chi_-\right)\\
 &+ 2\nabla (\zg) \cdot \nabla \left(\ve_A^{in}-\ve_A^{+}
  \chi_+-\ve_A^{-}\chi_-\right)
+ \Delta (\zg)  \left(\ve_A^{in}-\ve_A^{+}
  \chi_+-\ve_A^{-}\chi_-\right),\\
    \nabla {p}_A =&( \zg )\nabla p_A^{in}+(1-\zg )\left(\nabla p_A^{+}
  \chi_++\nabla p_A^{-}\chi_-\right)+ \nabla (\zg) \left(p_A^{in}-p_A^{+}
  \chi_+-p_A^{-}\chi_-\right).
  \end{align*}
 On the other hand, it follows from \eqref{yuning:outer1} and the first equation of \eqref{yuning:firstlimit} and \eqref{yuning:secondlimit} that
 \begin{equation*}
   -\Delta \ve_A^{+}\chi_++\nabla p_A^{+}\chi_+-\Delta \ve_A^{-}\chi_-+\nabla p_A^{-}\chi_-=O(\eps^2)\quad \text{in }L^\infty(\Om\times(0,T_\eps)).
 \end{equation*}
  The above two formulas together with \eqref{train:stokes}  imply
  \begin{alignat}{1}\nonumber
          -&\Delta(\ve_A-\tilde{\ve}_A)+\nabla (p_A-\tilde{p}_A)\\\nonumber
      =&\zg (-\Delta \ve_A^{in}+\nabla p_A^{in})+\eps\Div (\nabla c_{A,0}\otimes\nabla c_{A,0})\\\nonumber
      &-(1-\zg )(-\Delta \ve_A^{+}\chi_++\nabla p_A^{+}\chi_+-\Delta \ve_A^{-}\chi_-+\nabla p_A^{-}\chi_-)\\\nonumber
&+\eps(1-\zg )\Div (\nabla c_{A,0}\otimes\nabla c_{A,0})+O(e^{-\frac{\alpha\delta}{2\eps}})\\\label{eq:ApproxStokesEq1new}
      =&\zg (-\Delta \ve_A^{in}+\nabla p_A^{in})+\eps\Div (\nabla c_{A,0}\otimes\nabla c_{A,0})  +O(\eps^2)~\text{in}~L^\infty(0,T_\eps;L^2(\Om)).
      \end{alignat}

  To expand the concentration term, we need the following formula which can be easily derived from \eqref{yuning:ca1} together with $\nabla c^{out}_{0,\pm}=0$:
  \begin{equation}\label{yuning:1.51}
\nabla c_{A,0}=\nabla (\zg) (c^{in}_0-c_{0,+}^{out}
  \chi_+-c_{0,-}^{out}\chi_-)+\zg \nabla c^{in}_0.
  \end{equation}
  Applying the above formula to the second term on the right hand side of \eqref{eq:ApproxStokesEq1new} and extract the leading terms, we get
    \begin{alignat}{1}\nonumber
      &-\Delta(\ve_A-\tilde{\ve}_A)+\nabla (p_A-\tilde{p}_A)\\ \label{yuning:app2}
      =&\zg \(-\Delta \ve_A^{in}+\nabla p_A^{in}+\eps\Div (\nabla c^{in}_0\otimes\nabla c^{in}_0)\)+O(\eps^2)~\text{in}~L^\infty(0,T_\eps;L^2(\Om)).
  \end{alignat}
  Next we multiply \eqref{eq:ApproxStokesEq1} by an arbitrary $\varphi\in C_{0,\sigma}^\infty(\Omega)$ and we employ Lemma \ref{yuning:app1} to estimate the first part on the right hand side of \eqref{yuning:app2}:
  \begin{align*}
    &\left|\int_{\Omega}\zg (-\Delta \ve_A^{in}+\nabla p_A^{in}+\eps\Div (\nabla c^{in}_0\otimes\nabla c^{in}_0))\cdot\varphi \sd x\right| \\
    &\leq \int_{\Gamma_t(2\delta)}\left| r_\eps\left(\tfrac{d_\Gamma}\eps - h_\eps,x,t \right)\cdot \varphi\right|\sd x + \int_{\Gamma_t(2\delta)}  \left|\tilde{r}_\eps\left(\tfrac{d_\Gamma}\eps - h_\eps,x,t \right)\cdot \varphi\right|\sd x\\
&\ \ +\int_{\Gamma_t(2\delta)}\left|\sum_{i'\leq 2;  0\leq i,j,j'\leq 1} \eps^2  R_\eps^{ i' j' i j } (x,t ) (\partial_s^{j'} h_2)^j (\partial_s^{i'} h_2 )^i \varphi\right|\sd x\\
&\ \ +\int_{\Gamma_t(2\delta)}\left|\sum_{ 0\leq i,j,k,i',j',k'\leq 1}\eps^2  \tilde{R}^{ i' j' k' i j k }_\eps (x,t) (\partial_s^{j'} h_2)^j (\partial_s^{i'} h_2 )^i(\partial_s^{k'} h_2 )^k\varphi\right|\sd x\\
&\ \ + C\eps^2\|\varphi\|_{L^2(\Omega)}\qquad  \text{for all }t\in (0,T_\eps).
  \end{align*}
  Moreover, we have the embeddings $H^{3/2}({\T^1})\hookrightarrow W^1_p({\T^1})$ for every $1\leq p<\infty$  and the estimate
  $$\(\int_{|r|\leq 2\delta}\operatorname{ess\, sup}_{s\in\T^1}|\varphi(X_0(s,t)+r\no(s,t))|^4\sd r\)^{\frac 14}\leq C\|\varphi\|_{H^1(\Gamma_t(2\delta))}$$
 for some $C>0$ independent of $\varphi$. This estimate is proved in the same way as  \eqref{yuning:2.01} with the role of $r\in (-2\delta,2\delta)$ and $s\in \T^1$ interchanged.
These together with \eqref{eq:EmbeddingE1}, \eqref{eq:h2epsBound}, \eqref{yuning:class1} and the first inequality in Corollary~\ref{cor:repsEstim2}  yields
  \begin{equation}\label{yuning:app3}
  \begin{split}
    &\left|\int_{\Omega}\zg (-\Delta \ve_A^{in}+\nabla p_A^{in}+\eps\Div (\nabla c^{in}_0\otimes\nabla c^{in}_0)) \cdot\varphi\sd x\right| \\
    &\leq C(M)\eps^2\(1+ \|h_2(t)\|_{H^2({\T^1})} \|h_2(t)\|_{ H^1({\T^1}) }+ \|h_2(t)\|_{W^1_3({\T^1})}^3\)\|\varphi\|_{H^1(\Omega)}\\
    &
    \leq C(M)\eps^2\(1+\|h_2(t)\|_{H^2({\T^1})}\)\|\varphi\|_{H^1(\Omega)}.
  \end{split}
\end{equation}
Combining \eqref{yuning:app2} and \eqref{yuning:app3} yields
\begin{alignat}{1}\label{yuning:class2}
  \left|\int_{\Omega}\nabla(\ve_A-\tilde{\ve}_A):\nabla\varphi \sd x\right|
 &=\left| \int_{\Omega}\(-\Delta(\ve_A-\tilde{\ve}_A)+\nabla (p_A-\tilde{p}_A)\)\cdot\varphi \sd x\right|\\\nonumber
 &\leq  C(M)\eps^2\(1+\|h_2(t)\|_{H^{ 2}({\T^1})}\)\|\varphi\|_{H^1(\Omega)}\quad \text{for all }t\in (0,T_\eps),
\end{alignat}
where $h_2\in X_{T_\eps}\hookrightarrow L^4(0,T_\eps;H^2(\T^1))$ is bounded.
 On the other hand, using  \eqref{yuning:outer1} together with $\Div \ve_j^\pm=0$ for $j=0,1$, we obtain
 \begin{equation*}
   \begin{split}
      \Div \ve_A& =\Div\(\zg \ve_A^{in}+(1-\zg )\left(\ve_A^{+}
  \chi_++\ve_A^{-}\chi_-\right)\)\\
  &=\nabla (\zg) \cdot\(\ve_A^{in}- \ve_A^{+}
  \chi_+-\ve_A^{-}\chi_-\)+(\zg) \Div \ve_A^{in}+O(\eps^2)~ \text{in}~L^\infty(\Omega\times [0,T_\eps]).\\
   \end{split}
 \end{equation*}
So we can apply \eqref{eq:ApproxStokesEq2} to the above estimate and then use the second inequality in Corollary~\ref{cor:repsEstim2} as well as \eqref{yuning:chap2glue2} :
\begin{equation}\label{yuning:class3}
  \|\Div (\tilde{\ve}_A -\ve_A)\|_{{L^2(\Omega\times (0,T_\eps))}}\leq C\eps^2(T^{\frac12}+\eps^{\frac12})(\|h_2\|_{L^2(0,{T_\eps};H^{5/2}({\T^1}))}+1).
\end{equation}
Combining \eqref{yuning:class2} and \eqref{yuning:class3} together with  standard results on the weak solutions of the Stokes equation with general divergence leads to 
\begin{equation*}
  \|\tilde{\ve}_A -\ve_A\|_{L^2(0,T_\eps;H^1)}\leq C(M)\eps^2(T^{ \frac14 }+\eps^{\frac12}).
\end{equation*}
This implies the desired result
due to \eqref{eq:h2epsBound}.
\end{proof}

\subsection{The Leading Error in the Velocity}\label{subsec:LeadingErrorVelocity}

Within this subsection we denote $u_1=c_\eps-c_{A,0}$, where $c_{A,0}$ is defined by \eqref{yuning:ca1} for given $h_1,h_{2,\eps}=h_2$ satisfying
\eqref{eq:h2epsBound} for some $\eps_0\in (0,1)$ and  $M\geq 1$. Moreover, we will often write $\no$ instead of $\no_\Gamma$ for simplicity.

In the following
$\tilde{\we}_1, \tilde{\we}_2 \colon \Om\times [0,T_0]\to \R^2$, $q_1, q_2\colon \Om\times [0,T_0]\to \R$ are the solutions of
\begin{alignat}{2}\label{eq:w1a'}
      -\Delta \tilde{\we}_1 +\nabla q_1 &= -\eps \Div ((\nabla c_{A,0}-\g)\otimes \nabla u_1)-\eps \Div (\nabla {u}_1\otimes (\nabla c_{A,0}-\g)) &\quad &\text{in }\Om,\\\label{eq:w1b'}
 \Div \tilde{\we}_1 &=0&\quad &\text{in }\Om,\\
      -\Delta \tilde{\we}_2 +\nabla q_2 &= -\eps \Div (\g\otimes \nabla u_1)-\eps \Div (\nabla {u}_1\otimes \g)-\eps \Div (\nabla u_1\otimes \nabla u_1)&\quad &\text{in }\Om,\\\label{eq:w1a''}
 \Div \tilde{\we}_2 &=0&\quad &\text{in }\Om,\\\label{eq:wjBCs}
\Second{ \tilde{\we}_1|_{\partial\Om}}&\Second{= \tilde{\we}_2|_{\partial\Om}=0} &&\Second{\text{on }\partial\Om}
\end{alignat}
for every $t\in[0,T_0]$, where
\begin{equation*}
  \g =- \zg\, \theta_0'(\tfrac{d_\Gamma}\eps-h_\eps) \eps\nabla_\btau h_{2,\eps}.
\end{equation*}
Then $\ve_\eps= \tilde{\ve}_A+ \tilde{\we}_1+\tilde{\we}_2$ because of the unique solvability of the Stokes systems.
According to \eqref{yuning:1.51},
\begin{equation}\label{eq:cAg}
  \nabla c_{A,0}-\g = \zg \ \theta_0'(\tfrac{d_\Gamma}\eps-h_\eps)(\tfrac{\nabla d_\Gamma}\eps - \nabla_\btau h_1) + \nabla d_\Gamma \zeta'\circ d_\Gamma\ \(\theta_0(\tfrac{d_\Gamma}\eps-h_\eps)-\chi_++\chi_-\) .
\end{equation}
It will turn out that $\tilde{\we}_1$ is the leading part in the error estimates for $\ve_\eps-\tilde{\ve}_A$ and it plays an important role in the construction of $c_A$.
Before proceeding the estimates, we give a few words about the  Stokes equations above. These systems will be understood in the weak sense. E.g., $\ve=\tilde{\we}_1\in  \Second{H^1_0(\Om)^2\cap L^2_{\sigma}(\Om)}$ solves
\begin{equation}\label{eq:WeakStokes}
  \int_\Om \nabla \ve: \nabla \varphi \sd x = \weight{F,\varphi}\qquad \text{for all }\varphi \in H^1_{0}(\Om)^2\cap L^2_\sigma(\Om),
\end{equation}
where 
 $F\in \Second{(H^1_0(\Om)^2\cap L^2_{\sigma}(\Om))'}$ is defined by
\begin{equation*}
  \weight{F,\varphi}:= \eps \int_\Omega ((\nabla c_{A,0}-\g)\otimes \nabla u_1+ \nabla u_1\otimes (\nabla c_{A,0}-\g) ) : \nabla \varphi \sd x
\end{equation*}
for all $\varphi \in \Second{H^1(\Om)^2\cap L^2_{\sigma}(\Om)}$. Here $L^2_\sigma(\Omega)$ is the closure of divergence free $C^\infty_0(\Omega)$-vector fields in $L^2(\Omega)$.
Existence of weak solutions for every $\eps>0$ is a consequence of the Lax-Milgram Theorem. Moreover, if $\ve\in  \Second{H^1_0(\Om)^2\cap L^2_{\sigma}(\Om)}$ solves \eqref{eq:WeakStokes} for some $F\in \Second{(H^1_0(\Om)^2\cap L^2_{\sigma}(\Om))'}$ and $1<p<\infty$, then there is some $C_p>0$ independent of $\ve$ and $F$ such that
\begin{equation}\label{eq:VeryWeakEstim}
  \|\ve\|_{L^p(\Om)}\leq 
C\sup_{0\neq \varphi\in \Second{W^2_{p'}(\Om)\cap W^1_{p',0}(\Om)\cap L^2_{\sigma}(\Om)}} \frac{|\weight{F,\varphi}|}{\|\varphi\|_{W^2_{p'}(\Om)}}
\end{equation}
where $\frac 1p+\frac 1{p'}=1$.
The latter estimate follows by duality from the well-known fact that the Stokes operator $A_{p'}\colon \Second{W^2_{p'}(\Om)^2\cap W^1_{p',0}(\Omega)^2\cap L^2_{\sigma}(\Om)\to L^{p'}_{\sigma}(\Om)}$ is bijective.
\begin{prop}\label{yuning:velocityest}
  Let $N=2$, $1<q<2$ and $1\leq r \leq 2$.  We assume  that  there are some $R\geq 1$, $\eps_0\in (0,1)$, $T_\eps \in (0,T_0]$  such that
  \begin{subequations}\label{yuning:diffcaca1New}
    \begin{align}
      \|c_\eps -c_{A,0}\|_{L^4(0,T_\eps;L^2(\Om))}+\|\nabla (c_\eps -c_{A,0})\|_{L^2(\Om\times (0,T_\eps)\setminus \Gamma(\delta))} &\leq 2R\eps^{\order+\frac12},\\\label{yuning:diffcaca1New'}
      \|\nabla_\btau(c_\eps -c_{A,0})\|_{L^2(\Om\times (0,T_\eps)\cap \Gamma(2\delta))} + \eps\|\partial_\no(c_\eps -c_{A,0})\|_{L^2(\Om\times (0,T_\eps)\cap \Gamma(2\delta))} &\leq 2R\eps^{\order+\frac12}
    \end{align}
  \end{subequations}
  for all $\eps\in (0,\eps_0]$ and \eqref{eq:h2epsBound} holds true for some $M\geq 1$.
  Then there are some $C(R)>0$, $\eps_1\in (0,1)$, independent of $T_\eps,\eps_0$ such that
  \begin{subequations}
    \begin{align}\label{eq:estofw}
      \|\tilde{\we}_1\|_{L^2(0,T,H^1(\Om))}&\leq C(R)\eps^N,\\\label{eq:estofw'}
      \|\tilde{\we}_1\|_{L^2(0,T,L^2(\Om))}&\leq C(R)\left(T^{\frac14}+\eps^{\frac12}\right)\eps^N,\\  \label{eq:estofwc}
 \|\tilde{\we}_2\|_{L^r(0,T,L^q(\Om))}&\leq C(R)\eps^{\frac{4}r}
    \end{align}
  \end{subequations}
  for all $0<T\leq T_\eps$, $\eps\in (0,\min(\eps_0,\eps_1)]$. Moreover, $C(R)$ is independent of $M$.
\end{prop}
\begin{proof}
Within this proof we will often write $h_\eps$ instead of $h_\eps (S(x,t),t)$ for brevity. We also recall that $u_1=c_\eps-c_{A,0}$. Because of \eqref{eq:dGammaEstim}, we have
\begin{equation*}
\left|\theta_0\left(\tfrac{d_\Gamma(x,t)}\eps - h_\eps (S(x,t),t)\right) - \chi_+(x,t) + \chi_-(x,t)\right| +
  \left|\theta'_0\left(\tfrac{d_\Gamma(x,t)}\eps - h_\eps (S(x,t),t)\right)\right| \leq C e^{- \frac{\alpha\delta}{2\eps}}
\end{equation*}
for all $(x,t)\in \Gamma(3\delta)\setminus \Gamma(\delta)$ and $0 <\eps \leq \min(\eps_0,\eps_1)$ because of \eqref{yuning:decayopti} provided $\eps_1$ is chosen sufficiently small (in dependence on $M$). In particular, we choose $\eps_1\leq \frac1M$ such that $\eps\|h_{2,\eps}\|_{X_{T_\eps}}\leq 1$ independent of the choice of $M$. Here $C$ can be chosen independent of $M$. According to the definition of $\zeta$ at \eqref{yuning:1.34}, all terms involving derivatives of $\zeta\circ d_\Gamma$ will provide terms of order $O(e^{- \frac{\alpha\delta}{2\eps}})$ in the following and will be negligible.

\smallskip

\noindent
\emph{Proof of \eqref{eq:estofw}:} Using identity \eqref{eq:cAg} and  $ \nabla u_1=\nabla_\btau u_1+\no \p_\no u_1$
in $\Gamma(3\delta)$,
we have for all $\psi\in H^1(\Om)^2 \cap L^2_\sigma(\Om)$ with $\|\psi\|_{H^1(\Om)}\leq 1$ that
\begin{alignat}{1}\nonumber
    &\int_{\Om}\eps(\nabla c_{A,0}-\g)\otimes\nabla u_1:\nabla\psi\sd x =\int_{\Gamma_t(2\delta)}\zg   \eps(\nabla c_{A,0}-\g)\otimes\nabla u_1:\nabla\psi\sd x+ O(e^{- \frac{\alpha\delta}{2\eps}})\\\nonumber
  &=\int_{\Gamma_t(2\delta)} \zg \, \theta'_0\left(\tfrac{d_\Gamma}\eps-h_\eps(S(x,t),t)\right)\left(\no-\eps\nabla_\btau h_1(S(x,t),t)\right)\otimes \nabla u_1: \nabla\psi\sd x +O(e^{- \frac{\alpha\delta}{2\eps}})\\\nonumber
  &= {\int_{\Gamma_t(2\delta)} \zg \,  \theta'_0\left(\tfrac{d_\Gamma}\eps-h_\eps\right)\no\otimes \nabla_\btau u_1: \nabla\psi\sd x}  + {\int_{\Gamma_t(2\delta)}\zg \,  \theta'_0\left(\tfrac{d_\Gamma}\eps-h_\eps\right)\no\otimes \no \partial_\no u_1: \nabla\psi\sd x}\\\nonumber
  &\quad {-\int_{\Gamma_t(2\delta)}\zg \,  \theta'_0\left(\tfrac{d_\Gamma}\eps-h_\eps\right)\left(\eps\nabla_\btau h_1 \right)\otimes \nabla u_1: \nabla\psi\sd x}+O(e^{- \frac{\alpha\delta}{2\eps}})\\\label{yuning:1.20}
&\equiv I+II+III + O(e^{- \frac{\alpha\delta}{2\eps}})\qquad \text{in }L^2(0,T_\eps).
\end{alignat}
 Therefore the following estimates are due to Cauchy-Schwarz inequality:
\begin{equation*}
      |I |\leq C\|\nabla_\btau u_1(t)\|_{L^2(\Gamma_t(2\delta))}\|\psi\|_{H^1(\Om)},\quad
      |III |\leq C\eps\|\nabla  u_1(t)\|_{L^2(\Gamma_t(2\delta))} \|\psi\|_{H^1(\Om)}.
\end{equation*}
To estimate $II$, we employ the formula $\no\otimes \no : \nabla\psi  = \partial_\no \psi_\no =-\Div_\btau \psi$, a consequence of the divergence-free condition $\Div \psi=0$. This together with \eqref{yuning:1.23} implies
\begin{alignat*}{1}
  II =& -\int_{\Gamma_t(2\delta)}\zg \,  \theta'_0\left(\tfrac{d_\Gamma}\eps-h_\eps\right) \partial_\no u_1 \Div_\btau \psi \sd x\\
  =& \int_{\Gamma_t(2\delta)}\zg \,  \theta'_0\left(\tfrac{d_\Gamma}\eps-h_\eps\right) (\nabla_\btau \partial_\no  u_1 + {\kappa}\no\partial_\no  u_1) \cdot \psi\sd x\\
& - \int_{\Gamma_t(2\delta)}\zg \,  \theta''_0\left(\tfrac{d_\Gamma}\eps-h_\eps\right) \nabla_\btau h_\eps \partial_\no  u_1 \cdot \psi\sd x\\
  =& \int_{\Gamma_t(2\delta)}\zg \,  \theta'_0\left(\tfrac{d_\Gamma}\eps-h_\eps\right) \partial_\no \nabla_\btau  u_1 \cdot \psi \sd x + \int_{\Gamma_t(2\delta)}\zg \,  \theta'_0\left(\tfrac{d_\Gamma}\eps-h_\eps\right) [\partial_\no, \nabla_\btau]  u_1 \cdot \psi \sd x \\
& + \int_{\Gamma_t(2\delta)}\zg \,  \left(-\theta''_0\left(\tfrac{d_\Gamma}\eps-h_\eps\right)\nabla_\btau h_\eps+ \theta'_0\left(\tfrac{d_\Gamma}\eps-h_\eps\right){\kappa} \no\right) \partial_\no  u_1 \cdot \psi\sd x+O(e^{- \frac{\alpha\delta}{2\eps}})\\
=& -\int_{\Gamma_t(2\delta)}\zg \,  \tfrac1\eps \theta''_0\left(\tfrac{d_\Gamma}\eps-h_\eps\right) \nabla_\btau  u_1 \cdot \psi \sd x -{\int_{\Gamma_t(2\delta)}\zg \,  \theta'_0\left(\tfrac{d_\Gamma}\eps-h_\eps\right) \nabla_\btau  u_1 \cdot \partial_\no \psi \sd x}\\
& + {\int_{\Gamma_t(2\delta)}\zg \,  \theta'_0\left(\tfrac{d_\Gamma}\eps-h_\eps\right) [\partial_\no, \nabla_\btau]  u_1 \cdot \psi \sd x}  \\
& +{\int_{\Gamma_t(2\delta)}\zg \,  \left(-\theta''_0\left(\tfrac{d_\Gamma}\eps-h_\eps\right)\nabla_\btau h_\eps+ \theta'_0\left(\tfrac{d_\Gamma}\eps-h_\eps\right){\kappa} \no\right) \partial_\no  u_1 \cdot \psi\sd x}+O(e^{- \frac{\alpha\delta}{2\eps}})\\
\equiv&\, II_1+II_2+II_3+II_4 +  O(e^{- \frac{\alpha\delta}{2\eps}})\qquad \text{in }L^2(0,T_\eps).
\end{alignat*}
In the above calculation,  $[\cdot,\cdot]$ denotes the commutator of two differential operators. So
it follows from assumption \eqref{eq:h2epsBound} and Lemma \ref{yuning:1.21}   that
\begin{alignat*}{1}
  |II_1|&\leq \tfrac1{\sqrt\eps} \| \theta_0''\|_{L^2(\R)}\|\nabla_\btau u_1\|_{L^2(\Gamma_t(2\delta))}\|\psi\|_{L^{2,\infty}(\Gamma_t(2\delta))}\leq C\tfrac1{\sqrt\eps}\|\nabla_\btau u_1\|_{L^2(\Gamma_t(2\delta))}\|\psi\|_{H^1(\Om)},\\
|II_2| &\leq  C\|\nabla_\btau u_1\|_{L^2(\Gamma_t(2\delta))}\|\psi\|_{H^1(\Om)},\\
|II_3| &\leq C\sqrt\eps\|\theta_0'\|_{L^2(\R)}\|\nabla_\btau u_1\|_{L^2(\Gamma_t(2\delta))}\|\psi\|_{H^1(\Om)}\leq C'\sqrt{\eps}\|\nabla_\btau u_1\|_{L^2(\Gamma_t(2\delta))}\|\psi\|_{H^1(\Om)},\\
  |II_4|&\leq  C\sqrt\eps\left(\| \theta_0''\|_{L^2(\R)}+ \| \theta_0'\|_{L^2(\R)} \right)\|\partial_\no u_1\|_{L^2(\Gamma_t(2\delta))}\|\psi\|_{L^{4,\infty}(\Gamma_t(2\delta))}\left(1+\eps\|\partial_s h_{2,\eps}\|_{L^4(\T^1)}\right)\\
&\leq C\sqrt{\eps}\|\partial_\no u_1\|_{L^2(\Gamma_t(2\delta))}\|\psi\|_{H^1(\Om)},
\end{alignat*}
where we have used \eqref{yuning:1.22}, \eqref{eq:RemainderEstim2} and the following imbedding theorem
$$H^1(\Gamma_t(2\delta))\hookrightarrow L^{4,\infty}(\Gamma_t(2\delta)),\quad H^{\frac32}(\T^1)\hookrightarrow W^1_4(\T^1).$$ Similarly, we obtain
\begin{alignat*}{1}
  &\int_{\Om}\eps\nabla u_1\otimes (\nabla c_{A,0}-\g):\nabla\psi\sd x \\
  &= {\int_{\Gamma_t(2\delta)}\zg \,  \theta'_0\left(\tfrac{d_\Gamma}\eps-h_\eps\right) \nabla_\btau u_1\otimes \no: \nabla\psi\sd x}  + {\int_{\Gamma_t(2\delta)}\zg \,  \theta'_0\left(\tfrac{d_\Gamma}\eps-h_\eps\right)\partial_\no u_1 \no\otimes \no : \nabla\psi\sd x}\\
  &\ { -\int_{\Gamma_t(2\delta)}\zg \,  \theta'_0\left(\tfrac{d_\Gamma}\eps-h_\eps\right)\left(\eps\nabla u_1 \otimes \nabla_\btau h_1 \right): \nabla\psi\sd x}+O(e^{- \frac{\alpha\delta}{2\eps}})\quad \text{in }L^2(0,T_\eps),
\end{alignat*}
where each term can be estimated as before.
Hence we obtain
\begin{equation*}
  \|\tilde{\we}_1(t)\|_{H^1(\Om)}\leq C \left(\eps^{-\frac12}\|\nabla_\btau u_1\|_{L^2(\Gamma_t(2\delta))}+ \sqrt{\eps} \|\partial_\no u_1\|_{L^2(\Gamma_t(2\delta))}+e^{- \frac{\alpha\delta}{2\eps}}\right)\quad  \text{in }L^2(0,T_\eps)
\end{equation*}
 and this together with \eqref{yuning:diffcaca1New} leads to
\eqref{eq:estofw}.

\smallskip

\noindent
\emph{ Proof of \eqref{eq:estofw'}:}
 We still need to estimate the same terms as in
\eqref{yuning:1.20}. But it is sufficient to estimate $II$ since all other terms can be estimates in terms of $C\eps^{\order+\frac12}\|\psi\|_{H^1(\Om)}$. To this end let $\psi\in H^2(\Om)^2\Second{\cap H^1_0(\Om)^2}\cap L^2_\sigma(\Om)$. Then
\begin{alignat*}{1}
  II &= \int_{\Gamma_t(2\delta)} \zg \,\theta'_0\left(\tfrac{d_\Gamma}\eps-h_\eps\right)  u_1 \partial_\no \Div_\btau \psi \sd x \\
  &+
\int_{\Gamma_t(2\delta)}\zg \, \tfrac1\eps \theta''_0\left(\tfrac{d_\Gamma}\eps-h_\eps\right)  u_1  \Div_\btau \psi \sd x+O(e^{- \frac{\alpha\delta}{2\eps}})~\text{in}~L^2(0,T_\eps)
\end{alignat*}
and therefore Lemma \ref{yuning:1.21} implies
\begin{alignat*}{1}
  |II| &\leq C\|u_1(t)\|_{L^2(\Om)}\left(\|\psi\|_{H^2(\Om)}+ \tfrac1\eps\sqrt{\eps}\|\theta '' _0\|_{L^2(\R)}
  \|\Div_\btau\psi\|_{L^{2,\infty}(\Gamma_t(2\delta))}\right) +O(e^{- \frac{\alpha\delta}{2\eps}})\\
&\leq C'\eps^{-\frac12} \|u_1(t)\|_{L^2(\Om)}\|\psi\|_{H^2(\Om)}+O(e^{- \frac{\alpha\delta}{2\eps}})\qquad \text{in }L^2(0,T_\eps).
\end{alignat*}
Since we have $\|u_1\|_{L^2(0,T;L^2)}\leq C(R)T^{\frac14}\eps^{\order+\frac12}$ for all $0<T\leq T_\eps$ due to \eqref{yuning:diffcaca1New},  we obtain \eqref{eq:estofw'}.

\noindent
\emph{ Proof of \eqref{eq:estofwc}:} Because of $W^2_{q}(\Om)\hookrightarrow C^0(\ol\Om)$, we have the imbedding $L^1(\Om)\hookrightarrow (W^2_{q}(\Om))'$ and it follows from \eqref{eq:VeryWeakEstim}    that
\begin{equation*}
  \begin{split}
    &\|\tilde{\we}_2(t,\cdot)\|_{L^r(0,T;L^q)}
    \leq C\eps\left( \|\nabla u_1\otimes \nabla u_1\|_{L^r(0,T;L^1)}+ \|\g\otimes \nabla u_1\|_{L^r(0,T;L^1)} \right).
  \end{split}
\end{equation*}
On the other hand
\begin{equation*}
  \eps^{\frac12} \|\nabla u_1\|_{L^\infty(0,T;L^2)}\leq \eps^{\frac12} \|\nabla c_\eps\|_{L^\infty(0,T;L^2)}+ \eps^{\frac12} \|\nabla c_{A,0}\|_{L^\infty(0,T;L^2)}\leq C
\end{equation*}
uniformly in $T\in (0,T_\eps]$, $\eps\in (0,\min(\eps_0,\eps_1)]$ because of {the energy estimate \eqref{eq:Energy} for \eqref{eq:NSAC1}-\eqref{eq:NSAC4}} and the explicit form of $c_{A,0}$ at \eqref{yuning:ca1}. Altogether we conclude
\begin{alignat*}{1}
  &\|\tilde{\we}_2\|_{L^r(0,T;L^q)} \leq C\eps\left( \|\nabla u_1\otimes \nabla u_1\|_{L^r(0,T;L^1)}+ \|\g\otimes \nabla u_1\|_{L^r(0,T;L^1)} \right)\\
&
  \leq C\eps\left( \|\nabla u_1\|_{L^2(0,T;L^2)}^{\frac2r} \|\nabla u_1\|_{L^\infty(0,T;L^2)}^{2-\frac2r} +\eps \|\partial_s h_{2,\eps}\|_{L^\infty(0,T;L^2(\T^1))}\|\theta_0'(\tfrac{\cdot}\eps)\|_{L^2(\R)} \|\nabla u_1\|_{L^2(0,T;L^2)} \right)\\
&\leq C\left(\eps^{1+ \frac2r\frac32-1+ \frac1r}+ M\eps^{N+2}\right)\leq  C'\eps^{\frac4r}\qquad \text{for all }0<t \leq T_\eps, \eps\in (0,\eps_0].
\end{alignat*}
Therefore we have completed the proof.\end{proof}

%% file: SharpInterfaceLimitNSACApproxAC.tex
\section{Approximate Solution for the Allen-Cahn Part}\label{eq:Approx}

Let us start with an overview of this central section.
The goal of this section is to construct the full expansion of \eqref{eq:NSAC3} through Ansatz using inner and outer expansions, cf. \eqref{yuning:ca} below, and  to prove Theorem \ref{thm:ApproxSolutions}. The crucial step is to  prove Theorem \ref{thm:InnerExpansion} in Section~\ref{subsec:inner} for the inner expansions defined by \eqref{yuning:innerexpan'} in the sequel, which will be glued to the outer expansion in  \eqref{yuning:ca}. To this end, we shall first explicitely define $h_1, h_2, \hat{c}_2, \hat{c}_3$ in the inner expansions, while the technical calculations, which lead to this definitions, are done in the Appendix \ref{subsec:1.1}. Then,  in Lemma \ref{lem:c2c3}, Lemma \ref{lem:Existenceh2} and Lemma \ref{lem:c3Estim}, we prove the validity of these definitions and various estimates are given. In Lemma \ref{lemma:va}, a part of the error in convection term is treated carefully.  In Lemma \ref{lem:MatchEstim}, we deal with estimates away from the interface, which is less subtle. For a first reading, the  proofs of these technical lemmas can be skipped.

\subsection{Inner Expansion}\label{subsec:inner}

 In this section we will construct the inner expansion of the approximate solution $c_A$ with the aid of the following ansatz:
\begin{equation}\label{yuning:innerexpan'}
\begin{split}
  c^{in} (x,t)=\tc^{in}(\rho,s,t)=&\, \theta_0(\rho) + \eps^2 \tc_2(\rho,S(x,t),t) + \eps^3 \tc_3\(\rho,S(x,t),t\)\\
  =:&\, c_0^{in}(x,t)+\eps^2c_2^{in}(x,t)+\eps^3c_3^{in}(x,t)
\end{split}
\end{equation}
 where  $s=S(x,t)$ and $\rho$ is  related to $(x,t)$ by
 \begin{equation}\label{eq:rho}
   \rho = \frac{d_\Gamma(x,t)}\eps - h_1(S(x,t),t)-\eps h_{2,\eps}(S(x,t),t).
 \end{equation}
 We note that in the functions which will be defined in the following, $\rho$ is often an independent variable. But in the final definition of the approximate solutions $\rho$ is related to $(x,t)$ by \eqref{eq:rho}.
Here $\hat{c}_2, \hat{c}_3, h_1, h_{2,\eps}$ are chosen as follows: We define $h_1=h_1(s,t)$ as the solution of the following  linear parabolic equation, coupled with  \eqref{yuning:secondlimit}, which determines $\ve_1^\pm$:
  \begin{equation}\label{yuning:h1equ'}
      D_t  h_1 -X_0^\ast (\ve) \cdot  \nabla_\Gamma h_1  - \Delta_\Gamma h_1  -\kappa_1 h_1 + X_0^\ast(\divtau \ve) h_1=X_0^\ast(\ve_{1,\no}^\pm)\quad \text{on }\T^1\times [0,T_0]
   \end{equation}
together  with the initial data  $h_1|_{t=0}=0$, cf. Corollary~\ref{cor:ParabolicEqOnSurfaceStokes}.  Note that  here we use the notation \eqref{yuning:1.27} and $(\ve,\G)$ is the solution of  \eqref{eq:Limit1}-\eqref{eq:Limit4}. $\kappa_1$ is a curvature term defined in \eqref{yuning:ex1}.
  Using $h_1$, we shall define $\tc_2=\tc_2(\rho,s,t)$ as the bounded  solution of the following ordinary differential equation:
  \begin{equation}\label{yuning:c2equ'}
   \begin{split}
     &-\p_\rho^2 \tc_2(\rho,s,t)+f''(\theta_0(\rho)) \tc_2(\rho,s,t)\\
     &=|\nabla_\Gamma h_1(s,t)|^2\theta_0''(\rho)-\rho\theta_0'(\rho)    (\kappa_1(s,t)- (\divtau \ve)(X_0(s,t),t))
   \end{split}
\end{equation}
 together with $\tc_2(0,s,t)=0$
for all $(\rho,s,t)\in\R\times \T^1\times[0,T_0]$.
For the following we denote
\begin{alignat}{1}\nonumber
&b(\rho,s,t):=
 \frac12\left[(\partial_\no^2 \ve^+_{0,\no}+\partial_\no^2 \ve^-_{0,\no })(X_0(s,t ),t) +(\partial_\no^2 \ve^+_{0,\no }-\partial_\no^2 \ve^-_{0,\no })(X_0(s,t ),t)\eta(\rho)\right]  (\rho+h_1(s,t))^2 \\ \nonumber
&+\frac12\left[(\partial_\no\ve^+_{0 }+ \partial_\no\ve^-_{0 })(X_0(s,t),t)+ (\partial_\no\ve^+_{0 }- \partial_\no\ve^-_{0})(X_0(s,t),t)\eta(\rho)\right]\cdot \nabla_\Gamma h_1(s,t)(\rho+h_1(s,t))\\ \label{eq:DefnB1}
&+  \ve_1^\pm(X_0(s,t),t)\cdot \nabla_\Gamma h_1(s,t)- (\divtau \ve_1^\pm)(X_0(s,t),t) h_1(s,t) + \tilde{\ve}_{2,\no}(\rho,X_0(s,t),t),
\end{alignat}
where  $\ve^\pm_0=\ve|_{\Omega^\pm}$ and extended smoothly to $\Om\times [0,T_0]$ as before  and $\ve_1^\pm$ are determined by \eqref{yuning:secondlimit}.
Then $b$ is a smooth function in $(\rho,s,t)$, which  is independent of $h_{2,\eps}$ and satisfies
 \begin{equation}\label{yuning:1.36}
   |(b,\p_s b)|\leq C(1+|\rho|^2)~\text{for all}~(\rho,s,t)\in\R\times \T^1\times [0,T].
 \end{equation}
Furthermore, we define
\begin{alignat}{1}\nonumber
  \mathfrak{D}_1(r,\rho,s,t)&= -2 \nabla_\Gamma h_1(s,t)\cdot \Lgrad h_1(r,s,t) \theta_0''(\rho)- |\Lgrad h_1(r,s,t)|^2 \theta_0''(\rho)\\\nonumber
  &\quad+\theta_0'(\rho)\left[(\Ldelta h_1-\Lt h_1)(r,s,t) + \ve(X_0(s,t),t)\cdot \Lgrad h_1(r,s,t)\right],
  \\\label{yuning:1.37}
  \mathfrak{D}(\rho,s,t)&= \partial_r \mathfrak{D}_1(0,\rho,s,t),
\end{alignat}
where $r\in (-3\delta,3\delta)$, $s\in\T^1$, $t\in [0,T_0]$, and $\rho\in \R$.
Moreover,
\begin{equation}\label{yuning:deterb'}
 \mathfrak{B}(s,t)
  :=\frac 1{\int_\R(\theta_0')^2d\rho}\int_\R\left[  (b(\rho,s,t)-\kappa_2(s,t)\rho^2)\theta_0'(\rho) + (\rho+h_1(s,t))\mathfrak{D}(\rho,s,t) \right]\theta_0'(\rho)\sd \rho.
\end{equation}
Now we define $h_2=h_{2,\eps}\colon \T^1\times [0,T_0]\to \R$ as solution of the following parabolic equation
\begin{alignat}{1}\label{yuning:h2equ'}
   & D_t  h_2    -   \Delta_\Gamma h_2 - \kappa_1h_2 -   X_0^\ast(\ve)\cdot \nabla_\Gamma h_2 + X_0^\ast (\divtau \ve) h_2
 =\mathfrak{B}- \kappa_2h_1^2+X_0^\ast(\no\cdot\we_1)
\end{alignat}
on  $\T^1\times [0,T_0]$ together with initial  condition $h_2|_{t=0}=0$, where
\begin{equation}\label{yuning:2.02}
 \we_1=\frac{\tilde{\ve}_A-\ve_A+\tilde{\we}_1}{\eps^2}\quad \text{with}~\tilde{\we}_1\in L^\infty(0,T_0;H^1_0(\Om)^2\cap L^2_\sigma(\Om))
\end{equation}
is determined by \eqref{eq:w1a'}-\eqref{eq:w1b'}  and $\ve_A, \tilde{\ve}_A$ are as in the beginning of  previous section.
Finally, we shall define $\hat{c}_3=\hat{c}_3(\rho,s,t)$ as the solution of
\begin{equation}\label{yuning:monst4'}
\begin{split}
&\eps^2 (D_t \tc_3 -  \Delta_\G \tc_3) -\p_\rho^2\tc_3+f''(\theta_0)\tc_3\\
&=2\nabla_\Gamma h_1\cdot\nabla_\Gamma h_2\theta_0''-\(b- \mathfrak{B}-\kappa_2(\rho^2+2\rho h_1)\) \theta_0'- (\rho+h_1)\mathfrak{D} \quad \text{on }\R\times\T^1\times [0,T_0]
\end{split}
\end{equation}
together with initial  condition  $\tc_3|_{t=0}=0$.

 The next lemma provides important estimates to treat remainder terms in the following.
\begin{lem}\label{lem:c2c3}
  The functions  $h_1,\tc_2, h_2, \tc_3$ are well defined through the above formulae. For $\tc_2$ we have
  \begin{equation}\label{yuning:1.16}
   \p_\rho^i\p_s^j\p_t^k\tc_2(\rho,S(x,t),t) \in\mathcal{R}_{0,\alpha},\quad \forall i,j,k\geq 0.
  \end{equation}
  Moreover, for $\tc_3$ we have  for every $k\in\N_0$,  $\theta>0$,   $\eps\in (0,1)$ and $T_\eps \in (0,T_0]$ that
 \begin{alignat}{1}\nonumber
    &\eps\| (\rho^k \tc_3, \p_s \tc_3)\|_{L^\infty(0,T_\eps;L^2(\T^1\times \R))}\\\label{yuning:1.38}
&+ \eps^\theta \sup_{(s,t)\in \T^1\times (0,T_\eps)}\|\hat{c}_3(\cdot,s,t)\|_{H^1(\R)}+\eps^\theta \|\hat{c}_3\|_{L^\infty((0,T_\eps)\times \T^1\times \R)}\\\nonumber
&  +\| ( \rho^k \tc_3, \rho^k\p_\rho \tc_3, \p_s \tc_3,\p_\rho^2 \tc_3,\partial_s\p_\rho \tc_3, \eps\partial_s^2 \tc_3,\eps\rho^k\p_s\hat{c}_3) \|_{L^2((0,T_\eps)\times \R\times\T^1)}  \leq C_{k,\theta }(1+\|h_\eps\|_{X_{T_\eps}}),
  \end{alignat}
where $C_{k, \theta }$ is independent of $\eps, T_\eps$ and $h_\eps$.
\end{lem}
\begin{proof}
First of all,  we will show that all terms are well-defined.
Because of Corollary \ref{cor:ParabolicEqOnSurfaceStokes},  $h_1$ can be uniquely determined by solving the coupled system involving \eqref{yuning:secondlimit} and \eqref{yuning:h1equ'} together with $h_1|_{t=0}=0$. Moreover,  we can obtain $\tc_2=\tc_2(\rho,s,t)$ by solving \eqref{yuning:c2equ'} for every $s\in\T^1$, $t\in [0,T_0]$ using Proposition~\ref{yuning:solveode}. Note that the compatibility condition \eqref{lem:CompCondODE} is fulfilled as
\begin{alignat*}{1}
  &\int_{\R} \left(|\nabla_\Gamma h_1(s,t)|^2\theta_0''(\rho)-\theta_0'(\rho)   \rho \left(\kappa_1(s,t) -(\Div_{\btau} \ve)(X_0(s,t)) \right)\right)\theta_0'(\rho) \sd \rho \\
  &= |\nabla_\Gamma h_1(s,t)|^2\int_\R \theta_0''(\rho)\theta_0 '(\rho)\sd \rho - \left(\kappa_1(s,t)  -(\Div_{\btau} \ve)(X_0(s,t)) \right)\int_{\R}\rho (\theta_0'(\rho))^2\sd \rho =0,
\end{alignat*}
following from the fact that  $\theta'_0$ is an even function and $\theta_0'(\rho)\to_{|\rho|\to \infty} 0$.

This leads to a smooth function $\mathfrak{B}$ defined by \eqref{yuning:deterb'}. Moreover, according to \eqref{yuning:matching1} and Definition \ref{yuning:1.15}, we can use Proposition \ref{yuning:solveode} to show \eqref{yuning:1.16}, where $g^\pm=0$.
  The existence of a unique solution
  $ h_2\in X_{T_0} $ 
  of \eqref{yuning:h2equ'} is shown in the next lemma. Note that, as $\we_1$ depends on $h_2$, the equation for $h_2$ is non-linear.
  %

Finally,   we need to establish the estimate for  $\tc_3$ satisfying \eqref{yuning:monst4'}. In order to apply
 Theorem \ref{lem:InftyEstimLinPart}, we need to estimate
 \begin{equation*}
   \tilde{g}(\rho,s,t):=2\nabla_\Gamma h_1\cdot\nabla_\Gamma h_2\theta_0''(\rho)-\(b- \mathfrak{B}-\kappa_2(\rho^2+2\rho h_1)\) \theta_0'(\rho)- (\rho+h_1)\mathfrak{D}(\rho,s,t).
 \end{equation*}
 We note that the compatibility condition \eqref{eq:OrthogonalitySol} is fulfilled due to  the definition of $\mathfrak{B}$.
 It follows from  \eqref{yuning:1.36}, \eqref{yuning:1.37}, \eqref{yuning:deterb'} and the decay estimate \eqref{yuning:decayopti}  that
 \begin{alignat*}{1}
    &\eps\| (\tc_3,\p_s \tc_3)\|_{L^\infty(0,T_\eps;L^2(\T^1\times \R))}  +\| (  \tc_3,\p_\rho \tc_3, \p_s \tc_3,\p_\rho^2 \tc_3, \partial_s \p_\rho \tc_3, {\eps \partial_s^2 \tc_3}) \|_{L^2(\R\times\T^1\times(0,T_\eps))} \\
& \leq C\|(\tilde{g},\p_s \tilde{g},\partial_\rho \tilde{g})\|_{L^2(\R\times\T^1\times(0,T_\eps))}\leq C'(1+\|h_\eps\|_{X_{T_\eps}}).
  \end{alignat*}
  Moreover,  using the exponential decay of $\theta_0'(\rho)$, $\theta_0''(\rho)$ and $\tilde{g}(\rho,s,t)$ as $\rho \to \infty$ and \eqref{yuning:1.12Weigthed}, it is easy to observe that for any $k\in\N$ there is a constant $C_k$ independent of $\tc_3,\eps, h_\eps$ such that
  \begin{alignat*}{1}
&\| \eps\rho^k \tc_3\|_{L^\infty(0,T_\eps;L^2(\T^1\times \R))}  +    \| (  \rho^k\tc_3,\rho^k\p_\rho \tc_3) \|_{L^2(\R\times\T^1\times(0,T_\eps))}\\
&\leq C\|(1+|\rho|)^k\tilde{g}\|_{L^2(\R\times\T^1\times(0,T_\eps))}\leq  C_k(1+\|h_\eps\|_{X_{T_\eps}}).
  \end{alignat*}
  The  second and third  estimate in \eqref{yuning:1.38} follows from \eqref{yuning:1.11} and the boundness of
$$
h_{2,\eps}\in X_{T_\eps}\hookrightarrow BUC([0,T_\eps];H^{\frac32}(\T^1))\cap L^2(0,T_\eps;H^{\frac52}(\T^1))\hookrightarrow L^{2p}(0,T_\eps; C^1(\T^1))
$$
for any $1\leq p<\infty$.
\end{proof}

In the last proof we used the following lemma which is concerned with the solvability of \eqref{yuning:h2equ'}.
\begin{lem}\label{lem:Existenceh2}
Let $\eps\in (0,1)$. Then there is a unique solution $h_{2,\eps}\in X_{T_0}$ of \eqref{yuning:h2equ'}, where $\we_1$ is determined by \eqref{yuning:2.02} with $\tilde{\we}_1$ being  determined by \eqref{eq:w1a'}-\eqref{eq:w1b'}.  Moreover, there are some $\eps_1\in (0,1)$,  $T_1\leq T_0$  and
$M =M(R)>0$, independent of $\eps$, such that the solution $h_\eps$ satisfies \eqref{eq:h2epsBound} (for some $T_\eps\in (0,T_0]$) if \eqref{yuning:diffcaca1New} is valid  and $\eps\leq \eps_1$, $T_\eps\leq T_1$.
\end{lem}
\begin{proof}
  First of all, we note that $\tilde{\we}_1$ and therefore $\we_1$ depends on $h_{2,\eps}$ since the definition of $c_{A,0}$  in \eqref{yuning:ca1} contains the term
  \begin{equation*}
    \theta_0\left(\tfrac{d_\G(x,t)}\eps -h_1(s,t)-\eps h_{2,\eps}(s,t)\right).
  \end{equation*}
  Using Theorem~\ref{thm:ParabolicEqOnSurface}, one can reduce \eqref{yuning:h2equ'} to a fixed point equation
  \begin{equation*}
    h_{2,\eps}= S_T(h_{2,\eps})\qquad \text{in } X_T.
  \end{equation*}
  To solve it, one can first apply the contraction mapping principle to obtain a solution $h_{2,\eps}\in X_{T}$ for some $T=T_\eps \in (0,T_0]$. To this end one uses that for every $h_{2,\eps}\in X_T$
  \begin{alignat*}{1}
   &\sup_{0\leq t\leq T}\left|\eps\int_{\Om} \nabla (c_\eps-c_{A,0})\otimes \left(\nabla c_{A,0}-\g\right):D\bfvarphi(x) \sd x\right|\\
&\quad \leq C(\eps) \left(\sup_{0\leq t\leq T}\|h_{2,\eps}(.,t)\|_{H^1(\T^1)} + 1\right)\|\bfvarphi\|_{H^1(\Om)} 
  \end{alignat*}
  for all $\bfvarphi\in H^1(\Om)^2$ because of \eqref{eq:EmbeddingE1}, \eqref{eq:cAg},
and since $c_\eps$ is a known smooth function.
  Therefore the $L^2(0,T;H^{1/2}(\T^1))$-norm of $X_0^\ast (\no\cdot\we_1)$ on the right-hand side of \eqref{yuning:h2equ'} can be estimated by
  \begin{equation}\label{eq:Bddh2eps}
    C'(\eps)T^{\frac12}\left(\sup_{0\leq t\leq T}\|h_{2,\eps}(.,t)\|_{H^1(\T^1)} + 1\right).
  \end{equation}
  Similarly one shows that the right-hand side is Lipschitz-continuous with respect to $h_{2,\eps}$ with Lipschitz constant $C(\eps)T^{\frac12}$. Hence choosing $T=T_\eps'\in (0,T_0]$ sufficiently small one obtains a contraction. Moreover, because of the linear growth of the bound in \eqref{eq:Bddh2eps}, there is an a~priori bound of $\sup_{0\leq t\leq T_\eps'}\|h_{2,\eps}(\cdot,t)\|_{H^1(\T^1)}$, which depends only on $\eps$ and $T_0$. Therefore the solution can be extended on $[0,T_0]$ to a unique solution $h_{2,\eps}\in X_{T_0}$.  The details are left to the reader.

Now we assume that \eqref{yuning:diffcaca1New} is valid for some $T_\eps\in (0,T_0]$.
In order to show the validity of \eqref{eq:h2epsBound},
we use that
\begin{equation*}
  \|h_{2,\eps}\|_{X_{T_\eps}}\leq C(1+ \|X_0^\ast (\no\cdot\we_1)\|_{L^2(0,T_\eps;H^{ 1/2}(\T^1))})
\end{equation*}
because of Theorem~\ref{thm:ParabolicEqOnSurface}. To estimate the right hand side of the above inequality, we use \eqref{yuning:2.02}, Theorem \ref{thm:ApproxVA} and Proposition~\ref{yuning:velocityest} and obtain
\begin{equation*}
\begin{split}
&  C(1+\|X_0^\ast (\no\cdot\we_1)\|_{L^2(0,T_\eps;H^{1/2}(\T^1))})\\
\leq &C'(1+ \eps^{-2}  \|\ve_A-\tilde{\ve}_A\|_{L^2(0,T_\eps;H^1(\O))}+ \eps^{-2} \|\tilde{\we}_1\|_{L^2(0,T_\eps;H^1(\Om))}) \leq C(R)+C(M)(T^{1/4}+\eps^{1/2})
\end{split}
\end{equation*}
 for some ${C,C',C(R)}$ independent of $\eps$, $T_\eps$, $M$, and $h_{2,\eps}$ due to \eqref{eq:estofw} as long as $\eps\leq \min(\eps_0,\eps_1)$, where $\eps_1$ is as in Proposition~\ref{yuning:velocityest} and depends on $M$. Note that $C(R)$ is independent of $M$.
Now we choose  $M=C(R)+1$ and $T_1\in (0,T_0],\eps_1\in (0,1]$ such that $C(M)(T^{\frac14}+\eps_1^{1/2})\leq 1$ and Proposition~\ref{yuning:velocityest} can be applied. This determines $\eps_1$ and $T_1$   and finishes the proof.
\end{proof}

The next lemma is concerned with the estimate of the inner expansion defined by \eqref{yuning:innerexpan'}.
\begin{lem}\label{lem:c3Estim}
 Assume that  \eqref{eq:h2epsBound} holds true for some $M>0$, $T_\eps \in [0,T_0],\eps\in (0,\eps_0]$, and $\eps_0\in (0,1]$.
  Then there is some  $C(M)>0$ independent of $T_\eps$, $\eps\in (0,\eps_0]$, and $\eps_0\in (0,1]$ such that
  \begin{alignat*}{2}
    \eps^2\|(c^{in}_2,\nabla_\btau c^{in}_2) \|_{L^\infty(0,T_\eps;L^2(\Gamma_t(2\delta)))}&\leq C(M) \eps^{\order+\frac12},&\
    \eps^2\|\partial_\no c^{in}_2 \|_{L^\infty(0,T_\eps;L^2(\Gamma_t(2\delta)))}&\leq C(M)\eps^{\order-\frac12}.
\\
    \eps^3\|c^{in}_3 \|_{L^\infty(0,T_\eps;L^2(\Gamma_t(2\delta)))}&\leq C(M,\theta)\eps^{N+\frac32 -\theta },&\
    \eps^3\|\nabla c^{in}_3 \|_{L^\infty(0,T_\eps;L^2(\Gamma_t(2\delta)))}&\leq C(M,\theta)\eps^{\order+{\frac12}-\theta},\\
   {\eps^3\|\nabla c^{in}_3 \|_{L^2(0,T_\eps;L^2(\Gamma_t(2\delta)))}}&{\leq C(M)\eps^{\order+{\frac12}}}.
  \end{alignat*}
  for any $\theta\in (0,1)$.
\end{lem}
\begin{proof}
The estimates of $c^{in}_2$ follow from \eqref{yuning:1.16}, \eqref{yuningliu:formula1},  and Lemma~\ref{lem:Remainder} in a straight-forward manner.
Using Lemma \ref{lem:c2c3} and the same change of variables as in the proof of Lemma~\ref{lem:Remainder}, we conclude
  \begin{alignat*}{1}
    \eps^3\|c^{in}_3 \|_{L^\infty(0,T_\eps;L^2(\Gamma_t(2\delta)))}&\leq C(M)\eps^{3+\frac12}\sup_{t\in [0,T_\eps]}\|\tc_3(\cdot,\cdot,t)\|_{L^2(\T^1\times \R)}\leq C'(M,\theta)\eps^{3+\frac12 -\theta }
\end{alignat*}
as well as
  \begin{alignat*}{1}
    &\eps^3\|\nabla c^{in}_{3} \|_{L^\infty(0,T_\eps;L^2(\Gamma_t(2\delta)))}\\
&\leq C(M)\eps^{3+\frac12}\left( \frac1\eps \sup_{t\in [0,T_\eps], s\in\T^1}\|\partial_\rho \tc_{3}(\cdot,s,t)\|_{L^2(\R)} +\|\partial_s \tc_{3}\|_{L^\infty(0,T_\eps; L^2(\T^1\times \R))}\right)\\
    &\leq C(M,\theta)\eps^{\order+{\frac12}-\theta}.
  \end{alignat*}
Here we first applied  \eqref{yuningliu:formula1}   to $\tc_3$   and then employed \eqref{yuning:1.38}. 

  Finally,
the last inequality follows from the boundedness of $h_2\in X_T\hookrightarrow BUC([0,T_\eps]; W^1_4(\T^1))$, \eqref{yuningliu:formula1} and
  \begin{alignat*}{1}
    &\eps^3\|\nabla c^{in}_{3} \|_{L^2(0,T_\eps;L^2(\Gamma_t(2\delta)))}\\
&\leq C(M)\eps^{3+\frac12}\left( \|(\tfrac 1\eps\partial_\rho \tc_{3},\partial_s  \tc_{3})\|_{L^2((0,T_\eps)\times \T^1\times \R)} +\|\partial_\rho \tc_{3}\|_{L^2((0,T_\eps)\times \R;L^4(\T^1))}\right)
    \leq C(M)\eps^{\order+{\frac12}}.
  \end{alignat*}

\end{proof}

The following result gives an important expansion formula for the convection term. Recall that, $\Lgrad$ is defined by \eqref{Prelim:1.11}.
\begin{lem}\label{lemma:va}
  Let $\ve_A^{in}$ be defined via \eqref{yuning:innerd}, \eqref{yuning:summery1}-\eqref{yuning:summery2}  and assume that  \eqref{eq:h2epsBound} and \eqref{yuning:diffcaca1New} hold true.  Then
  \begin{alignat}{1}\nonumber
    &\ve^{in}_A(\rho,x,t)\cdot \nabla c^{in}(x,t) = (\tfrac1\eps \ve_\no|_{\Gamma}+ \ve_{1,\no}^\pm|_{\Gamma})(x,t)\cdot \theta_0'(\rho)\\\nonumber
&  + \left(\ve|_{\Gamma}\cdot \nabla_\Gamma h_1+ \ve|_{\Gamma}\cdot \Lgrad h_1 - (\divtau \ve)|_{\Gamma} (\rho+h_1)\right)(x,t)\theta_0'(\rho)+ \eps b(\rho,s,t)\theta_0'(\rho)\\
&
+ \eps \left(\ve|_{\Gamma}\cdot \nabla_\Gamma h_{2,\eps} - \divtau \ve|_{\Gamma} h_{2,\eps}\right)(x,t)\theta_0'(\rho)\label{eq:ExpansionConvTerm}
 +  \eps\ve_\no|_{\Gamma}(x,t)\partial_\rho \hat{c}_2(\rho,s,t)+ R_\eps (x,t)
  \end{alignat}
in $\Gamma(3\delta)$, where $s=S(x,t)$,  $\rho$ is as in \eqref{eq:rho}, and   $b=b(\rho,s,t)$ is defined in \eqref{eq:DefnB1}. Moreover,
\begin{equation*}
\|R_\eps\|_{ L^2(0,T_\eps;L^2(\Gamma_t(2\delta))) } \leq C(R,M)\eps^{\order+1/2}\qquad \text{for all }\eps\in (0,\eps_0],
\end{equation*}
where $C(R,M)$ is independent of $T_\eps, \eps$, and $\eps_0$.
\end{lem}
\begin{proof}
In the proof, all identities should be interpreted in terms of the variable $x$ instead of the surface coordinate $s\in \T^1$. For example, according to our definition, the function $\nabla_\G h_2=(\nabla_\G h_2)(s,t)$ with $s\in\T^1$,  is a function defined on the chart of the interface. However, in this proof, it is understood as $\nabla_\G h_2=(\nabla_\G h_2)(S(x,t),t)$. Similarly, $\ve|_\G$ is the restriction of $\ve$ on the interface $\G$. However, it should be understood via $\ve|_\G=\ve(P_{\G_t}(x),t)$, as a function of $x$. In the sequel the dependence on variables $(x,t)$ is emphasised occasionally for clarity.

  In the following $R_\eps$ will denote any term such that $ \|R_\eps\|_{L^2(\Gamma(2\delta)\cap (\Omega\times (0,T_\eps)))}  \leq  C\eps^{\order+1/2}$ for all $\eps \in (0,\eps_0]$ and for some $C>0$, which depends only on the quantity $M$ in \eqref{eq:h2epsBound}, $R$ in \eqref{yuning:diffcaca1New}, $\ve,\ve_1, $ and $\Gamma$.
  We use that
  \begin{alignat*}{1}
    &\ve^{in}_A(\rho,x,t)\cdot\nabla c^{in}_0(x,t)=\\
&  \ve_A^{in}\left(\tfrac{d_\G }\eps - h_\eps(S(x,t),t),x,t\right)\cdot \left(\tfrac{\no_{\Gamma_t}(x)}\eps - (\nabla^\Gamma h_\eps)(d_\G ,S(x,t),t)\right) \theta_0'\left(\tfrac{d_\G }\eps - h_\eps(S(x,t),t)\right), 
  \end{alignat*}
  where $d_\Gamma=d_\Gamma(x,t)$.
Furthermore, we have $\partial_{\no} \ve_{0,\no}^{ \pm} |_{\Gamma} =-\Div_\btau \ve_0^{ \pm}|_{\Gamma}$, $\ve^\pm_0|_{\Gamma}= \ve|_{\Gamma}$ and therefore $\divtau \ve^\pm_0|_{\Gamma}=\divtau \ve|_{\Gamma}$. Thus a Taylor expansion of $\ve_{0,\no}^\pm$  in terms of distance function $r=d_\G(x,t)$ gives
\begin{equation*}
    \ve_{0,\no}^\pm(x,t)= \ve_\no|_{\Gamma} (x,t) -\Div_\btau \ve|_{\Gamma}(x,t)d_\G (x,t) + \partial_\no^2 \ve^\pm_{0,\no}|_\Gamma(x,t)\tfrac{d_\G (x,t)^2}2
+  r^\pm  (x,t)d_\G (x,t)^3,
\end{equation*}
 where the remainder $r^\pm$ satisfies
 \begin{equation}\label{eq:repsEstim}
   \sup_{(x,t)\in \Gamma(3\delta)}|r^\pm (x,t)|\leq C.
 \end{equation}
 On the other hand, it follows from \eqref{yuning:1.41} and \eqref{yuning:summery1} that,
\begin{equation}\label{yuning:1.53}
  \ve_0(\rho,x,t)=\frac 12(\ve_0^+(x,t)+\ve_0^-(x,t))+\frac {\eta(\rho)}2(\ve_0^+(x,t)-\ve_0^-(x,t)).
\end{equation}
Combining the above two identities leads to
  \begin{alignat*}{1}
   & \frac1\eps\ve_0\left(\rho,x,t\right)\cdot \no(s,t)\\ 
   &= \frac1\eps \ve_\no|_{\Gamma}(x,t) -\divtau \ve|_{\Gamma}(x,t)(\rho +h_1(s,t)+\eps h_{2,\eps}(s,t))\\
&\quad + \frac{\eps}2(\partial_\no^2 \ve_{0,\no}^+|_{\Gamma}+\partial_\no^2 \ve_{0,\no}^-|_{\Gamma})(x,t)(\rho+h_1(s,t)+\eps h_{2,\eps}(s,t))^2\\
 &\quad + \frac{\eps}2(\partial_\no^2 \ve_{0,\no}^+|_{\Gamma}-\partial_\no^2 \ve_{0,\no}^-|_{\Gamma})(x,t)\eta(\rho)(\rho+h_1(s,t)+\eps h_{2,\eps}(s,t))^2 \\
&\quad + \eps^2 \left(\frac{r^+ (x,t)+r^- (x,t)}2+ \frac{r^+ (x,t)-r^- (x,t)}2\eta(\rho) \right)(\rho+h_1(s,t)+\eps h_{2,\eps}(s,t))^3
  \end{alignat*}
where $s=S(x,t)$ and $\rho$ is defined via \eqref{eq:rho}. Lemma \ref{lem:Remainder} implies that the terms in the last line of the above identity give a contribution to $R_\eps$ after multiplication with $\theta_0'(\tfrac{d_\G (x,t)}\eps-h_\eps(s,t))$. Moreover, one can replace $(\rho+h_1(s,t)+\eps h_{2,\eps}(s,t))$ by $(\rho+h_1(s,t))$ in all terms except the second since the remainders give another contribution to $R_\eps$. As a result
  \begin{alignat*}{1}
   & \frac1\eps\ve_0\left(\rho,x,t\right)\cdot \no(s,t)\theta_0'(\tfrac{d_\G (x,t)}\eps-h_\eps(s,t))\\ 
   &= \left(\frac1\eps \ve_{\no}|_{\Gamma}(x,t) -\divtau \ve|_{\Gamma}(x,t)(\rho +h_\eps)\right)\theta_0'(\rho)
 + \frac{\eps}2(\partial_\no^2 \ve_{0,\no}^+|_{\Gamma}+\partial_\no^2 \ve_{0,\no}^-|_{\Gamma})(x,t)(\rho+h_1)^2\theta_0'(\rho)\\
&\quad + \frac{\eps}2(\partial_\no^2 \ve_{0,\no}^+|_{\Gamma}-\partial_\no^2 \ve_{0,\no}^-|_{\Gamma})(x,t)\eta(\rho)(\rho+h_1)^2\theta_0'(\rho) +R_\eps
  \end{alignat*}
 Moreover, a Taylor expansion of $\ve^\pm_{0}$ near the interface is given by
 \begin{equation*}
        \ve^\pm_{0}(x,t)= \ve|_{\Gamma} (x,t) + \partial_\no\ve^\pm_{0}|_{\Gamma}(x,t)d_\G (x,t) + \tilde{r}^\pm  (x,t)d_\G (x,t)^2,
  \end{equation*}
  where  $\sup_{(x,t)\in \Gamma(3\delta)}|\tilde{r}^\pm (x,t)|\leq C.$
This together with \eqref{yuning:1.53} leads to
  \begin{alignat*}{1}
   & \ve_0\left(\rho,x,t\right)\cdot (\nabla^\Gamma h_\eps)(d_\Gamma(x,t),S(x,t),t)\\ 
   &=  \ve|_{\Gamma}\cdot \nabla_\Gamma h_1 + \ve|_{\Gamma}\cdot (\Lgrad h_1)(d_\G,s,t)+ \eps \ve|_{\Gamma}\cdot \nabla_\Gamma h_{2,\eps}+\eps \ve|_{\Gamma} \cdot (\Lgrad h_{2,\eps})(d_\G,s,t) \\
&\quad +\frac{\eps}2\left[\left(\partial_\no\ve^+_{0}|_{\Gamma}+ \partial_\no\ve^-_{0}|_{\Gamma}\right) + (\partial_\no\ve^+_{0}|_{\Gamma}- \partial_\no\ve^-_{0}|_{\Gamma})\eta(\rho)\right]\cdot  \nabla_\Gamma h_1(s,t) (\rho+h_1+\eps h_{2,\eps})\\
&\quad +\frac{\eps}2\left[\left(\partial_\no\ve^+_{0}|_{\Gamma}+ \partial_\no\ve^-_{0}|_{\Gamma}\right) + (\partial_\no\ve^+_{0}|_{\Gamma}- \partial_\no\ve^-_{0}|_{\Gamma})\eta(\rho)\right]\cdot  \Lgrad h_1(d_\Gamma,s,t) (\rho+h_1+\eps h_{2,\eps})\\
&\quad +\frac{\eps^2}2\left[\left(\partial_\no\ve^+_{0}|_{\Gamma}+ \partial_\no\ve^-_{0}|_{\Gamma}\right) + (\partial_\no\ve^+_{0}|_{\Gamma}- \partial_\no\ve^-_{0}|_{\Gamma})\eta(\rho)\right]\cdot (\nabla^\Gamma h_{2,\eps})(d_\G ,s,t)(\rho+h_1+\eps h_{2,\eps})\\
&\quad + \eps^2 \left(\frac{\tilde{r}^+ (x,t)+\tilde{r}^- (x,t)}2+ \frac{\tilde{r}^+ (x,t)-\tilde{r}^- (x,t)}2\eta(\rho) \right)\cdot (\nabla^\Gamma h_\eps)(d_\G ,s,t)(\rho+h_1+\eps h_{2,\eps})^2.
  \end{alignat*}
  Using Lemma~\ref{lem:Remainder} and $\Lgrad h_1|_{\Gamma}\equiv 0$, it is easy to observe that the last three terms give rise to terms $R_\eps$ after multiplication with $\theta_0'(\tfrac{d_\G (x,t)}\eps-h_\eps(s,t))$.
Hence we obtain
  \begin{alignat*}{1}
   & \ve_0\left(\rho,x,t\right)\cdot \nabla c^{in}_0(x,t)\\ 
   &= \frac1\eps \ve_\no|_{\Gamma}\theta'_0(\rho) + \(\ve|_{\Gamma}\cdot \nabla_\Gamma h_1 +\ve|_{\Gamma}\cdot \Lgrad h_1(d_\G,s,t)-\divtau \ve|_{\Gamma}  (\rho+h_1)\)\theta_0'(\rho)\\
& \quad +  \eps (\ve|_{\Gamma}\cdot \nabla_\Gamma h_{2,\eps} -\divtau \ve|_{\Gamma}h_{2,\eps})\theta_0'(\rho)\\
&\quad + \frac{\eps}2(\partial_\no^2 \ve_{0,\no}^+|_{\Gamma}+\partial_\no^2 \ve_{0,\no}^-|_{\Gamma})(\rho+h_1)^2\theta'_0(\rho) + \frac{\eps}2(\partial_\no^2 \ve_{0,\no}^+|_{\Gamma}-\partial_\no^2 \ve_{0,\no}^-|_{\Gamma})\eta(\rho)(\rho+h_1)^2\theta'_0(\rho)\\
&\quad +\frac{\eps}2\big(\left(\partial_\no\ve^+_{0}|_{\Gamma}+ \partial_\no\ve^-_{0}|_{\Gamma}\right) + (\partial_\no\ve^+_{0}|_{\Gamma}- \partial_\no\ve^-_{0}|_{\Gamma})\eta(\rho)\big)  \cdot \nabla_\Gamma h_1(\rho+h_1)\theta'_0(\rho) + R_\eps.
  \end{alignat*}
Similarly we derive
  \begin{alignat*}{1}
   & \eps\ve_1\left(\rho,x,t\right)\cdot \nabla c^{in}_0(x,t)\\ 
   &= \ve_{1,\no}^\pm|_{\Gamma}\theta'_0(\rho) + \eps \left(\ve_1^\pm|_{\Gamma}\cdot \nabla_\Gamma h_1 +\ve_1^\pm|_{\Gamma}\cdot \Lgrad h_1(d_\Gamma,s,t) -\divtau \ve_1^\pm|_{\Gamma} h_1\right)\theta_0'(\rho)\\
& \quad +  \eps^2 (\ve_1^\pm|_{\Gamma}\cdot \nabla^\Gamma h_{2,\eps}(d_\Gamma,s,t)-\divtau \ve_1^\pm|_{\Gamma}h_{2,\eps})\theta_0'(\rho)\\
&\quad + \eps^2 \left(\frac{r^+_{1,\no} +r^-_{1,\no} }2+ \frac{r^+_{1,\no} -r^-_{1,\no} }2\eta(\rho) \right)(\rho + h_\eps)^2\theta'_0(\rho)\\
&\quad + \eps^2 \left(\frac{\tilde{r}^+_{1} +\tilde{r}^-_{1} }2+ \frac{\tilde{r}^+_{1} -\tilde{r}^-_{1} }2\eta(\rho) \right)\cdot \nabla^\G h_\eps (d_\Gamma,s,t)(\rho+h_\eps)\theta_0'(\rho)\\
&= \ve_{1,\no}^\pm|_{\Gamma}\theta'_0(\rho) + \eps \left(\ve_1^\pm|_{\Gamma}\cdot \nabla_\Gamma h_1 -\divtau \ve_1^\pm|_{\Gamma} h_1\right)\theta_0'(\rho) + R_\eps
  \end{alignat*}
as well as
  \begin{alignat*}{1}
   & \eps^2\ve_2\left(\rho,x,t\right)\cdot \nabla c^{in}_0(x,t)= \eps \tilde{\ve}_{2,\no}(\rho, P_{\Gamma_t}(x),t)\theta'_0(\rho) - \eps^2\ve_2(\rho,x,t)\cdot (\nabla^\Gamma h_\eps)(d_\G,s,t)\theta'_0(\rho)\\
&\quad +\eps^2\left(\frac{\tilde{\ve}_{2,\no}(\rho, x,t)-\tilde{\ve}_{2,\no}(\rho, P_{\Gamma_t}(x),t)}{d_\Gamma}+  \eta(\rho)\hat{\ve}_{2,\no}(x,t)\right)(\rho + h_\eps)\theta'_0(\rho)\\
&= \eps \tilde{\ve}_{2,\no}(\rho, X_0(s,t),t)\theta'_0(\rho) + R_\eps
  \end{alignat*}
because of \eqref{yuning:summery3}, \eqref{yuning:summery2}, $\Div \ve_1^{ \pm} =0$ and
\begin{alignat*}{1}
    \ve_{1,\no}^\pm(x,t)&= \ve_{1,\no} (s,t) -\Div_\btau \ve_1(s,t)d_\G (x,t)
+  r^\pm_{1,\no}  (x,t)d_\G (x,t)^2, \\
  \ve^\pm_1(x,t)&= \ve_1 (s,t) +  \tilde{r}_{1}^\pm (x,t)d_\G (x,t),
\end{alignat*}
where $r^\pm_{1,\no}$ and $\tilde{r}_{1}^\pm$  satisfy the same estimate as $r^\pm$ in \eqref{eq:repsEstim}.
Finally, with the aid of Lemma~\ref{lem:Remainder} it is easy to show that
\begin{alignat*}{1}
\ve_A(x,t)\cdot \eps^2 \nabla c^{in}_2(x,t) &=\eps \ve_{\no}|_{\Gamma}(x,t) \partial_\rho \tc_2(\rho,S(x,t),t)  +R_\eps(x,t),\quad
\end{alignat*}
and Lemma~\ref{lem:c3Estim} yields
$
\ve_A(x,t)\cdot\eps^3 \nabla c^{in}_3(\rho,S(x,t),t) =R_\eps(x,t).
$
  Hence the statement of the lemma follows if $b$ is defined as in \eqref{eq:DefnB1}.
\end{proof}

\begin{thm}\label{thm:InnerExpansion}
  Let $\ve_A^{in}$ be defined via \eqref{yuning:innerd} and assume that 
 \eqref{eq:h2epsBound} and \eqref{yuning:diffcaca1New} hold true. Then we have
  \begin{equation}\label{eq:IdCin}
\begin{split}
&\p_t c^{in}(x,t)+\ve^{in}_A(\rho,x,t)\cdot\nabla c^{in}(x,t)+ \eps^2\we_1|_{\Gamma}(x,t)\cdot\nabla c_0^{in}(x,t)-\Delta c^{in}(x,t) +\tfrac1{\eps^2} f'(c^{in})\\
    &\qquad =\mathfrak{C}(x,t) 
\qquad \text{for all }(x,t)\in \Gamma(2\delta),
\end{split}
\end{equation}
where
\begin{equation}\label{yuning:1.14}
  \int_0^{T}\left\|\mathfrak{C}(.,t)\right\|_{L^2(\G_t(2\delta))}\sd t\leq C(M,R,T,\eps)\eps^{\order+\frac12} \qquad \text{for all }T\in (0,T_\eps], \eps\in (0,\eps_0)
\end{equation}
for some $C(M,R,T,\eps)$ independent of $(T_\eps)_{\eps\in (0,\eps_0)}$, $\eps_0\in (0,1]$ such that $C(M,R,T,\eps)\to_{(T,\eps)\to 0} 0$.
\end{thm}
\begin{proof}
In the Appendix \ref{subsec:1.1} it is shown by careful, but lengthy calculations (see equation \eqref{yuning:1.43}) that
\begin{alignat}{1}\nonumber
   \mathfrak{C}(x,t)&=R_\eps (x,t) +\sum_{k=2}^5\eps^k\mathfrak{R}_k(\tfrac{d_\G}\eps-h_\eps(S(x,t),t),S(x,t),t)
\\\label{eq:ExpandAC}&+\mathfrak{R}_1 (\tfrac{d_\G}\eps-h_\eps(S(x,t),t),S(x,t),t) +\mathfrak{R}(\tfrac{d_\G}\eps-h_\eps(S(x,t),t),S(x,t),t)
\end{alignat}
where $R_\eps(x,t)$ is defined as  in Lemma~\ref{lemma:va} 
and $b=b(\rho,s,t)$ is defined   in \eqref{eq:DefnB1}. Moreover, 
\begin{equation*}
  \begin{split}
  \mathfrak{R}:=&-\eps^2 \kappa_{3,\eps}\theta'_0-\eps^2 \kappa_2(2(\rho+h_1) h_2+ \eps h^2_2)\theta'_0+\eps^2\widehat{\mathfrak{D}}_\eps+\eps(\Ldelta -\Lt )h_2\theta'_0\\
&-2\eps\theta''_0( \nabla_\G h_1\cdot\Lgrad h_2+\Ldelta h_1\cdot \nabla_\G h_2+\Lgrad h_1\cdot\Lgrad h_2)\\
&+\eps^2\((\rho+h_\eps)\kappa_1(s,t)- \eps  \kappa_2 (\rho+h_\eps )^2-\eps^2 \kappa_{3,\eps}(\rho,s,t)\)\p_\rho\tc_2,\\
\mathfrak{R}_1: =& \frac{\eps^2}2f'''\big( \theta_0(\rho)+ \xi(\rho,s,t)(\eps^2\tc_2+\eps^3\tc_3)(\rho,s,t)\big)\left(\tc_2(\rho,s,t)+\eps\tc_3(\rho,s,t)\right)^2,\\
    \mathfrak{R}_2:=& \p_t^\G \tc_2-\Delta^{\Gamma } \tc_2+2\nabla^\Gamma h_1\cdot\nabla^\Gamma \p_\rho \tc_2-(\p_t^\G h_1-\Delta^{\G} h_1)\p_\rho \tc_2-\nabla^\G h_1\cdot \we_1|_{\Gamma}\theta'_0\\
    &-|\nabla^\G h_1|^2 \p_\rho^2 \tc_2-|\nabla^\G h_2|^2 \theta''_0-\p_\rho \tc_3(V+\Delta d_\G),\\
      \mathfrak{R}_3 :=&2\nabla^\G h_2\cdot\nabla^\G \p_\rho \tc_2+2\nabla^\G h_1\cdot\nabla^\G \p_\rho \tc_3-(\p_t^\G h_2 -\Delta^\G h_2)\p_\rho \tc_2 {+ {(\Lt-\Ldelta)\tc_3}}\\
    &-2\nabla^\G h_1\cdot\nabla^\G h_2 \p_\rho^2 \tc_2-|\nabla^\G h_1|^2\p_\rho^2 \tc_3-(\p_t^\G h_1-\Delta^\G h_1)\p_\rho \tc_3-\nabla^\G h_2\cdot\we_1|_{\Gamma} \theta'_0,\\
    \mathfrak{R}_4:=&2\nabla^\G h_2\cdot\nabla^\G \p_\rho
     \tc_3-(\p_t^\G h_2-\Delta^\G h_2)\p_\rho \tc_3-|\nabla^\G h_2|^2\p_\rho^2
      \tc_2-2\nabla^\G h_1\cdot\nabla^\G h_2 \p_\rho^2 \tc_3,\\
    \mathfrak{R}_5:=&-|\nabla^\G h_2|^2 \p_\rho^2 \tc_3,
  \end{split}
\end{equation*}
where $\xi$ is some function with $|\xi(\rho,s,t)| \leq 1$, $\widehat{\mathfrak{D}}_\eps$ is as in \eqref{yuning:1.28} below, and   $\kappa_1, \kappa_2, \kappa_{3,\eps}$ are defined by \eqref{yuning:ex1}.
We recall that $\kappa_1(s,t)$ and $\kappa_2(s,t)$ are smooth  and $\eps$-independent functions while the estimate of $\kappa_{3,\eps}$ is given by \eqref{yuning:1.44}:
\begin{equation}\label{yuning:1.45}
  |\kappa_{3,\eps}(\rho,s,t)|\leq C|\rho+h_\eps(s,t)|^3\qquad \text{for all }\rho\in \R, s\in \mathbb{T}^1, t\in [0,T_0].
\end{equation}
In order to  prove \eqref{yuning:1.14} we estimate $\mathfrak{R}$ and $\mathfrak{R}_k, k=1,\ldots, 5$  individually.

\smallskip

\noindent
{\bf To $\mathfrak{R}$:}  By the Definition \ref{yuning:1.15} and \eqref{Prelim:1.11}, we have
 \begin{alignat*}{1}
 &\eps(\Ldelta h_2-\Lt h_2)(d_\G (x,t),S(x,t),t)\theta'_0(\rho)\\
&-2\eps\theta''_0(\rho)\( \nabla_\G h_1\cdot\Lgrad h_2+\Ldelta h_1\cdot \nabla_\G h_2+\Lgrad h_1\cdot\Lgrad h_2\)(d_\G (x,t),S(x,t),t) \\
&=
 a_1(\rho,x,t)\partial_s h_2+ a_2(\rho,x,t)\partial_s^2 h_2,
\end{alignat*}
where $a_1,a_2\in \mathcal{R}^0_{1,\alpha}$.
  So we can apply Corollary~\ref{cor:repsEstim2}   to deduce
       \begin{equation*}
         \int_0^T\|\eps(\Ldelta h_2-\Lt h_2)\theta'_0(\tfrac{d_\G}\eps-h_\eps)\|_{L^2(\G_t(2\delta))}\sd t \leq C(M)T^{\frac12}\eps^{\order+\frac 12}
       \end{equation*}
       and
       \begin{equation*}
         \int_0^T\left\| 2\eps\theta_0''(\tfrac{d_\G}\eps-h_\eps)( \nabla_\G h_1\cdot\Lgrad h_2+\Ldelta h_1\cdot \nabla_\G h_2+\Lgrad h_1\cdot\Lgrad h_2)\right\|_{L^2(\G_t(2\delta))}\, \sd t\leq C(M)T^{\frac12}\eps^{\order+\frac 12}
       \end{equation*}
       for all $T\in (0,T_\eps]$, $\eps\in (0,\eps_0]$.
       All the rest terms in $\mathfrak{R}$ are multiplied by $\eps^2$, depend only on $h_2(s,t)$ (and not on its derivatives), have exponential decay as $|\rho|\to \infty$ uniformly in $(s,t),\eps$, 
and can be estimated with the help of Lemma~\ref{lem:Remainder} because of $X_T\hookrightarrow C^0(\Om\times [0,T])$ and \eqref{yuning:1.32}. So
    all these estimates together imply
    \begin{equation*}
     \int_0^T \|\mathfrak{R}(\tfrac{d_\G}\eps-h_\eps(S(\cdot),t),S(\cdot),t)\|_{L^2(\G_t(2\delta)))}\sd t\leq C(M,\eps,T)\eps^{\order+\frac 12},
    \end{equation*}
    where $C(M,\eps,T)\to_{(\eps,T)\to 0} 0$.

\noindent
{\bf To $\mathfrak{R}_1$:} First of all, because of \eqref{yuning:1.38}, 
$\eps\|\hat{c}_3\|_{L^\infty((0,T_\eps)\times \T^1\times \R)}$ is bounded.
Therefore there is some $C>0$ such that
\begin{equation*}
  |\mathfrak{R}_1  (\rho,s,t)|\leq C\eps^2|\tc_2(\rho,s,t)+\eps \tc_3(\rho,s,t)|^2\qquad \text{for all }\rho \in\R, s\in \T^1,t\in [0,T_0].
\end{equation*}
Hence, using a change of variable, \eqref{yuning:1.16}, and \eqref{yuning:1.38} again,  we obtain
\begin{alignat*}{1}
  &\int_0^T\|\mathfrak{R}_1 (\tfrac{d_{\G}}\eps-h_\eps(\cdot,t),\cdot,t)\|_{L^2(\Gamma_t(2\delta))}\sd t\leq
   C\eps^{2+\frac12}\|(\tc_2+ \eps \tc_3)^2\|_{L^1(0,T;L^2(\R\times \T^1))}\\
&\leq C\eps^{\order+\frac12}\left(T+\|\eps \tc_3\|_{L^1(0,T;L^2(\R\times \T^1))}+ \|\eps \tc_3\|_{L^2(0,T;L^4(\R\times \T^1))}^2\right) \leq C(M)(T+\eps) \eps^{\order+\frac12}
\end{alignat*}
for all $T\in (0,T_\eps]$, $\eps\in (0,\eps_0]$.

\smallskip

\noindent
{\bf To $\mathfrak{R}_2$:}
 Using \eqref{yuning:1.16} and the smoothness of $h_1$, we can show that all the terms in $\mathfrak{R}_2$ that are related to $\tc_2$ belong to $\mathcal{R}_{0,\alpha}$, i.e.,
      $$
     {  \mathcal{R}_{0,\alpha}\ni \mathfrak{R}_2(\rho,s,t)+|\nabla^\G h_2|^2 \theta_0''(\rho) +\nabla^\G h_1\theta'_0\cdot \we_1|_{\Gamma}+(V+\Delta d_\G)\p_\rho \tc_3(\rho,s,t).
      }$$
      It remains to estimate the last three terms on the right hand side.  A change of variable together with \eqref{yuning:1.38} implies that
      \begin{equation*}
       \begin{split}
         &\int_0^T\|(V+\Delta d_\G)\p_\rho \tc_3(\tfrac{d_\G}\eps-h_\eps(S(\cdot),t),S(\cdot),t)\|_{L^2(\G_t(2\delta))}\sd t\\
         &\leq  CT^{\frac12}\sqrt{\eps}\|\p_\rho \tc_3\|_{L^2(\R\times\T^1\times(0,T'))}\leq C' T^{\frac12}\sqrt{\eps}
       \end{split}
      \end{equation*}
      for all $T\in (0,T_\eps]$, $\eps\in (0,\eps_0]$.
    The estimate for $|\nabla^\G h_2|^2 \theta_0''(\rho)$ follows from \eqref{eq:h2epsBound} and \eqref{yuning:1.46} together with Corollary \ref{cor:repsEstim2}:
       \begin{equation*}
         \begin{split}
         &\int_0^T\left\||\nabla^\G h_2|^2 \theta_0''(\tfrac{d_\G}\eps-h_\eps(S(\cdot,t),t)) \right\|_{L^2(\G_t(2\delta))}\sd t\\
         &\leq  C\int_0^T\left\||\p_s h_2(S(\cdot,t),t)|^2 \theta_0''(\tfrac{d_\G}\eps-h_\eps(S(\cdot,t),t) ) \right\|_{L^2(\Gamma_t(2\delta)))}\sd t\\
         &\leq C \sqrt{\eps}\int_0^T\|\p_s h_2\|^2_{L^4(\T^1)}d t\leq CT\eps^{\frac12}\|h_2\|^2_{X_T}
       \end{split}
       \end{equation*}
       for all $T\in (0,T_\eps]$, $\eps\in (0,\eps_0]$.
       Similarly, it follows from the smoothness of $h_1$,  Corollary \ref{cor:repsEstim2} and trace estimate that
       \begin{equation*}
         \begin{split}
         &\int_0^T\left\|\nabla^\G h_1\theta'_0\cdot \we_1|_{\Gamma}\right\|_{L^2(\G_t(2\delta))}\sd t\leq C\int_0^T \left\| \theta'_0\cdot \we_1|_{\Gamma}\right\|_{L^2(\G_t(2\delta))}\sd t
        \\ & \leq C\sqrt{\eps} \int_0^T\left\|  \we_1|_{\Gamma}\right\|_{L^2(\G_t)}\sd t\leq C  \sqrt{T\eps} \left\|  \we_1\right\|_{L^2(0,T;H^1(\Om))}
       \end{split}
       \end{equation*}
       Since we assume \eqref{eq:h2epsBound} and \eqref{yuning:diffcaca1New}, Proposition \ref{yuning:velocityest} and Theorem \ref{thm:ApproxVA} are applicable. These together with \eqref{yuning:2.02} imply the   estimate of $\we_1$
       \begin{equation}\label{yuning:2.03}
         \left\|  \we_1\right\|_{L^2(0,T;H^1(\Om))}\leq  \eps^{-2}\(C(R)\eps^\order+C(M)(T^{\frac 14}+\eps^{\frac12})\eps^\order\)= C(R) +C(M)(T^{\frac14}+\eps^{\frac12})
       \end{equation}
       for all $T\in (0,T_\eps]$, $\eps\in (0,\eps_0]$. The above two estimates together  imply the estimate for $\mathfrak{R}_2$.

\smallskip

\noindent
{\bf To $\mathfrak{R}_3$:} We rearrange the terms of $\mathfrak{R}_3$ as
\begin{alignat}{1}\nonumber
 \mathfrak{R}_3 =&\overbrace{2\nabla^\G h_2\cdot\nabla^\G \p_\rho \tc_2+\Delta^\G h_2\p_\rho \tc_2-2\nabla^\G h_1\cdot\nabla^\G h_2 \p_\rho^2 \tc_2-\p_\rho \tc_2\p_t^\G h_2}^{=:\mathfrak{R}_3'}-\nabla^\G h_2\cdot\we_1|_{\Gamma} \theta'_0\\\label{yuning:1.62}
 &\underbrace{+2\nabla^\G h_1\cdot\nabla^\G \p_\rho \tc_3+\Delta^\G h_1\p_\rho \tc_3 -|\nabla^\G h_1|^2\p_\rho^2 \tc_3-\p_t^\G h_1 \p_\rho \tc_3+  (\Lt-\Ldelta)\tc_3}_{=:\mathfrak{R}_3''}.
\end{alignat}
The common feature of terms in $\mathfrak{R}_3'$ is that  derivatives of $h_2$ are only multiplied with derivatives of $\tc_2$ but not that of $\tc_3$. So they can be estimated by
\begin{alignat*}{1}
  &\int_0^T\|\eps^3\mathfrak{R}_3' (\tfrac{d_{\G}}\eps-h_\eps(\cdot,t),\cdot,t)\|_{L^2(\Gamma_t(2\delta))}\sd t\leq
   C\eps^{3+\frac12}T^{\frac12}\|h_2\|_{X_T}
\leq {C}(M)\eps^{\order+\frac32},
\end{alignat*}
for all $T\in (0,T_\eps]$, $\eps\in (0,\eps_0]$. Moreover, $\mathfrak{R}_3''$ consists of terms that includes  derivatives of $\tc_3$. Note that every terms here is multiplied by derivatives of smooth and $\eps$-independent functions  and according to \eqref{Prelim:1.11}, $\Lt-\Ldelta$ is a second order operator with coefficients vanishing on $\G$. They can be estimated using \eqref{yuning:1.38}:
\begin{alignat*}{1}
  &\int_0^T\|\eps^3\mathfrak{R}_3'' (\tfrac{d_{\G}}\eps-h_\eps(\cdot,t),\cdot,t)\|_{L^2(\Gamma_t(2\delta)))}\sd t\\
&\leq
   C(M)\eps^{3+\frac12}T^{\frac12}\|(\partial_\rho \tc_3, \partial_\rho^2 \tc_3, \partial_s \partial_\rho \tc_3, \partial_s \tc_3,  \partial_s^2 \tc_3 )\|_{L^2(\R\times \T^1\times (0,T_\eps)) } \leq C(M)\eps^{\order+\frac12}T^{\frac12}
\end{alignat*}
for all $T\in (0,T_\eps]$, $\eps\in (0,\eps_0]$.
Finally,  the estimate of the last summand in \eqref{yuning:1.62}  follows from    Corollary \ref{cor:repsEstim2},
\begin{equation}
  \label{eq:fgInterpol}
  \|fg\|_{L^2(\T^1)}\leq C \|f\|_{H^{1/2}(\T^1)}\|g\|_{H^{1/2}(\T^1)} \quad \text{for all }f,g\in H^{\frac12}(\T^1)
\end{equation}
and \eqref{yuning:2.03} successively:
\begin{equation*}
  \begin{split}
    & \int_0^T\eps^3\| \nabla^\Gamma h_2\cdot \we_1|_{\Gamma}\theta_0' (\tfrac{d_{\G}}\eps-h_\eps(\cdot,t))\|_{L^2(\Gamma_t(2\delta)))}\sd t\\
     &\leq  C\int_0^T\eps^{3+\frac12}\| \p_s h_2(S(\cdot,t),t)\cdot \we_1|_{\Gamma}\|_{L^2(\Gamma_t)}\sd t\\
&\leq
   C\eps^{3+\frac12}T^{\frac12}\|h_2\|_{BUC([0,T_\eps];H^{3/2}(\T^1))}\|X_0^\ast(\we_1)\|_{L^2(0,T;H^{1/2}(\T^1))} \leq C(M,R)\eps^{\order+\frac32}.
  \end{split}
\end{equation*}


  \noindent
{\bf To $\mathfrak{R}_4$:}
For the first term in $\mathfrak{R}_4$, we use that 
      \begin{equation*}
        \begin{split}
          \sup_{r\in (-2\delta,2\delta)}\|\nabla^\G h_2 (r,\cdot,t) \cdot\nabla^\G \p_\rho
     \tc_3 (r,\cdot,\cdot,t) \|_{L^2(\T^1\times\R)}&\leq C\|\partial_s h_2(t)\|_{L^\infty(\T^1)}\|\partial_s \p_\rho
     \tc_3(t)\|_{L^2(\T^1\times\R)}
        \end{split}
      \end{equation*}
      for almost every $t\in [0,T]$
      and $X_T\hookrightarrow L^4(0,T;C^1(\T^1))$.
      Therefore it follows from \eqref{yuning:1.38} that
        \begin{equation*}
        \begin{split}
       &\eps^4\int_0^T \|\nabla^\G h_2 (d_\Gamma,\cdot,t) \cdot\nabla^\G \p_\rho
     \tc_3(d_\Gamma,\tfrac{d_\G}\eps-h_\eps,\cdot,t)\|_{L^2(\G_t(2\delta))}\sd t\\
          &\leq C\eps^{4+\frac12} T^{\frac14}\|h_2\|_{X_{T_\eps}}\|\partial_s \p_\rho
     \tc_3\|_{L^2((0,T_\eps)\times\T^1\times\R)} \leq C(M)\eps^{4+\frac12}
        \end{split}
      \end{equation*}
      for all $T\in (0,T_\eps]$, $\eps\in (0,\eps_0]$.
       For the second term  in $\mathfrak{R}_4$, we apply \eqref{eq:fgInterpol}, H\"older's inequality, and Sobolev interpolation inequality:
       \begin{alignat*}{1}
       & \sup_{r\in (-2\delta,2\delta)}   \|(\partial_t^\Gamma-\Delta^\G) h_2 (r,\cdot,t) \p_\rho \tc_3(\rho,\cdot,t)\|_{L^2(\T^1)}\\
&\leq C(\| h_2\|_{H^{ 5/2}(\T^1)} +\| \partial_t h_2\|_{H^{1/2}(\T^1)} )\|\p_\rho \tc_3(\rho,\cdot,t)\|_{H^1(\T^1)}
       \end{alignat*}
       for almost every $t\in [0,T_\eps]$.
       This together with a change of variable and \eqref{yuning:1.38} leads to
       \begin{equation*}
       \begin{split}
         &\eps^4\int_0^T \|(\partial_t^\Gamma-\Delta^\G) h_2\p_\rho \tc_3(\tfrac{d_\G}\eps-h_\eps,\cdot,t)\|_{L^2(\G_t(2\delta))}\sd t\\
         &\leq C\eps^{4+\frac12}(\|  h_2\|_{L^2(0,T;H^{ 5/2}(\T^1))}+ \|  \partial_t h_2\|_{L^2(0,T;H^{  1/2}(\T^1))})\|\p_\rho \tc_3\|_{L^2((0,T_\eps)\times \R;H^1(\T^1))}\leq C(M)\eps^{\order + \frac52}
       \end{split}
       \end{equation*}
       for all $T\in (0,T_\eps]$, $\eps\in (0,\eps_0]$.
       The remaining terms can be treated in a similar manner and finally we get
       \begin{equation*}
         \eps^4\|\mathfrak{R}_4\|_{L^1(0,T;L^2(\G_t(2\delta)))}=o(\eps^{\order+\frac 12})\quad \text{as }\eps\to 0.
       \end{equation*}
 \noindent
{\bf To $\mathfrak{R}_5$:} It follows from \eqref{eq:h2epsBound}  and Sobolev imbedding that  $\|h_2\|_{L^4(0,T_\eps; C^1(\T^1))}$ is uniformly bounded in $\eps$. So we can show similarly as before that
       \begin{equation*}
        \eps^5\int_0^T\left\||\nabla^\G h_2|^2\p_\rho^2 \tc_3 (\tfrac{d_\G}\eps-h_\eps,\cdot ,t)\right\|_{L^2(\G_t(2\delta)))}\sd t
         \leq C(M)\eps^{5+\frac12}
       \end{equation*}
       for all $T\in (0,T_\eps),\eps\in (0,\eps_0]$.
\end{proof}

\subsection{The Full Expansion}

In this subsection, we shall use the shorter notation $\chi_\pm:=\chi_{\O^\pm(t)}$. 
After determination of the inner expansion, we define the approximate solutions $c_A$ in $\Omega\times [0,T_0]$ as
\begin{equation}\label{yuning:ca}
  \begin{split}
    c_A(x,t)&=\zg c^{in} (x,t) +(1-\zg )\left(c_+^{out}
  \chi_++c_-^{out}\chi_-  \right)\\
  &=c_+^{out}
  \chi_++c_-^{out}\chi_-+ \left(c^{in} (x,t) -c_+^{out}
  \chi_+-c_-^{out}\chi_- \right)\zg
  \end{split}
\end{equation}
where $c_\pm^{out}=\pm 1$ and $\zeta$ is as in  \eqref{yuning:1.34}. Moreover,  $\rho$ is as in \eqref{eq:rho}
and we define $c_{A,2}$, $c_{A,3}$ by
\begin{alignat}{1}
  c_{A,j}(x,t)&=\zg \eps^jc_j^{in}(x,t),\quad j=2,3
\end{alignat}
with $c_j^{in}$ as in \eqref{yuning:innerexpan'}.
\begin{rem}
  We note that we have chosen $c^{out}_\pm = \pm 1$ as approximation of $c_\eps$ in $(\Om\times [0,T_0])\setminus \Gamma(\delta)$ (the region of the outer expansion). One can also derive this (formally) by expanding $c_\eps$ in this outer region in the form $\sum_{k=0}^{2} \eps^kc_k^\pm (x,t)$. Since the calculations are simple and standard, we omit them.
\end{rem}
\begin{lem}\label{lem:MatchEstim}
Assume that \eqref{eq:h2epsBound} holds true for some $M>0$ and let $k\in\N$. Then
  there are $C(M),C_k(M)>0$, independent of $T_\eps,\eps$, and $\eps_0$ such that   for all $0<\eps \leq \eps_0$ and $\theta\in (0,1)$
  \begin{alignat}{1}\label{eq:MatchEstim1}
  &  \sup_{0\leq t\leq T_\eps }\|c^{in}(t,\cdot) -c_+^{out}
  \chi_+ -c_-^{out}\chi_-\|_{L^2(\Gamma_t(2\delta)\setminus \Gamma_t(\delta))}\leq C(M)\eps^{3\order},\\\label{eq:MatchEstim2}
 &    \sup_{0\leq t\leq T_\eps }\eps^2\|c^{in}_2(t)\|_{L^2(\Gamma_t(2\delta)\setminus \Gamma_t(\delta))} +\max_{j=0,2}\sup_{0\leq t\leq T_\eps }\eps^j\|\nabla c^{in}_j(t)\|_{L^2(\Gamma_t(2\delta)\setminus \Gamma_t(\delta))}\leq C(M)\eps^{3\order},
\\\label{eq:MatchEstim4}
& \sup_{0\leq t\leq T_\eps }\eps^3\|c^{in}_3(t)\|_{L^2(\Gamma_t(2\delta)\setminus \Gamma_t(\delta))}+\eps^3\|\nabla c^{in}_3\|_{L^2(0,T_\eps;L^2(\Gamma_t(2\delta)\setminus \Gamma_t(\delta))} \leq C_k(M)\eps^{k}.
  \end{alignat}
\end{lem}
\begin{proof}
First of all $h_{2,\eps}$ is uniformly bounded since $X_{T_\eps}\hookrightarrow L^\infty([0,T_\eps]\times \T^1)$ with operator norm bounded independently of $T_\eps\in (0,T_0]$.
On the other hand, if $2\delta>|d_\Gamma(x,t)|>\delta$, then
\begin{equation}
  \label{eq:DistEstim}
 |\rho|=\left|\frac{d_\G}\eps -h_1(s,t)-\eps h_{2,\eps}(s,t)\right|\geq \frac{\delta}{2\eps}
\end{equation}
for all $\eps\in (0,\eps_0]$ if $\eps_0\in (0,1]$ is chosen sufficently small.
Moreover, we have, because of \eqref{yuning:decayopti}, \eqref{yuning:1.16}, and \eqref{yuning:1.38}, that
\begin{subequations}\label{eq:Matching}
\begin{align}\label{eq:Matching1}
  |\theta_0(\rho)-c_\pm^{out} |+ |\theta'_0(\rho)|&\leq Ce^{-\alpha|\rho|}&\ & \text{if }\rho\gtrless 0,
\\\label{eq:Matching2}
   |\tc_2(\rho,s,t)|+|\partial_\rho \tc_2(\rho,s,t)|+|\partial_s \tc_2(\rho,s,t)|&\leq Ce^{-\alpha|\rho|},&\ & 
   \\
  \eps^\theta|\hat{c}_3(\rho,s,t)|\leq C(M,\theta)\label{eq:Matching3}
\end{align}
\end{subequations}
for all $(\rho,s,t)\in\R\times \T^1\times[0,T_0]$ and some $C,C(M,\theta),\alpha>0$, where $\theta>0$ is arbitrary.

For the following we denote $\Sigma_t= \Gamma_t(2\delta)\setminus \Gamma_t(\delta)$.
Because of \eqref{eq:DistEstim}, \eqref{eq:Matching1}, and \eqref{eq:Matching2}, one easily obtains
  \begin{alignat*}{1}
    \sup_{0\leq t\leq T_\eps}\left[\|c^{in}_0(t) -c_+^{out}
  \chi_+ -c_-^{out}\chi_-\|_{L^2(\Sigma_t)}
    + \eps^2\|c^{in}_2(t) \|_{L^2(\Sigma_t)}\right]&\leq Ce^{-\frac{\alpha\delta}{2\eps}}\leq C'\eps^{3\order},\\
    \sup_{0\leq t\leq T_\eps}\left[\|\nabla c^{in}_0(t) \|_{L^2(\Sigma_t)}
    + \eps^2\|\nabla c^{in}_2(t) \|_{L^2(\Sigma_t)}\right]&\leq C\frac{e^{-\frac{\alpha\delta}{2\eps}}}{\eps}\leq C'\eps^{3\order}
  \end{alignat*}
  for all $\eps\in (0,\eps_0]$ and
   some $C,C',\alpha>0$.  This shows \eqref{eq:MatchEstim1}-\eqref{eq:MatchEstim2}.  

To prove the first inequality in \eqref{eq:MatchEstim4}, we employ  \eqref{eq:DistEstim} and a change of variable to deduce
\begin{equation*}
\begin{split}
&{\sup_{0\leq t\leq T_\eps }\eps^3\|c^{in}_3\|_{L^2(\Sigma_t)}}
={\sup_{0\leq t\leq T_\eps }\eps^3\|\rho^{-k}\rho^kc^{in}_3\|_{L^2(\Sigma_t)}}\\
\leq &C \eps^k\delta^{-k}{\sup_{0\leq t\leq T_\eps }\eps^3\|\rho^kc^{in}_3\|_{L^2(\Sigma_t)}}\leq C_k \eps^k{\sup_{0\leq t\leq T_\eps }\eps^3\|\rho^k \tc_3(t)\|_{L^2(\T^1\times\R)}}.
\end{split}
\end{equation*}
 So the inequality  follows by applying  \eqref{yuning:1.38}.
 The proof of the second inequality in \eqref{eq:MatchEstim4} is done in the same way using additionally
  \begin{equation*}
    \|\rho^k\partial_\rho  \tc_3 \|_{L^4((0,T_\eps)\times \T^1;L^2(\R))}\leq C\left(\sup_{t\in [0,T_\eps], s\in\T^1}\|\partial_\rho \tc_3(\cdot,s,t,)\|_{L^2(\R)}\right)^{\frac12}\|\rho^{2k}\partial_\rho \tc_3\|_{L^2((0,T_\eps)\times \T^1\times \R)}^{\frac12}
  \end{equation*}
to estimate the leading term related to  $\partial_\rho \tc_3\nabla^\Gamma h_{2,\eps}$.
\end{proof}

\begin{lem}\label{lem:ModDiffEstim}
Let \eqref{assumptions}  hold true. Then there are some $\eps_1,T_1>0$ independent of $\eps, T_\eps, c_\eps, c_A$ such that
  \begin{subequations}\label{yuning:diffcaca1}
  \begin{align}
  \|c_\eps(t) -c_{A,0}(t)\|_{L^4(0,T_\eps;L^2(\Om))}+\|\nabla (c_\eps -c_{A,0})\|_{L^2(\Om\times (0,T_\eps)\setminus \Gamma(\delta)))} &\leq\frac32 R\eps^{\order+\frac12},\\
 \|\nabla_\btau(c_\eps -c_{A,0})\|_{L^2(\Omega\times (0,T_\eps)\cap \Gamma(2\delta))} + \eps\|\partial_\no(c_\eps -c_{A,0})\|_{L^2(\Omega\times (0,T_\eps)\cap \Gamma(2\delta))} &\leq \frac32 R\eps^{\order+\frac12}
\end{align}
\end{subequations}
and \eqref{eq:h2epsBound} holds true, where $M =M(R) $ is as in Lemma~\ref{lem:Existenceh2},
provided $\eps\leq \eps_1$ and $T_\eps\leq T_1$.
\end{lem}
\begin{proof}
 Using the triangle inequality, the proof of \eqref{yuning:diffcaca1} can be reduced to the corresponding estimate for $c_A-c_{A,0}$.
 For example, for the  first estimate, since
 \begin{equation*}
   \|c_\eps(t) -c_{A,0}(t)\|_{L^2(\Om)}
     \leq \|c_\eps(t) -c_{A}(t)\|_{L^2(\Om)}+\|c_A(t) -c_{A,0}(t)\|_{L^2(\Om)}
 \end{equation*}
 and because of the assumptions \eqref{assumptions},
 we only need to estimate the last term suitably. We also know from  \eqref{yuning:ca} and  \eqref{yuning:ca1} that
 \begin{equation*}
  c_A-c_{A,0}=(c^{in}(x,t)-\theta_0(\rho))\zg=(\eps^2c^{in}_2(x,t)+\eps^3c^{in}_3(x,t))\zg.
\end{equation*}
Now let  $M=M(R)$, $\eps_1$, and $T_1$  be as in Lemma~\ref{lem:Existenceh2}. Because of $h_{2,\eps}|_{t=0}=0$, there is some $T_\eps'\in (0,T_\eps]$ such that $h_{2,\eps}$ satisfies \eqref{eq:h2epsBound} with $T_\eps'$ instead of $T_\eps$.
Hence we can apply Lemma~\ref{lem:c3Estim} and Lemma~\ref{lem:MatchEstim} to conclude
\begin{alignat*}{1}
    \|c_A(t) -c_{A,0}(t)\|_{L^4(0,T'_\eps;L^2(\Om))}+\|\nabla (c_A -c_{A,0})\|_{L^2(\Om\times (0,T_\eps')\setminus \Gamma(\delta)))} &\leq C(R,\eps,T_\eps')\eps^{\order+\frac12},\\
  \|\nabla_\btau(c_A -c_{A,0})\|_{L^2(\Omega\times (0,T_\eps')\cap \Gamma(2\delta))} + \eps\|\partial_\no(c_A -c_{A,0})\|_{L^2(\Omega\times (0,T_\eps')\cap \Gamma(2\delta))} &\leq C(R,\eps,T_\eps')\eps^{\order+\frac12},
\end{alignat*}
where $C(R,\eps,T)\to_{(\eps,T)\to 0} 0$.  Choosing $\eps_1,T_1>0$ possibly even smaller we can achieve  that $C(R,\eps,T)\leq \frac{R}2$ provided $\eps\leq \eps_1$, $T\leq T_1$. This shows \eqref{yuning:diffcaca1} for some $T_\eps'\in (0, T_\eps]$ instead of $T_\eps$. In order to show the estimate for $T_\eps$, let
\begin{equation*}
  \tilde{T}_\eps := \sup\{ T_\eps'\in (0,T_\eps]: \text{\eqref{yuning:diffcaca1} holds true for } T_\eps'\text{ instead of }T_\eps \}.
\end{equation*}
The previous step implies $\tilde{T}_\eps>0$. Now assume that $\tilde{T}_\eps< T_\eps$. Then \eqref{yuning:diffcaca1} holds true for $\tilde{T}_\eps$ instead of $T_\eps$. Hence there is some $T_\eps'\in (\tilde{T}_\eps,T_\eps]$ such that \eqref{yuning:diffcaca1New} holds true with $T_\eps'$ instead of $T_\eps$.
Then Lemma~\ref{lem:Existenceh2} implies that $h_{2,\eps}$ satisfies \eqref{eq:h2epsBound} with $T_\eps'$ instead of $T_\eps$ and the first part of the proof shows \eqref{yuning:diffcaca1} for $T_\eps'$ instead of $T_\eps$ provided $T_\eps\leq T_1$ and $\eps\leq \eps_1$. This is a contradiction to the definition of $\tilde{T}_\eps$. Hence $\tilde{T}_\eps= T_\eps$ and \eqref{yuning:diffcaca1} holds true  provided $T_\eps\leq T_1$ and $\eps\leq \eps_1$.
\end{proof}

\begin{cor}\label{yuning:1.55}
 Let \eqref{assumptions} hold true for some $T_\eps\leq T_1$ and $\eps_0\leq \eps_1$, where $T_1,\eps_1$ are as in Lemma~\ref{lem:ModDiffEstim}. Then
  \begin{equation*}
    \left|\int_0^T\int_{\Om} \zg\tilde{\we}_2\cdot \nabla c^{in} (c_\eps-c_A)\sd x\sd t\right|\leq C(R,T,\eps)\eps^{2N+1}
  \end{equation*}
  for every $T\in (0,T_\eps]$ and $\eps\in (0,\eps_0]$,
where $C(R,T,\eps)$ is independent of $T_\eps$ and $\eps_0$ and $C(R,T,\eps)\to_{(T,\eps)\to 0} 0$.
\end{cor}
\begin{proof}
We shall denote $u=c_\eps-c_A $ as before.
We use
\begin{equation}
  \tilde{\we}_2 \cdot \nabla c^{in}=\tilde{\we}_2 \cdot \nabla ( \theta_0+\eps^2 c_2^{in}+\eps^3 c_3^{in})
\end{equation}
and estimate each term separately.\\[1ex]
{\bf Step 1:} We have
  \begin{alignat*}{1}
    \tilde{\we}_2 \cdot \nabla c_0^{in}=&
    \tilde{\we}_{2} \cdot\left(\no-\eps\nabla_\btau h_1(s,t)-\eps^2 \nabla_\btau h_{2,\eps}(s,t)\right)\tfrac1\eps \theta_0'\left(\tfrac{d_{\G } }\eps - h_\eps(s,t) \right) .
  \end{alignat*}
Moreover, we use that
  \begin{equation*}
    h_{2,\eps} \in X_{T_\eps}\hookrightarrow BUC([0,T_\eps];H^{\frac32}(\T^1))\hookrightarrow BUC([0,T_\eps];W^1_s(\T^1)),
  \end{equation*}
  is bounded with respect to $\eps\in (0,\eps_0]$
  for any $s\in (1,\infty)$.
Proposition \ref{yuning:velocityest} and H\"older's inequality imply that
  \begin{equation}\label{yuning:1.63}
    \left\| \zg\tilde{\we}_2 \cdot \tfrac1\eps  \theta_0'\left(\tfrac{d_\Gamma}\eps -  h_\eps \right)\left(\no-\eps\nabla_\btau h_\eps\right)\right\|_{L^r(0,T;L^{q})}\leq \tfrac{C}\eps\|\tilde{\we}_{2} \|_{L^r(0,T;L^{\tilde{q}})} \leq C(R)\eps^{\frac{4}r-1},
  \end{equation}
  where  $1< r<2$, $1<q<\tilde{q}<2$, $1<s<\infty$ and $\frac1q= \frac1s+\frac1{\tilde{q}}$.  In the following let $r',q'\in (1,\infty)$ be such that $\frac1r+\frac1{r'}=\frac1q+\frac1{q'}=1$.
Moreover, the Gagliardo-Nirenberg interpolation inequality and \eqref{assumptions} imply
  \begin{equation}\label{yuning:1.64}
    \|u\|_{L^{r'}(0,T;L^{q'})}\leq C\|u\|_{L^\infty(0,T;L^2)}^{1-\frac2{r'}}\|u\|_{L^2(0,T;H^1)}^{\frac{2}{r'}}\leq C(R)\eps^{N+\frac12-\frac{2}{r'}}= C(R)\eps^{\frac12+\frac2r}
  \end{equation}
  if we choose $q'\in (2,\infty)$ (and therefore $q\in (1,2)$) such that $\frac1{q'}= \frac12(1-\frac2{r'})$.
  Hence
  \begin{alignat*}{1}
        \left|\int_0^T\int_{\Om} \zg \tilde{\we}_2 \cdot\tfrac1\eps \theta_0'\left(\tfrac{d_{\G}}\eps - h_\eps(s,t) \right)\left(\no-\eps\nabla_\btau h_\eps(s,t)\right) u\, \sd x\sd t\right|\leq C(R)\eps^{\frac4r-1+\frac 2r+\frac12}=C(R)\eps^{\frac 6r-\frac 12}
  \end{alignat*}
  for all $T\in (0,T_\eps]$, $\eps\in (0,\eps_0]$.
If we choose now $r\in (1,\frac{12}{11})$, we have $\frac6r-\frac12> 5$ and therefore
  \begin{alignat*}{1}
        \left|\int_0^T\int_{\Om} \zg \tilde{\we}_2 \cdot\tfrac1\eps \theta_0'\left(\tfrac{d_{\G}}\eps - h_\eps(s,t) \right)\left(\no-\eps\nabla_\btau h_\eps(s,t)\right) u\sd x\sd t\right|\leq C(R,T,\eps)\eps^{2N+1},
  \end{alignat*}
where $C(R,T,\eps)\to_{(T,\eps)\to 0} 0$. 

\noindent
{\bf Step 2:}  Now we estimate $\eps^2 \int_0^T\int_{\Om} \zg\tilde{\we}_2\cdot \nabla c_2^{in} u\sd x\sd t.$
 Because of \eqref{yuning:1.16}, we obtain in a similar way as for \eqref{yuning:1.63} that
\begin{equation*}
  \|\zg\tilde{\we}_2\cdot \nabla c_2^{in}\|_{L^r(0,T;L^q)}= \left\|\zg\tilde{\we}_2 \cdot \((\tfrac{\no}\eps-\nabla_\btau h_\eps)\p_\rho\hat{c}_2+\nabla_\btau \hat{c}_2 \)\right\|_{L^r(0,T;L^{q})}  \leq C(R)\eps^{\frac{4}r-1}.
\end{equation*}
This together with \eqref{yuning:1.64} implies
\begin{equation*}
 \left|\eps^2 \int_0^T\int_{\Om} \zg\tilde{\we}_2\cdot \nabla c_2^{in} (c_\eps-c_A)\sd x\sd t\right|\leq C(R,T,\eps)\eps^{2N+1}.
\end{equation*}
{\bf Step 3:} We treat
 $$\eps^3 \int_0^T\int_{\Om} \zg\tilde{\we}_2\cdot \nabla c_3^{in} (c_\eps-c_A)\sd x\sd t.$$
In view of \eqref{yuning:1.65} and \eqref{eq:Matching}, we know that $|c_\eps-c_A|$ is uniformly bounded in $(x,t)$ and $\eps$. On the other hand, it follows from \eqref{yuning:1.38} and the Gagliardo-Nirenberg interpolation inequality used in  \eqref{yuning:1.64} that
\begin{equation*}
\|( \eps^\theta \partial_\rho \hat{c}_3, \eps\partial_s \hat{c}_3)\|_{L^4((0,T)\times \R\times \T^1)}  \leq \|( \eps^{2\theta} \partial_\rho \hat{c}_3, \eps\partial_s \hat{c}_3)\|_{L^\infty(0,T;L^2(\R\times \T^1))}^{\frac12}\|(\partial_\rho \hat{c}_3, \eps\partial_s \hat{c}_3)\|_{L^2(0,T;H^1(\R\times \T^1))}^{\frac12}
\end{equation*}
 is bounded, where we choose $\theta =\frac18$. Hence
\begin{equation*}
  \begin{split}
    &\| \nabla c_3^{in}\|_{L^{4}(0,T;L^{3}(\G_t(2\delta)))}=\|(\tfrac\no \eps-\nabla_\btau h_\eps) \p_\rho \hat{c}_3+\nabla_\btau \hat{c}_3\|_{L^{4}(0,T;L^{3}(\G_t(2\delta)))}\\
    &\leq \eps^{-1}\|\p_\rho\hat{c}_3\|_{L^{4}(0,T;L^{3}(\G_t(2\delta)))}+\|\nabla_\btau h_\eps \p_\rho \hat{c}_3 \|_{L^{4}(0,T;L^{3}(\G_t(2\delta)))}+\|\nabla_\btau \hat{c}_3\|_{L^{4}(0,T;L^{3}(\G_t(2\delta)))}\\
&\leq C\eps^{-1+\frac14-\theta}\|( \eps^\theta \p_\rho\hat{c}_3, \eps \p_s\hat{c}_3) \|_{L^{4}((0,T)\times \T^1\times \R)}\leq C(R)\eps^{-1+\frac18}.
  \end{split}
\end{equation*}
 Combining this with \eqref{eq:estofwc}, $r=q=\frac43$, leads to
\begin{equation*}
 \left|\eps^3 \int_0^T\int_{\Om} \zg\tilde{\we}_2\cdot \nabla c_3^{in} (c_\eps-c_A)\sd x\sd t\right|\leq C(R)\eps^{6-1+\frac18 }\leq C(R,T,\eps)\eps^{2N+1}
\end{equation*}
with $C(R,T,\eps)\to_{(T,\eps)\to 0} 0$.
\end{proof}

\noindent
\begin{proof*}{of Theorem~\ref{thm:ApproxSolutions}}  Again we denote $u=c_\eps-c_A$.
We first note that, since \eqref{assumptions} is assumed, it follows from Lemma~\ref{lem:ModDiffEstim} that both  \eqref{eq:h2epsBound} and \eqref{yuning:diffcaca1} and thus \eqref{yuning:diffcaca1New}  are valid if $\eps\leq \min\{\eps_0,\eps_1\}$ and $T_\eps \leq \min\{T_0,T_1\}$. So the assumptions for applying Proposition~\ref{yuning:velocityest} and Theorem~\ref{thm:InnerExpansion} are fulfilled.
 These together with \eqref{yuning:2.02} imply the   estimate of $\we_1$
       \begin{equation}\label{yuning:2.04}
         \left\|  \we_1\right\|_{L^2(0,T;H^1(\Om))}\leq  \eps^{-2}\(C(R)\eps^\order+C(M)(T^{\frac 14}+\eps^{\frac 12})\eps^\order\)= C(R) +C(M)(T^{\frac 14}+\eps^{\frac 12})
       \end{equation}
       for all $T\in (0,T_\eps]$, $\eps\in (0,\min\{\eps_0,\eps_1\}]$.
To proceed, we shall calculate
\begin{equation*}
  \partial_t c_A+ (\ve_A+\eps^2\we_2) \cdot \nabla c_A + \eps^2\we_1|_{\Gamma}\cdot \nabla c_{A,0}-\Delta c_A+\frac{f'(c_A)}{\eps^2}.
\end{equation*}
Using \eqref{yuning:ca} and Lemma~\ref{lem:MatchEstim}, we have the following asymptotics in $L^2(0,T_\eps;L^2(\Om))$
\begin{alignat*}{1}
    \Delta c_A&=  \zeta\circ d_\Gamma  \Delta c^{in}+ 2\nabla(\zeta\circ d_\Gamma) \cdot\nabla c^{in}+\Delta(\zeta\circ d_\Gamma)(c^{in}-c_+^{out}
  \chi_+ -c_-^{out}\chi_-)\\
  &= \zeta\circ d_\Gamma\Delta c^{in}+O(\eps^{\order+\frac12}),\\
    \p_t c_A&=\p_t(\zeta\circ d_\Gamma)(c^{in}-c_+^{out}
  \chi_+ -c_-^{out}\chi_-)+\zeta\circ d_\Gamma\p_t c^{in}
  = O(\eps^{\order+\frac32})+ \zeta\circ d_\Gamma\p_t c^{in}.
\end{alignat*}
It follows from  \eqref{eq:Matching1}, \eqref{eq:Matching2}, and \eqref{eq:Matching3} that
\begin{equation}\label{yuning:1.54}
|c^{in}-c_+^{out}
  \chi_+ -c_-^{out}\chi_-|\leq C(M)\eps^{N+\frac34}\quad \text{in } (\Gamma(3\delta)\setminus \Gamma(\delta))\cap (\Omega \times (0,T_\eps)).
\end{equation}
This together with \eqref{eq:Energy} and \eqref{yuning:2.04} implies
\begin{alignat*}{1}
    (\ve_A+\eps^2\we_2)\cdot\nabla c_A&= \zeta\circ d_\Gamma(\ve_A+\eps^2\we_2)\cdot\nabla c^{in} +(\ve_A+\eps^2\we_2)\cdot\nabla(\zeta\circ d_\Gamma)(c^{in}-c_+^{out}
  \chi_+ -c_-^{out}\chi_-)\\
  &= \zeta\circ d_\Gamma(\ve_A+\eps^2\we_2)\cdot\nabla c^{in} +\underbrace{(\ve_\eps-\eps^2\we_1)\cdot\nabla(\zeta\circ d_\Gamma)(c^{in}-c_+^{out}
  \chi_+ -c_-^{out}\chi_-)}_{=O(\eps^{N+\frac34})~\text{in}~L^2(0,T_\eps;L^2(\O))}.
\end{alignat*}
The trace estimate, \eqref{eq:Matching1}, and \eqref{eq:estofw} imply
\begin{equation*}
\begin{split}
  \eps^2\we_1|_{\Gamma}\cdot\nabla c_{A,0}&= \zeta\circ d_\Gamma \eps^2\we_1|_{\Gamma}\cdot\nabla  c^{in}_0+\eps^2\we_1|_{\Gamma}\cdot\nabla (\zg) (c^{in}_0-c_+^{out}
  \chi_+-c_-^{out}\chi_-)\\
  &= \zeta\circ d_\Gamma \eps^2\we_1|_{\Gamma}\cdot\nabla  c^{in}_0+O(\eps^{N+\frac34+N})\qquad \text{in}~ L^2(0,T_\eps;L^2(\O)).
\end{split}
\end{equation*}
To treat the bulk term we use  \eqref{yuning:ca} and \eqref{yuning:1.54} and obtain
\begin{equation*}
  \begin{split}
    \frac{1}{\eps^2}f'( c_A)&=\left\{
    \begin{array}{rl}
      \frac{1}{\eps^2}f'(  c^{in})&\quad \text{in }\Gamma(\delta),\\
      O(\eps^{\order+1})&\quad \text{in } \Gamma(2\delta)\setminus \Gamma(\delta),\\
      0&\quad   \text{in }\Om\times [0,T_\eps]\setminus \Gamma(2\delta),
    \end{array}\right.
  \end{split}
\end{equation*}
with respect to the $L^\infty(0,T_\eps; L^2)$-norm.
Actually, we only need to verify the case when $\delta <|d_\G|<2\delta$. 
With a Taylor expansion, \eqref{eq:MatchEstim1}, and \eqref{eq:MatchEstim4} we obtain
\begin{equation*}
  \begin{split}
    f'( c_A) &=f'(c_+^{out}
  \chi_++c_-^{out}\chi_-)+ f''(\xi_\eps (x,t))\left[  (c^{in}-c_+^{out}
  \chi_+ -c_-^{out}\chi_-)\zg\right]
  =0+O(\eps^{N+3})
  \end{split}
\end{equation*}
in $L^\infty(0,T_\eps;L^2(\Gamma_t(2\delta)\setminus \Gamma_t(\delta))$, where $\xi_\eps(x,t)$ is uniformly bounded.
In particular, we obtain
\begin{equation*}
  \tfrac1{\eps^2} f'(c_A)= \zeta\circ d_\Gamma \tfrac1{\eps^2} f'(c^{in}) + O(\eps^{\order+1})\quad \text{in}~L^\infty(0,T_\eps;L^2(\Om)).
\end{equation*}
   By collecting the above asymptotics, we arrive at
\begin{alignat}{1}\nonumber
 & \partial_t c_A  +(\ve_A+\eps^2\we_2)\cdot \nabla c_A+ \eps^2\we_1|_{\Gamma}\cdot \nabla c_{A,0}-\Delta c_A + \tfrac1{\eps^2} f'(c_A)\\\nonumber
&= \zeta\circ d_\Gamma\left(\partial_t c^{in}  + \ve^{in}_A \cdot \nabla c^{in}+ \eps^2\we_1|_{\Gamma} \cdot \nabla c_0^{in}-\Delta c^{in} + \tfrac1{\eps^2} f'(c^{in})\right)\\
&+\zeta\circ d_\G (\eps^2 \we_2+\ve_A-\ve_A^{in})\cdot\nabla c^{in}+ s_A,\label{eq:Localizing}
\end{alignat}
where $s_{A}=O(\eps^{\order+\frac12})$ in $L^2(0,T_\eps;L^2(\Om))$.
This together with Theorem~\ref{thm:InnerExpansion} leads to
\begin{alignat*}{1}
\partial_t c_A  &+(\ve_A+\eps^2\we_2)\cdot \nabla c_A+ \eps^2\we_1|_{\Gamma}\cdot \nabla c_{A}-\Delta c_A + \tfrac 1{\eps^2} f'(c_A)\\&= \zeta\circ d_\Gamma  \mathfrak{C} +\zeta\circ d_\G (\eps^2 \we_2+\ve_A-\ve_A^{in})\cdot\nabla c^{in}+\eps^2\we_1|_{\Gamma}\cdot \nabla (c_A-c_{A,0}) + s_A.
\end{alignat*}
Due to \eqref{yuning:chap2glue1}, we have
\begin{equation*}
     \ve_A-\ve_A^{in}=(\zg -1)( \ve_A^{in}-\ve_A^{+}
  \chi_+ -\ve_A^-\chi_- )
  =\sum_{i=0}^2(\zg -1)\eps^i( \ve_i-\ve_i^{+}
  \chi_+ -\ve_i^-\chi_- ).
\end{equation*}
This together with \eqref{eq:DistEstim} and Lemma \ref{lem:ExpDecay} implies that, the term $\zeta\circ d_\G (\ve_A-\ve_A^{in})\cdot\nabla c^{in}$  in \eqref{eq:Localizing} can be absorbed into $s_A$.
 Moreover, because of \eqref{assumptions}, Theorem~\ref{thm:InnerExpansion}  and Corollary \ref{yuning:1.55}, 
we obtain
\begin{alignat*}{1}
  &\int_0^T\left|\int_{\Om} \zeta\circ d_\G \(\mathfrak{C}+\eps^2 \we_2\cdot\nabla c^{in}\)(c_\eps-c_A)(x,t)\sd x \right|\sd t
\leq  C(R,T,\eps)\eps^{2N+1} 
\end{alignat*}
for all $\eps \in (0,\eps_0]$ and $T\in (0,\min(T_\eps,T_1)]$, where $C(R,T,\eps)\to_{(T,\eps)\to 0} 0$.

To show  \eqref{eq:rEstim}  it remains to estimate $\eps^2\we_1|_{\Gamma}\cdot \nabla (c_A-c_{A,0})$ since
\begin{equation*}
  \int_0^T\left| s_A(x,t)u(x,t)\sd x \right|\sd t \leq C\eps^{N+\frac12} \|u\|_{L^2(0,T;L^2)}\leq C(R)T^{\frac12}\eps^{2N+1}.
\end{equation*}
With the aid of Sobolev embeddings we obtain
\begin{equation*}
  \begin{split}
   & \int_0^T\int_{\Om}|u\eps^2\we_1|_{\Gamma}\cdot \nabla (c_A-c_{A,0}) |\sd x\sd t\\
&\leq
CT^{\frac14}\|\eps^2\we_1\|_{L^2(0,T;H^1)}\|\nabla (c_A-c_{A,0})\|_{L^\infty(0,T;L^2)}\|u\|_{L^4(0,T;L^4)}\\
    &\leq C_pT^{\frac14}\|\eps^2\we_1\|_{L^2(0,T;H^1)}\|\nabla (c_A-c_{A,0})\|_{L^\infty(0,T;L^2)}\|u\|_{L^\infty(0,T;L^2)}^{\frac12}\|u\|_{L^2(0,T;H^1)}^{\frac12}
  \end{split}
\end{equation*}
 due to $\|u\|_{L^4(\Om)}\leq C \|u\|_{L^2(\Om)}^{\frac12}\|u\|_{H^1(\Om)}^{\frac12}$. Because of 
\eqref{yuning:innerexpan'} and \eqref{yuning:ca}, we deduce
\begin{alignat*}{1}
\nabla (c_A-c_{A,0})&=\nabla(\zg)(c^{in}-c^{in}_0)+\zg\nabla (c^{in}-c^{in}_0)\\
&=\nabla(\zg)(\eps^2c_2^{in}+\eps^3c_3^{in})+\zg\nabla (\eps^2c_2^{in}+\eps^3c_3^{in}).
\end{alignat*}
Hence Lemma \ref{lem:c3Estim} yields
\begin{equation}\label{yuning:1.56}
\|\nabla (c_A-c_{A,0})\|_{L^\infty(0,T;L^2)}\leq C(R)\eps^{N-\frac12}.
\end{equation}
This together with \eqref{assumptions} and  \eqref{yuning:2.04}
implies
\begin{equation*}
  \begin{split}
    \int_0^T\int_{\Om}|u\,\eps^2\we_1|_{\Gamma}\cdot \nabla (c_A-c_{A,0}) |\sd x\sd t
    \leq C(R)\eps^{3N- 1/2 }= C(R)\eps^{2N+3/2}.
  \end{split}
\end{equation*}
Therefore the proof of Theorem~\ref{thm:ApproxSolutions} is finished.
\end{proof*}


%% file: SharpInterfaceLimitNSACDifference.tex
\section{Estimates of Approximate and Exact Solutions}
\label{sec:DifferenceEstimate}

\subsection{Estimates of the Error in the Velocity and the Remainder in the Linearization of $f'$}

We first recall that $c_{A,0}$ is defined in \eqref{yuning:ca1}, $u_1=c_\eps-c_{A,0}$ and   $\ve_\eps = \tilde{\ve}_A+\tilde{\we}_1  + \tilde{\we}_2, $ where $\tilde{\ve}_A$ is the solution of \eqref{train:stokes}
 and $\tilde{\we}_1, \tilde{\we}_2$ are as in Section~\ref{subsec:LeadingErrorVelocity}, i.e., solve \eqref{eq:w1a'}-\eqref{eq:wjBCs}. Moreover, we define
 \begin{equation*}
   \we_1:= \frac{\tilde{\ve}_A -{\ve}_A+\tilde{\we}_1}{\eps^2}\quad \text{and}  \quad  \we_2= \frac{\tilde{\we}_2}{\eps^2}.
 \end{equation*}
  Then $\ve_\eps = \ve_A + \eps^2(\we_1+\we_2)$. We note that, under the assumption \eqref{assumptions}, it follows from Lemma~\ref{lem:ModDiffEstim} that both  \eqref{eq:h2epsBound} and \eqref{yuning:diffcaca1} and thus \eqref{yuning:diffcaca1New}  are valid if $\eps\leq \min\{\eps_0,\eps_1\}$ and $T_\eps \leq \min\{T_0,T_1\}$. So the assumptions for applying Proposition~\ref{yuning:velocityest} and Theorem~\ref{thm:InnerExpansion} are fulfilled and we have the following estimate on $\we_1$:
\begin{equation}\label{yuning:2.05}
  \|\we_1\|_{L^2(0,T,H^1(\Om))}\leq C(R),\quad \|\we_1\|_{L^2(0,T,L^2(\Om))}\leq C(R)\left(T^{\frac14}+\eps^{\frac12}\right)
\end{equation}
   for all $T\in (0, \min (T_1,T_\eps) ]$, $\eps\in (0, \min(\eps_0,\eps_1) ]$ provided \eqref{assumptions} holds true.

\begin{lem}\label{eq:ConvTermEstim}
  Let $\eps_1,T_1$ be as in Lemma~\ref{lem:ModDiffEstim}. Then there  is some $C(R)$ such that
  under the assumptions \eqref{assumptions} we have for all $T\in (0, \min (T_1,T_\eps) ]$, $\eps\in (0, \min(\eps_0,\eps_1) ]$
  \begin{equation*}
    \begin{split}
      &{\eps^2}\int_0^T\left|\int_{\Gamma_t(2\delta)} \left(\tfrac1\eps({\we_{1,\no}-\we_{1,\no}}|_{\Gamma})\right)\theta_0'(\rho) (c_\eps-c_A)\sd x  \right|\sd t\leq C(R)\left(T^{\frac18}+\eps^{\frac14}\right)\eps^{2N+1}, \\
 &{\eps^2}\int_0^T\left|\int_{\Gamma_t(2\delta)}  \theta'_0(\rho)({\we_{1,\btau}-\we_{1,\btau}}|_{\Gamma})
 \cdot\nabla_\btau h_\eps(S(x,t),t)  (c_\eps-c_A)\sd x  \right|\sd t\leq C(R)\eps^{2N+\frac 32}.
\end{split}
  \end{equation*}
\end{lem}
\begin{proof}
For the sake of simplifying the presentation,  let us denote $u=c_\eps-c_A$ and {$\we =\eps^2{\we}_1$}.

\noindent
\emph{Proof of the first inequality:}
It follows from \eqref{yuning:1.23} and   $\Div \we=-\Div \ve_A$  that
  \begin{align*}
    &\int_{\Gamma_t(2\delta)} \tfrac1\eps(\we_\no-\we_\no|_{\Gamma})\theta_0'
    (\rho) u(x,t)\sd x \\
&=\int_{-2\delta}^{2\delta}\int_{\Gamma_t}\tfrac1\eps(\we_\no(r,p,t)-\we_\no(0,p,t))\theta_0'(\tfrac{r}\eps - h_\eps) u(r,p,t)J(r,p,t)\sd \sigma(p)\sd r \\
&=\int_{-2\delta }^{2\delta }\int_{\Gamma_t} \tfrac1\eps\int_0^r\partial_\no \we_\no(r',p,t)\sd r'\, \theta_0'(\tfrac{r}\eps - h_\eps) u(r,p,t)J(r,p,t)\sd \sigma(p)\sd r\\
&= - \int_{-2\delta}^{2\delta}\int_{\Gamma_t} \tfrac1\eps\int_0^{r}\Div  \ve_A(r',p,t)\sd r'\, \theta_0'(\tfrac{r}\eps - h_\eps) u(r,p,t)J(r,p,t)\sd \sigma(p)\sd r\\
&\quad - \int_{-2\delta}^{2\delta}\int_{\Gamma_t} \tfrac1\eps\int_0^{r}\Div_\btau  \we(r',p,t)\sd r'\, \theta_0'(\tfrac{r}\eps - h_\eps) u(r,p,t)J(r,p,t)\sd \sigma(p)\sd r \\
&= -\int_{-2\delta}^{2\delta}\int_{\Gamma_t} \tfrac1\eps\int_0^{r}\Div  \ve_A(r',p,t)\sd r'\, \theta_0'(\tfrac{r}\eps - h_\eps) u(r,p,t)J(r,p,t)\sd \sigma(p)\sd r\\
&\quad +\int_{-2\delta}^{2\delta}\int_{\Gamma_t} \tfrac1\eps\int_0^{r}\we(r',p,t)\sd r'\, \theta_0'(\tfrac{r}\eps - h_\eps)\cdot \nabla_\btau ( u(r,p,t)J(r,p,t))\sd \sigma(p)\sd r\\
&\quad+\int_{-2\delta}^{2\delta}\int_{\Gamma_t} \tfrac1\eps\int_0^{r}\we(r',p,t)\sd r'\cdot\nabla_\btau h_\eps(X_0^{-1}(p,t),t)\theta_0''(\tfrac{r}\eps - h_\eps)   u(r,p,t)J(r,p,t)\sd \sigma(p)\sd r\\
&\quad + \int_{-2\delta}^{2\delta}\int_{\Gamma_t} \tfrac1\eps\int_0^{r}\we_\no(r',p,t)\sd r'\, \theta_0'(\tfrac{r}\eps - h_\eps) u(r,p,t) \kappa(r,p,t)J(r,p,t)\sd \sigma(p)\sd r\\
&\qquad +O(e^{-\frac {\alpha\delta}{2\eps}})~\text{in}~L^2(0,T).
  \end{align*}
 Using Lemma~\ref{yuning:app1} one can show in a straight forward manner that
  \begin{equation*}
    |\Div  \ve_A(r',p,t)|\leq C(M)(\eps |r'|+ \eps^2)\left(\|h_{2,\eps}(t)\|_{C^1(\T^1)}+1\right)\quad \text{for all }r'\in (-2\delta,\delta), p\in\Gamma(t)
  \end{equation*}
and $t\in [0,\min (T_\eps,T_1)], \eps\in (0,\min(\eps_0,\eps_1)).$
Hence
\begin{alignat*}{1}
 & \left|\int_{-2\delta}^{2\delta}\int_{\Gamma_t} \tfrac1\eps\int_0^{r}\Div  \ve_A(r',p,t)\sd r'\, \theta_0'(\tfrac{r}\eps - h_\eps) u(r,p,t)J(r,p,t)\sd \sigma(p)\sd r\right|\\
&\leq C(M)
\int_{-2\delta}^{2\delta}\int_{\Gamma_t} (\eps^2+\eps |r|)|\theta_0'(\tfrac{r}\eps - h_\eps)| |u(r,p,t)|J(r,p,t)\sd \sigma(p)\sd r\left(\|h_{2,\eps}(t)\|_{C^1(\T^1)}+1\right)\\
&\leq C(M)\eps^{N+\frac12}\|u(t)\|_{L^2(\Omega)}\left(\|h_{2,\eps}(t)\|_{C^1(\T^1)}+1\right)
\end{alignat*}
Moreover, we  note that
\begin{equation}\label{yuning:1.24}
  \left|\tfrac1\eps \int_0^{r}\we  (r',p,t)\sd r'\right|\leq C\|\we(\cdot,p,t)\|_{L^\infty(-2\delta,2\delta)}\tfrac{|r|}{\eps}.
\end{equation}
 Hence we arrive at
  \begin{equation*}
    \begin{split}
     &\left|\int_{\Gamma_t(2\delta)} \tfrac1\eps(\we_\no-\we_\no|_{\Gamma})\theta_0'
    (\rho) u(x,t)\sd x\right| \\
&\leq  C(M)\eps^{N+\frac12}\|u(t)\|_{L^2(\Omega)}\left(\|h_{2,\eps}(t)\|_{C^1(\T^1)}+1\right)\\
&\quad + C\int_{-2\delta}^{2\delta}\int_{\Gamma_t} \|\we(\cdot,p,t)\|_{L^\infty(-2\delta,2\delta)}\(1+|( \partial_s  h_\eps)(X_0^{-1}(p,t),t)|\)\\&\qquad\qquad\qquad\ \tfrac{|r|}{\eps} \left(|\theta_0''(\tfrac{r}\eps - h_\eps)|+|\theta_0'(\tfrac{r}\eps - h_\eps)|\right)| (u,\nabla_\btau  u)(r,p,t)| \sd \sigma(p)\sd r + O(e^{-\frac {\alpha\delta}{2\eps}})\\
&\leq C(M)\eps^{N+\frac12}\|u(t)\|_{L^2(\Omega)}\left(\|h_{2,\eps}(t)\|_{C^1(\T^1)}+1\right)\\
&\quad + C\left\| \|\we(\cdot, P_{\Gamma_t}(\cdot) ,t)\|_{L^\infty(-2\delta,2\delta)}  \(1+ |( \partial_s  h_\eps)(S(\cdot,t),t)| \) |\tfrac{d_\Gamma}\eps| (|\theta_0'(\tfrac{d_\Gamma}\eps - h_\eps)|+|\theta_0''(\tfrac{d_\Gamma}\eps - h_\eps)|)\right\|_{L^2(\Gamma_t(2\delta))}\\
&\quad \cdot\left(\| \nabla_\btau  u\|_{L^2(\Gamma_t(2\delta))}+ \| u\|_{L^2(\Gamma_t(2\delta))} \right) + O(e^{-\frac {\alpha\delta}{2\eps}})\qquad  \text{in }L^2(0,T).
    \end{split}
  \end{equation*}
Since $h_\eps=h_1+\eps h_{2,\eps}$ is uniformly bounded   and
 \begin{equation*}
   \|f\|_{L^\infty(-2\delta,2\delta)}\leq C\|f\|_{L^2(-2\delta,2\delta)}^{\frac12}\|f\|_{H^1(-2\delta,2\delta)}^{\frac12},
 \end{equation*}
 applying the second inequality of  Corollary \ref{cor:repsEstim2} yields
\begin{equation*}
\begin{split}
  &\left\| \|\we(\cdot, P_{\Gamma_t}(\cdot) ,t)\|_{L^\infty(-2\delta,2\delta)}\(1+ |( \partial_s  h_\eps)(S(\cdot,t),t)| \)|\tfrac{d_\Gamma}\eps| (|\theta_0''(\tfrac{d_\Gamma}\eps - h_\eps)|+|\theta_0'(\tfrac{d_\Gamma}\eps - h_\eps)|)\right\|_{L^2(\Gamma_t(2\delta))}\\
  &\leq C  \eps^{\frac 12}\left\|  \|\we(\cdot, P_{\Gamma_t}(\cdot) ,t)\|_{L^\infty(-2\delta,2\delta)}  \(1+ |( \partial_s  h_\eps)( S(\cdot,t) ,t)| \) \right\|_{L^2(\G_t)}\\
  &\leq C \eps^{\frac 12}\(\|\we\|_{L^2(\Om)}^{\frac12}\|\we\|_{H^1(\Om)}^{\frac12}+\eps\|\we\|_{H^1(\O)}\|h_{2,\eps}\|_{W^1_4(\T^1)}\).
\end{split}
\end{equation*}
Combining the above estimates   with \eqref{assumptions}, \eqref{eq:h2epsBound}, and \eqref{yuning:2.05}, 
we conclude
  \begin{equation}\label{yuning:w1}
    \int_0^T\left|\int_{\Gamma_t(2\delta)} \left(\tfrac1\eps(\we_\no-\we_\no|_{\Gamma})\right)
    \theta_0'(\tfrac{d_\Gamma}\eps - h_\eps) u(x,t)\sd x \right|\sd t\leq C(M,R)\left(T^{\frac18}+\eps^{\frac 14}\right)\eps^{2N+1}.
  \end{equation}

\noindent
\emph{Proof of the second inequality:}
\begin{equation*}
    \begin{split}
 &    \int_{\Gamma_t(2\delta)}\theta'_0(\tfrac{d_\G}\eps - h_\eps)(\we_\btau-\we_\btau|_{\Gamma})
 \cdot\nabla_\btau h_\eps (S(x,t),t)  u(x,t)\sd x \\
 =& \int_{-2\delta}^{2\delta} \int_{\Gamma_t}\theta'_0(\tfrac{r}\eps - h_\eps)(\we_\btau(r,p,t)-\we_\btau(0,p,t))
 \cdot\nabla_\btau h_\eps (X_0^{-1}(p,t),t)  u(r,p,t)J(r,p,t)\sd \sigma(p)\sd r \\
 =& \int_{-2\delta}^{2\delta} \int_{\Gamma_t}\theta'_0(\tfrac{r}\eps - h_\eps)\underbrace{\int_0^r\partial_\no\we_\btau(r',p,t)\sd r'}_{|.|\leq \sqrt{r}\|\we(\cdot,p,t)\|_{H^1(-2\delta,2\delta)}}
 \cdot\nabla_\btau h_\eps (X_0^{-1}(p,t),t)  u(r,p,t)J(r,p,t)\sd \sigma(p)\sd r. \\
    \end{split}
  \end{equation*}
  This along with   $h_\eps\in X_{T_\eps}\hookrightarrow L^2(0,T_\eps;C^1(\T^1))$ 
leads to
    \begin{equation*}
    \begin{split}
 & \left|   \int_{\Gamma_t(2\delta)}(\we_\btau-\we_\btau|_{\Gamma})
 \cdot\nabla_\btau h_\eps (S(x,t),t)  u(x,t)\sd x \right|\\
 \leq &C \int_{-2\delta}^{2\delta} \int_{\Gamma_t}
  \sqrt{|r|}\left|\theta'_0(\tfrac{r}\eps - h_\eps)\right|\|\we(\cdot,p,t)\|_{H^1(-2\delta,2\delta)}
    |u(r,p,t)||\partial_s h_\eps(X_0^{-1}(p,t),t)|\sd \sigma(p) \sd r \\
= & C \int_{-2\delta}^{2\delta} \int_{\Gamma_t}
  \eps^{\frac 12}\sqrt{|\tfrac{r}\eps|}|\theta'_0(\tfrac{r}\eps - h_\eps)|\|\we(\cdot,p,t)\|_{H^1(-2\delta,2\delta)}
    |u(r,p,t)|\sd\sigma(p)\sd r \|h_\eps(\cdot,t)\|_{C^1(\T^1)} \\
    \leq  & C\eps^{\frac 12}\left\|
  \sqrt{|\tfrac{\cdot }\eps|}\theta'_0(\tfrac{\cdot}\eps - h_\eps)\|\we(\cdot, P_{\Gamma_t}(\cdot),t)\|_{H^1(-2\delta,2\delta)}\right\|_{L^2(\Gamma_t(2\delta))}
   \| u(\cdot,t)\|_{L^2(\Gamma_t(2\delta))} \|h_\eps(\cdot,t)\|_{C^1(\T^1)} \\
   \leq  & C\eps\|\we\|_{H^1(\Gamma_t(2\delta))}
   \| u\|_{L^2(\Gamma_t(2\delta))} (1+ \eps\|h_{2,\eps}\|_{C^1(\T^1)})
    \end{split}
  \end{equation*}
  In this last step, we employed Corollary \ref{cor:repsEstim2} to gain the factor $\eps^{\frac 12}$. Then the  second inequality follows by integrating the above estimate, \eqref{yuning:2.05}, and \eqref{assumptions}.
\end{proof}

Using the above lemma, we can prove the following important estimate:
\begin{lem}\label{lem:RuEstim}
 Let $\eps_1,T_1$ be as in Lemma~\ref{lem:ModDiffEstim}. Then there  is some $C(R)$ such that
  under the assumptions \eqref{assumptions}, we have
   \begin{equation*}
   \eps^2\int_0^T\left|\int_{\Om} (\we_1-\we_1|_{\Gamma})\cdot\nabla c_A ~(c_\eps-c_A)\sd x\right|\sd t\leq C(R)(T^{1/8}+\eps^{1/4})\eps^{2N+1}
 \end{equation*}
 provided $0<\eps\leq \min(\eps_0,\eps_1)$ and $0<T\leq \min(T_\eps,T_1)$.
\end{lem}
\begin{proof}
We shall denote $\mathcal{R}:=\eps^2(\we_1-\we_1|_{\Gamma})\cdot\nabla c_A$.
As before, it follows from Proposition~\ref{yuning:velocityest} and Lemma~\ref{lem:ModDiffEstim} that, under the assumption \eqref{assumptions} the estimate \eqref{eq:estofw} holds true.
We shall prove the statement by distinguishing the cases $|d|\leq \delta,~ \delta<|d|<2\delta$ and $|d|\geq 2\delta$.
According to the definition of $c_A$ in \eqref{yuning:ca} 
\begin{equation*}
   \int_0^T\int_{\Om\backslash\Gamma_t(2\delta)} \left|\mathcal{R} ~u\right|\sd x\sd t= 0,
 \end{equation*}
where $u=c_\eps-c_A$.
 For the integral over  the domain $\Gamma_t({2\delta})\backslash\Gamma_t(\delta)$, we  first note  that $\|\we_1\|_{L^2(0,T_\eps;H^1)}$ is uniformly bounded due to \eqref{yuning:2.05}. {Hence $\|\we_1-\we_1|_\Gamma\|_{L^2(0,T_\eps;L^4)}$ is uniformly bounded because of $H^{\frac12}(\Gamma_t)\hookrightarrow L^4(\Gamma_t)$ and $H^1(\Om)\hookrightarrow L^4(\Om)$. Hence \eqref{assumptions}, \eqref{yuning:1.56} and  Lemma \ref{lem:MatchEstim}  imply
 \begin{equation*}
   \int_0^T\int_{\Gamma_t({2\delta})\backslash\Gamma_t(\delta)} \left|\mathcal{R} ~u\right|\sd x\sd t\leq C\eps^2\|\nabla c_A\|_{L^\infty(0,T_\eps;L^2(\Gamma_t({2\delta})\backslash\Gamma_t(\delta))}\|u\|_{L^2(0,T_\eps; H^1(\Om\setminus \Gamma_t(\delta))}\leq C\eps^{2\order +2}
 \end{equation*}
for any $T\in (0,T_\eps]$.}
 So it remains to estimate the integral in $\Gamma_t(\delta)$. For any $x\in \Gamma_t(\delta)$, it follows from \eqref{yuning:ca1} that $c_{A,0}=c_0^{in}$ and thus we can decompose $\mathcal{R}$ into
\begin{alignat*}{1}
    \mathcal{R}  = &-\eps^2 (\we_1-\we_1|_\Gamma)\cdot\nabla(c_A-c_{A,0}) - \eps^2(\we_1-\we_1|_{\Gamma})\cdot\nabla c_{A,0} \\
  =&-\underbrace{\eps^2 (\we_1-\we_1|_\Gamma)\cdot\nabla(c_A-c_{A,0})}_{=:\mathcal{R}_1}- \underbrace{\tfrac{\theta'_0(\rho)}{\eps}\eps^2(\we_{1,\no}-\we_{1,\no}|_\Gamma)
   }_{=:\mathcal{R}_2}+\underbrace{\theta'_0(\rho)\nabla_\btau  h_\eps\eps^2\cdot(\we_{1,\btau}-\we_{1,\btau}|_\Gamma)}_{=:\mathcal{R}_3}.
\end{alignat*}
 The terms $\mathcal{R}_2$ and $\mathcal{R}_3$ are treated in Lemma \ref{eq:ConvTermEstim}.
For $\mathcal{R}_1$, we proceed in a similar way as in the proof of Theorem~\ref{thm:ApproxSolutions} using  Sobolev embeddings,
\begin{equation*}
  \begin{split}
   & \int_0^T\int_{\Gamma_t(\delta)}\left|\mathcal{R}_1 u\right|\sd x\sd t\leq C\|\eps^2\we_1\|_{L^2(0,T;H^1)}\|\nabla (c_A-c_{A,0})\|_{L^\infty(0,T;L^2)}\|u\|_{L^2(0,T;L^{4})}\\
    &\leq C\|\eps^2\we_1\|_{L^2(0,T;H^1)}\|\nabla (c_A-c_{A,0})\|_{L^\infty(0,T;L^2)}\|u\|_{L^\infty(0,T;L^2)}^{\frac12}\|u\|_{L^2(0,T;H^1)}^{\frac12}.
  \end{split}
\end{equation*}
In view of  \eqref{assumptions}, \eqref{yuning:2.05}, and \eqref{yuning:1.56}
\begin{equation*}
    \int_0^T\int_{\Gamma_t( \delta)}\left|\mathcal{R}_1 u\right|\sd x\sd t
     \leq C\eps^{3N- 1/2}= C\eps^{2N+3/2} .
\end{equation*}
\end{proof}

\begin{lem} \label{yuning:bulkest}
Under the assumptions \eqref{assumptions}, for every  $0\leq T\leq T_\eps$ and $\eps \in (0,\eps_0]$ we have
  \begin{equation*}
  \frac1{\eps^2}\left|\int_0^{T}\int_{\Om} \left[f'(c_\eps)-f'(c_A) - f''(c_A)(c_\eps-c_A)\right](c_\eps-c_A) \sd x\sd t \right|\leq C(R)(T^{\frac14}+\eps^{\frac12})\eps^{2N+1}
  \end{equation*}
  provided $\eps_0\in (0,1)$ is sufficiently small.
\end{lem}
The proof is based on:
\begin{lem}\label{yuning:1.48}
There is some $C>0$ such that for every $u\in H^1(\Om)$ and $t\in [0,T_0],\delta'\in (0,2\delta]$:
\begin{equation*}
  \|u\|_{L^3(\Gamma_t(\delta'))}^3\leq C\left(\|u\|_{L^2(\Gamma_t(\delta'))}+\|\partial_\no u\|_{L^2(\Gamma_t(\delta'))}\right)^{\frac12}\left(\|u\|_
  {L^2(\Gamma_t(\delta'))}+\|\nabla_\btau u\|_{L^2(\Gamma_t(\delta'))}\right)\|u\|_
  {L^2(\Gamma_t(\delta'))}^{\frac32}.
\end{equation*}
\end{lem}
\begin{proof}
  We use that
  \begin{align*}
    \|u\|_{L^3(\Gamma_t(\delta'))}^3 &\leq C \int_{-\delta'}^{\delta'}\int_{\Gamma_t} |u(p,r,t)|^3 \sd\sigma(p)\sd r=  C\left\| \|u(\cdot,r,t)\|_{L^3(\Gamma_t)} \right\|_{L^3(-\delta',\delta')}^3,
  \end{align*}
where $\|u\|_{L^3(\Gamma_t)}\leq C\|u\|_{H^1(\Gamma_t)}^{\frac13}\|u\|_{L^2(\Gamma_t)}^{\frac23}$ since $\Gamma_t$ is one-dimensional.
Hence H\"older's inequality together with Minkowski inequality imply
  \begin{align*}
    &\|u\|_{L^3(\Gamma_t(\delta'))}^3 \\
    &\leq C \left\| \|u(r,\cdot,t)\|_{H^1(\Gamma_t)} \right\|_{L^2(-\delta',\delta')}\left\| \|u(r,\cdot ,t)\|_{L^2(\Gamma_t)} \right\|_{L^4(-\delta',\delta')}^{2}\\
&\leq C \left\| \|u(r,\cdot ,t)\|_{H^1(\Gamma_t)} \right\|_{L^2(-\delta',\delta')}\left\| \|u(\cdot,p,t)\|_{L^4(-\delta',\delta')} \right\|_{L^2(\Gamma_t)}^2\\
&\leq C \left\| \|u(r,\cdot,t)\|_{H^1(\Gamma_t)} \right\|_{L^2(-\delta',\delta')}\left\| \|u(\cdot,p,t)\|_{L^2(-\delta',\delta')} \right\|_{L^2(\Gamma_t)}^{\frac32}\left\| \|u(\cdot,p,t)\|_{H^1(-\delta',\delta')} \right\|_{L^2(\Gamma_t)}^{\frac12}
  \end{align*}
since $\|f\|_{L^4(-\delta',\delta')}\leq C\|f\|_{L^2(-\delta',\delta')}^{\frac34}\|f\|_{H^1(-\delta',\delta'
)}^{\frac14}$ for all $f\in H^1(-\delta',\delta')$.
\end{proof}

\noindent
\begin{proof*}{of Lemma~\ref{yuning:bulkest}}
As before, we shall denote $u=c_\eps-c_A$ for simplicity.
  The estimate in $\Gamma(2\delta)$  follows from Lemma \ref{yuning:1.48}  and \eqref{assumptions}
  \begin{align*}
  &\int_0^T\|u\|_{L^3(\Gamma_t(2\delta))}^3\sd t\\
&\leq C\|(u,\partial_\no u)\|_{L^2(\Omega\times(0,T)\cap \Gamma(2\delta))}^{\frac12}
\|(u,\nabla_\btau u)\|_{L^2(\Omega\times(0,T)\cap \Gamma(2\delta))} \left(\int_0^T\|u(t)\|_{L^2(\Gamma_t(2\delta))}^6\right)^{\frac14}\\
&\leq C\|(u,\partial_\no u)\|_{L^2(\Omega\times(0,T)\cap \Gamma(2\delta))}^{\frac12}
\|(u,\nabla_\btau u)\|_{L^2(\Omega\times(0,T)\cap \Gamma(2\delta))} T^{\frac14}\sup_{0\leq t\leq T}\|u(t)\|_{L^2(\Gamma_t(2\delta))}^{\frac32}\\
&\leq CR^3T^{\frac14} \eps^{3N+1}
\end{align*}
   and in $\Omega \setminus\Gamma_t(\delta)$ we can use the Gagliardo-Nirenberg inequality in two dimensions and \eqref{assumptions}:
   \begin{equation*}
  \int_0^T\|u\|^3_{L^3(\O\backslash\Gamma_t(\delta))}\sd t\leq C\int_0^T\|\nabla u\|_{L^2(\Om\backslash\Gamma_t(\delta))}
  \|u\|_{L^2(\Om\backslash\Gamma_t(\delta))}^2\sd t\leq C( R\eps^{\order+\frac 12})^3.
\end{equation*}
The above two estimates together with the following formula imply the desired result:
  \begin{equation*}
    |[f'(a)-f'(b) - f''(a)(a-b)](a-b)|=\frac12|f'''((1-\theta)b+\theta a)(a-b)^3|\leq C |a-b|^3
  \end{equation*}
for all $a,b\in [-L,L]$, where $\theta=\theta(a,b)\in [0,1]$. Here we used that $$\sup_{x\in\O ,t\in [0,T_0],\eps \in (0,1]}|c_\eps(x,t)|\leq L:= \max (2,\sup_{\eps\in (0,1]}\|c_{\eps,0}\|_{L^\infty(\Om)})$$ cf. Remarks~\ref{rem:LInftyboundedness} and $$\sup_{x\in\O ,t\in [0,T_0]}|c_A(x,t)|\leq 2$$ for all $\eps\in (0,\eps_0]$ if $\eps_0$ is sufficiently small. The later assertion follows from   $$\lim_{\eps\to 0}\sup_{x\in\Om,t\in [0,T]}|c_A(x,t)|=1,$$ due to the construction of $c_A$.
\end{proof*}

\subsection{Proof of Theorem~\ref{thm:main}}


We first note that $h_1,h_2$, and $\hat{c}_3$ are determined by \eqref{yuning:h1equ'}, \eqref{yuning:h2equ'}, and \eqref{yuning:monst4'}, resp., with  initial data $(h_1,h_2,\hat{c}_3)\mid_{t=0}\equiv 0$. This together with the construction of $c_A$, more precisely \eqref{yuning:innerexpan'} and \eqref{yuning:ca} implies that
\begin{equation*}
  \begin{split}
    c_{A}(x,0)&=\zeta(d_{\Gamma_0}(x))\(\theta_0(\tfrac{d_{\Gamma_0}(x)}\eps)
    +\eps^2\hat{c}_2(\tfrac{d_{\Gamma_0}(x)}\eps,S(x,0),0)\)\\
    &\qquad +(1-\zeta(d_{\Gamma_0}(x)))\left(
  \chi_{\Omega^+(0)}(x)-\chi_{\Omega^-(0)}(x)\right).
  \end{split}
\end{equation*}
Since the equation \eqref{yuning:c2equ'} that  determines $\hat{c}_2$ has a smooth solution, we have
\begin{equation*}
  c_A(x,0)= c_{A,0}^0(x)+ O(\eps^{2+\frac12}) \qquad \text{in } L^2(\Om),
\end{equation*}
where we gained a factor $\sqrt{\eps}$ through the same change of variable as before.
This together with the assumption that $\|c_{\eps,0}-c_{A,0}^0\|_{L^2(\Om)}\leq C\eps^{\order+\frac12}$  implies the existence of  some $R\geq 1$ such that
  \begin{equation}\label{yuning:initialdata}
\sup_{\eps\in (0,1]}  \|c_{0,\eps}-c_A|_{t=0}\|_{L^2(\Om)}^2 \leq  \frac{R^2}4 \eps^{2\order+1}.
\end{equation}
In the following let $c_A$
be as in Theorem~\ref{thm:ApproxSolutions} with $R$ as determined before. We consider the validity of
 \begin{subequations}\label{eq:EstimT'1}
    \begin{align}
 \sup_{0\leq t\leq \tau} \|c_\eps(t) -c_A(t)\|_{L^2(\Om)}^2 +\|\nabla (c_\eps -c_A)\|_{L^2(\Om\times (0,\tau)\setminus\Gamma(\delta))}^2 &\leq \frac{R^2}2\eps^{2\order+1},\\
 \|\nabla_{\btau}(c_\eps -c_A)\|_{L^2(\Om\times (0,\tau)\cap \Gamma(2\delta))}^2+\eps^2\|\partial_\no(c_\eps -c_A)\|_{L^2(\Om\times (0,\tau)\cap \Gamma(2\delta))}^2 &\leq \frac{R^2}2\eps^{2\order+1}
\end{align}
 \end{subequations}
for some $\tau=\tau(\eps)\in (0,T_0]$ and all $\eps\in (0,\eps_0]$ for sufficiently small $\eps_0\in (0,1)$, which imply \eqref{assumptions'} 
for $T=\tau$. 
Since the statement of Theorem~\ref{thm:main} is for sufficiently short time and small $\eps$, we can assume that $\First{\tau}\leq T_1$ and $\eps_0\leq \eps_1$, where $T_1,\eps_1$ are as in Lemma~\ref{lem:ModDiffEstim}. Hence \eqref{eq:h2epsBound} and \eqref{yuning:diffcaca1} hold true as well.
In the following we will prove that \eqref{eq:EstimT'1} 
remain valid as long as $\tau<T$ and $\eps\in (0,\eps_0]$ for some $T\in (0,T_0]$ independent of $\eps$ and some sufficiently small $\eps_0\in (0,1)$.

Now we define
 \begin{equation*}
   T_\eps:=\sup\left\{\tau\in [0,T_0]: \text{\eqref{eq:EstimT'1} 
holds true.}\right\}.
 \end{equation*}
Because of \eqref{yuning:initialdata}, since $c_\eps,c_A$ are smooth, and since $\|c_\eps(t) -c_A(t)\|_{L^2(\Om)}$ is continuous in $t\in [0,T_0]$, we have $T_\eps>0$.
  It remains to   show that $T_\eps$ has a positive lower bound that is independent of $\eps$.
To this end, we apply Theorem~\ref{thm:ApproxSolutions} and obtain
 \begin{equation}\label{eq:ACApprox}
\partial_t c_A +(\ve_A+\eps^2\we_2)\cdot \nabla c_A+ \eps^2\we_1|_{\Gamma}\cdot \nabla c_A-\Delta c_A + \eps^{-2} f'(c_A)
  = r_A
\end{equation}
where $\ve_\eps = \ve_A +\eps^2(\we_1+\we_2)$ and
\begin{equation}\label{yuning:assump2}
\int_0^T\left|\int_\Omega r_A(x,t)(c_\eps(x,t)-c_A(x,t))\sd x\right|\sd t \leq M_R(\eps,T) \eps^{2\order+1}
\end{equation}
for all $T\in (0,T_\eps]$ and $\eps\in (0,\eps_0]$,
where $M_R(\eps,T)\to_{(\eps,T)\to 0} 0$. For the following let $u=c_\eps -c_A$ and $\mathcal{L}_\eps = -\Delta +\frac1{\eps^2} f''(c_A)$. Then we have
  \begin{equation*}
    \ve_\eps\cdot\nabla c_\eps-(\ve_A+\eps^2 \we_2)\cdot\nabla c_A=\ve_\eps\cdot\nabla u+\eps^2\we_1\cdot\nabla c_A.
  \end{equation*}
  Subtracting \eqref{eq:ACApprox} from \eqref{eq:NSAC3}  and substituting the latter formula leads to
  \begin{equation}\label{yuning:diff1}
    \partial_t u + \ve_\eps\cdot \nabla u + \mathcal{L}_\eps u  =- r_\eps(c_\eps,c_A) -r_A +\mathcal{R}
  \end{equation}
where  $\mathcal{L}_\eps$ is as in \eqref{eq:Leps} and
\begin{alignat*}{1}
r_\eps (c_\eps,c_A)&= \frac1{\eps^2}\left(f'(c_\eps)-f'(c_A)- f''(c_A)(c_\eps-c_A)\right),\\
\mathcal{R}&= - \eps^2\we_1\cdot\nabla c_A+\eps^2\we_1|_{\Gamma}\cdot\nabla c_A.
\end{alignat*}
Taking the inner product of \eqref{yuning:diff1} with $u$ in $L^2(\Om)$  yields
\begin{alignat}{1}\label{eq:RemEstim}
  & \frac12\|u(t)\|_{L^2(\Om)}^2 + \int_0^t\int_{\Om} \left(|\nabla u|^2 + \frac{f''(c_A)}{\eps^2}u^2\right)\sd x \sd s\\\nonumber
& \leq   \int_0^t\int_{\Om} |r_\eps (c_\eps,c_A) u| \sd x\sd s  +  \int_0^t\left|\int_{\Om}r_A~u\sd x\right|\sd s    + \int_0^t \left|\int_{\Om}\mathcal{R}  ~u\sd x\right|\sd s+ \frac12\|u(0)\|_{L^2(\Om)}^2
\end{alignat}
for all $t\in [0,T_\eps]$.
Now using Theorem~\ref{thm:Spectral}, \eqref{yuning:initialdata}, and \eqref{yuning:assump2} we obtain
\begin{alignat}{1}\nonumber
     & \sup_{0\leq s \leq t}\|u(s,\cdot)\|_{L^2(\Om)}^2+\|\nabla_\btau u\|_{L^2(\Om\times(0,t)\cap \Gamma(2\delta))}^2+\|\nabla  u\|_{L^2(\Om\times (0,t)\setminus\Gamma(\delta))}^2 \\\nonumber
& \leq  C\int_0^t\|u(s,\cdot)\|^2_{L^2(\Om)}\sd s +\int_0^t\int_{\Om} |r_\eps (c_\eps,c_A) u| \sd x \sd s
     + \int_0^t\left|\int_{\Om} \mathcal{R}  ~u\sd x\right|\sd s\\\label{yuning:afterspectral}
&\quad  + \left(\frac{R^2}4+M_R(\eps,t)\right)\eps^{2N+1}
\end{alignat}
for all $t\in [0,T_\eps]$.
 Here we used that $\|\nabla_\tau u\|_{L^2(\Gamma_t(2\delta))}+\|\nabla u\|_{L^2(\Om\setminus \Gamma_t(\delta))}$ is equivalent to $\|\nabla_\tau u\|_{L^2(\Gamma_t(\delta))}+\|\nabla u\|_{L^2(\Om\setminus \Gamma_t(\delta))}$.
 So we can apply  Gronwall's inequality on the interval $[0,T]$ with $T\leq T_\eps$ and Young's inequality to obtain
\begin{alignat}{1}\label{yuning:firstorderest}
     &\sup_{0\leq t \leq T }\|u(t)\|^2_{L^2(\Om)}+\|\nabla_\btau u\|_{L^2(\Om\times(0,T)\cap \Gamma(2\delta))}^2+\|\nabla  u\|_{L^2(\Om\times (0,T)\setminus\Gamma(\delta))}^2\\\nonumber
&\leq  e^{CT}\left(\int_0^T\int_{\Om} |r_\eps (c_\eps,c_A) u| \sd x\sd t +\int_0^T  \left|\int_{\Om} \mathcal{R}  ~u\sd x\right|\sd t + \left(\frac{R^2}4 +M_R(\eps,T)\right)\eps^{2N+1} \right).
\end{alignat}
By the definition of $T_\eps$, the first two terms on the right-hand side  can be estimated using Lemma~\ref{lem:RuEstim} and Lemma~\ref{yuning:bulkest}:
\begin{alignat}{1}\nonumber
    &e^{CT}\left(\int_0^T\int_{\Om} |r_\eps (c_\eps,c_A) u| \sd x\sd t    +\int_0^T  \left|\int_{\Om} \mathcal{R}  ~u\sd x\right|\sd t+ \left(\frac{R^2}4 +M_R(\eps,T)\right)\eps^{2N+1}\right)\\\label{yuning:loworder1}
    &\leq  e^{CT}\left(C(R)(T^{1/8}+\eps^{1/4})\eps^{2N+1}    + \left(\frac{R^2}4 +M_R(\eps,T)\right)\eps^{2N+1}\right)\leq \frac{R^2}3\eps^{2N+1}
\end{alignat}
for all $\eps \in (0,\eps_0]$ provided $T\leq \min(T_1,T_\eps)$ and $T_1>0$, $\eps_0\in (0,1]$ are sufficiently small.
This together with \eqref{yuning:firstorderest} implies
\begin{equation*}
    \sup_{0\leq t\leq T} \|u(t)\|^2_{L^2(\Om)} + \|\nabla_\btau u\|^2_{L^2(\Om\times(0,T)\cap \Gamma(2\delta))}  + \|\nabla u\|^2_{L^2(\Om\times(0,T) \setminus(\Gamma(\delta)))}
    \leq  \frac{R^2}3\eps^{2N+1} < \frac{R^2}2\eps^{2N+1}
\end{equation*}
for all $T\in (0,\min(T_1,T_\eps))$ and $\eps\in (0,\eps_0]$ if $T_1\in (0,T_0]$ and $\eps_0$ are sufficiently small.
To complete the proof of \eqref{eq:EstimT'1} with strict inequality, it remains  to estimate $\p_\no u(x,t)$. To this end we use \eqref{eq:RemEstim}, $\inf_{s\in\R}f''(s)>-\infty $, \eqref{yuning:loworder1}  and the previous  estimates:
\begin{equation*}
    \begin{split}
    &\eps^2\|\p_\no (c_\eps-c_A)\|^2_{L^2(\Om\times(0,T)\cap \Gamma(2\delta)))}
    \leq   \eps^2\|\nabla u\|^2_{L^2(0,T;L^2(\Om))}\\
    &\leq \int_0^T\int_{\Om} \(\eps^2|\nabla u|^2  +f''(c_A) u^2\)\sd x\sd t-\inf_{s\in\R}f''(s)\int_0^T\int_{\Om}   u^2\sd x\sd t\\
&  \leq \frac{R^2}3\eps^{2N+3}+ CT\sup_{0\leq t\leq T}\|u(t)\|_{L^2(\Omega)}^2\\
&  \leq   \frac{R^2}3\eps^{2N+3}+CT\frac{R^2}2\eps^{2N+1}< \frac{R^2}6\eps^{2N+1}
    \end{split}
\end{equation*}
for all $0<T\leq \min(T_1,T_\eps)$ and $0<\eps\leq \eps_0$ provided $T_1\in (0,T_0]$ and $\eps_0\in (0,1]$ are sufficiently small.

Altogether we obtain \eqref{eq:EstimT'1} with $\tau=\min(T_1,T_\eps)$ and strict inequality for all $\eps\in (0,\eps_0]$, provided $T_1\in (0,T_0]$ and $\eps_0 \in (0,1]$ are sufficiently small. Because of the definition of $T_\eps$ and since the norms in  \eqref{eq:EstimT'1} depend continuously on $\tau$, this implies $T_\eps>  T_1$ for all $\eps \in (0,\eps_0]$. Hence the estimate \eqref{assumptions'} of Theorem~\ref{thm:main} is proved.

Now \eqref{eq:convVelocityb} follows directly from Proposition~\ref{yuning:velocityest} and Theorem~\ref{thm:ApproxVA} since $\ve_\eps-\ve_A=\tilde{\ve}_A -\ve_A +\tilde{ \we}_1+\tilde{\we}_2$.
Finally, $\lim_{\eps\to 0}\sup_{x\in\Om,t\in [0,T]}|c_A(x,t)|=1$ follows easily from the construction of $c_A$ as well as
 \eqref{yuning:1.38}.
The last statement on convergence of $\ve_A$ follows easily from the constructions.
This finishes the proof of  Theorem~\ref{thm:main}.


%% file: SharpInterfaceLimitNSACAppendix.tex
\subsection{Formally Matched Asymptotics for Stokes System}
The main task of this and the next subsection is to prove the Lemma \ref{yuning:app1}. One approach is to plug \eqref{yuning:summery1} and \eqref{yuning:summery2} into \eqref{yuning:innerd} and to verify. However, in order to show the reader how we obtained these formula, we shall first assume some ansatz and see what kind of equations these terms in the expansion should satisfy. To this end we will construct an approximate solution with the inner expansion
\begin{equation}\label{yuning:1.66}
  \begin{split}
    \ve_A^{in}(\rho,x,t)&=\ve_0(\rho,x,t)+\eps \ve_1(\rho,x,t)+\eps^2\ve_2(\rho,x,t),\\
p_A^{in}(\rho,x,t)&=\eps^{-1}p_{-1}(\rho,x,t)+p_0(\rho,x,t)+\eps p_1(\rho,x,t),
  \end{split}
\end{equation}
where we assume that $(\ve_j,p_j)$ are given by
\begin{equation}\label{yuning:ansatz1}
  \begin{split}
    \ve_j (\rho,x,t)&= \tilde{\ve}_j(\rho,x,t) + d_\Gamma(x,t) \hat{\ve}_j(x,t)\eta(\rho),\quad 0\leq j\leq 2,\\
  p_j (\rho,x,t)&= \tilde{p}_j(\rho,x,t) + d_\Gamma(x,t) \hat{p}_j(x,t)\eta(\rho),\quad-1\leq j\leq 1.
  \end{split}
\end{equation}
where $\eta(\rho)$ satisfies \eqref{eq:eta}.
Moreover, we assume that \eqref{yuning:matching4} holds, which implies the matching conditions
\begin{alignat}{2}\label{yuning:matching2}
    & \lim_{\rho\to\pm\infty} \ve_j(\rho,x,t)=\ve_j^\pm(x,t),& \lim_{\rho\to\pm\infty}\p_\rho \tilde{\ve}_j(\rho,x,t)&=\lim_{\rho\to\pm\infty}\p_\rho \ve_j(\rho,x,t)=0,\quad 0\leq j\leq 2\\
\label{yuning:matching3}
    &\lim_{\rho\to\pm\infty} p_j(\rho,x,t)=p_j^\pm(x,t),& \lim_{\rho\to\pm\infty}\p_\rho \tilde{p}_j(\rho,x,t)&=\lim_{\rho\to\pm\infty}\p_\rho p_j(\rho,x,t)=0,\quad -1\leq j\leq 1.
\end{alignat}
for all $(x,t)\in \Gamma(3\delta)$.
For the preceding analysis it will be sufficient to solve the first equation in \eqref{train:stokes} up to order $O(\eps)$ and the divergence equation up to order $O(\eps^2)$. This leads to the outer expansion to satisfy $p^\pm_{-1}$ being constant and
\begin{equation}\label{yuning:1.68}
  -\Delta \ve_j^\pm+\nabla p_j^\pm=0,\quad \Div \ve_j^\pm=0\quad \text{in}~\O^\pm(t)~\text{for}~j=0,1.
\end{equation}
After determining  $\ve^\pm_j$, they  are extended smoothly from $\Omega^\pm(t)$ to $\Omega$ such that $\Div \ve^\pm_j=0$ is preserved. The existence of such an extension can be seen as in  the discussion before \eqref{yuning:firstlimit}.

The general routine is: By matching terms with the same powers of $\eps$ of the divergence equation, we get $\tilde{\ve}_{k,\no}$ and take the normal component in the Stokes equation to get $\tilde{p}_{k-1}$ and finally use this to solve the Stokes equation to get $\tilde{\ve}_{k}$.
\noindent
\subsubsection{Divergence equation for $\ve_0$:}
In the following we will use:
\begin{lem}\label{yuning:compulemma1}
Let $k=0,1$. If $  \tilde{\ve}_k$ is independent of $\rho$,
then   the following formula holds
\begin{equation*}
\tilde{\ve}_{k} = \tfrac12 \(\ve_{k}^+ +\ve_{k}^-\)\text{ on }\Gamma(3\delta),\quad   \hat{\ve}_{k} = \tfrac1{2d_\Gamma} \(\ve_{k}^+ -\ve_{k}^-\)\text{ on }\Gamma(3\delta)\setminus \Gamma.
\end{equation*}
Moreover, if $\ve_{k}^\pm \in C^{\infty}(\overline{\Omega})$   and  $\ve_{k}^+=\ve_{k}^-$ on $\Gamma$, then  we can define
\begin{equation*}
  \hat{\ve}_k:=\tfrac 12\p_\no (\ve_k^+-\ve_k^-)\quad \text{on }\Gamma_t
\end{equation*}
and   $\hat{\ve}_k\in C^{\infty}(\overline{\Omega})$. The same statements also hold for  the normal and tangential components of $\hat{\ve}_k$.
\end{lem}
\begin{proof}
To prove the first part, since $\tilde{\ve}_{k}(\rho,x,t)$ is independent of $\rho$, we take $\rho\to\pm\infty$ in the first equation of \eqref{yuning:ansatz1} and use \eqref{yuning:matching2} to  conclude
 \begin{equation*}
   \tilde{\ve}_k(\rho,x,t) \pm d_\Gamma(x,t) \hat{\ve}_k(x,t) =\ve_k^\pm(x,t)\quad \text{for all }(x,t)\in\Gamma(3\delta).
 \end{equation*}
 Solving the linear equations leads to the first statement. The proof of the
last statement follows from a Taylor expansion with respect to $d_\Gamma$ as described in the end of Subsection~\ref{subsec:Coordinates}.
\end{proof}

Moreover, we have for $j=0,1,2$
\begin{equation*}
  \begin{split}
    &\Div\left( \ve_j(\tfrac{d_\Gamma}\eps-h_\eps, x,t)\right) = \tfrac1\eps\partial_\rho \tilde{\ve}_{j,\no}(\rho,x,t) -\nabla^\G  h_\eps(d_\Gamma,s,t)\cdot\partial_\rho\tilde{\ve}_{j,\btau}(\rho,x,t)   \\
  &\quad + \tfrac{d_\Gamma}\eps  \eta'(\rho)\hat{\ve}_{j,\no}(x,t) -d_\Gamma\eta'(\rho)
  \nabla^\G h_\eps(d_\Gamma,s,t)\cdot\hat{\ve}_{j,\btau}(x,t) +\Div_x\(\tilde{\ve}_j (\rho,x,t)+ \hat{\ve}_j(x,t)d_\Gamma \eta(\rho)\)
  \end{split}
\end{equation*}
Therefore it follows that
\begin{equation} \label{yuning:diver1}
  \begin{split}
  &\Div\left( \ve_A^{in}(\tfrac{d_\Gamma}\eps-h_\eps, x,t)\right)= \tfrac1\eps\partial_\rho \tilde{\ve}_{0,\no}(\rho,x,t) \\
  &\quad-\nabla^\G  h_\eps\cdot\partial_\rho\tilde{\ve}_{0,\btau}(\rho,x,t) + (\rho+h_\eps)  \eta'(\rho)\hat{\ve}_{0,\no}(x,t)\\
&\quad -d_\Gamma\eta'(\rho)
  \nabla^\G  h_\eps\cdot\hat{\ve}_{0,\btau}(x,t)+\Div_x\(\tilde{\ve}_0 (\rho,x,t)+ \hat{\ve}_0(x,t)d_\Gamma \eta(\rho)\) + \partial_\rho \tilde{\ve}_{1,\no}(\rho,x,t)\\
  &\quad-\eps\nabla^\G  h_\eps\cdot\partial_\rho\tilde{\ve}_{1,\btau}(\rho,x,t) + \eps(\rho+h_\eps)  \eta'(\rho)\hat{\ve}_{1,\no}(x,t)  -\eps d_\Gamma\eta'(\rho)
  \nabla^\G  h_\eps\cdot\hat{\ve}_{1,\btau}(x,t) \\
  &\quad+\eps\Div_x\(\tilde{\ve}_1 (\rho,x,t)+ \hat{\ve}_1(x,t)d_\Gamma \eta(\rho)\)  + \eps\partial_\rho \tilde{\ve}_{2,\no}(\rho,x,t)\\
  &\quad-\eps^2\nabla^\G  h_\eps\cdot\partial_\rho\tilde{\ve}_{2,\btau}(\rho,x,t) + \eps^2(\rho+h_\eps)  \eta'(\rho)\hat{\ve}_{2,\no} (x,t) \\
&\quad -\eps^2 d_\Gamma\eta'(\rho)
  \nabla^\G  h_\eps\cdot\hat{\ve}_{2,\btau}(x,t) +\eps^2\Div_x\(\tilde{\ve}_2 (\rho,x,t)+ \hat{\ve}_2(x,t)d_\Gamma \eta(\rho)\).
  \end{split}
\end{equation}
Hence equating terms with the same power of $\eps$ in the expansion of $\Div \ve_A^{in}$, we obtain at order $O(\tfrac1\eps)$ that
\begin{equation}\label{yuning:v0nvanish}
  \partial_\rho \tilde{\ve}_{0,\no} (\rho,x,t)=0 \qquad \text{for all }\rho\in\R, (x,t)\in\G(3\delta)
\end{equation}
 and this together with Lemma \ref{yuning:compulemma1} implies
\begin{equation}\label{yuning:1.52}
  \tilde{\ve}_{0,\no}= \tfrac12(\ve^+_{0,\no}(x,t)+\ve^-_{0,\no}(x,t))\text{ in } \Gamma(3\delta),\
\hat{\ve}_{0,\no}= \tfrac1{2d_\Gamma}(\ve^+_{0,\no}(x,t)-\ve^-_{0,\no}(x,t))\text{ in }\G(3\delta)\backslash\G,
\end{equation}
where $\ve^\pm_0$ are extended smoothly from $\Omega^\pm(t)$ to $\Omega$ such that $\Div \ve^\pm_0=0$ is preserved. So we can define:
\begin{equation*}
  \hat{\ve}_{0,\no}= \tfrac12\p_\no (\ve^+_{0,\no}- \ve^-_{0,\no})~\text{on}~ \G.
\end{equation*}
For the expansion of the Stokes system, we obtain:
 \begin{lem}
 If  we relate $\rho=\frac {d_\G}\eps-h_1-\eps h_2$ in \eqref{yuning:1.66}, then  expanding
\begin{equation}\label{yuning:1.67}
    -\Delta\ve_A^{in}+\nabla p_A^{in}\quad \text{and}\quad -\eps \Div (\nabla c^{in}_0\otimes \nabla c^{in}_0)
  \end{equation}and equating the terms with the same power of $\eps$ yields:
  \begin{alignat}{1}
    O(\tfrac{1}{\eps^2}):&-\p_\rho^2\tilde{\ve}_0+\p_\rho \tilde{p}_{-1}\no=-\p_\rho((\theta_0'(\rho))^2)\no\label{train:scal-a}
   \end{alignat}
   and, if $\partial_\rho \tilde{\ve}_0= \nabla_x \tilde{p}_{-1}=0$ and $ \hat{p}_{-1}=0$,
 \begin{alignat}{1}\nonumber
     O(\tfrac{1}{\eps}):&-\p_\rho^2 \tilde{\ve}_1-\nabla^\G  h_1\p_\rho \tilde{p}_{-1}+\no\p_\rho \tilde{p}_0 -(\rho+h_1)\hat{\ve}_0\eta''(\rho)
     -2\eta'(\rho)\partial_\no(d_\Gamma\hat{\ve}_0)\\
& \quad =\p_\rho(\theta_0'(\rho))^2\nabla^\G  h_1-(\theta_0'(\rho))^2\no \Delta d_\Gamma.\label{train:scal-b}
   \end{alignat}
  If additionally $\partial_\rho \tilde{\ve}_1=0$  and $(\p_\rho\tilde{p}_1,\p_\rho\tilde{\ve}_2,\p_\rho^2 \tilde{\ve}_2)\in\mathcal{R}_{0,\alpha}$, then
 \begin{alignat}{1}\nonumber
     O(1):&-\p_\rho^2 \tilde{\ve}_2-\eta''(\rho)(h_2\hat{\ve}_0+(\rho+h_1)\hat{\ve}_1)+\eta'(\rho)\hat{p}_0\no(\rho+h_1)+\no\p_\rho\tilde{p}_1-\p_\rho\tilde{p}_0\nabla^\G  h_1\nonumber\\
     &-\p_\rho\tilde{p}_{-1}\nabla^\G  h_2+\eta(\rho)\(\nabla_x(d_\Gamma\hat{p}_0)-\Delta_x(d_\Gamma\hat{\ve}_0)\)-\Delta_x\tilde{\ve}_0+\nabla_x\tilde{p}_0+2\eta'(\rho)\nabla^\G  h_1\cdot \nabla_x(d_\Gamma \hat{\ve}_0)\nonumber\\
     &-2\eta'(\rho)\partial_\no(d_\Gamma\hat{\ve}_1)-(\eta'(\rho)h_1\hat{\ve}_0
     +\eta'(\rho)\hat{\ve}_0\rho)\Delta_x d_\Gamma\nonumber\\
    &\ = \p_\rho(\theta_0'(\rho))^2\nabla^\G  h_2-\p_\rho(\theta_0'(\rho))^2|\nabla^\G  h_1|^2\no+(\theta_0'(\rho))^2\(\Delta_x d_\Gamma \nabla^\G  h_1+\Delta^\G  h_1\no\),\label{train:scal-c}\\
     O(\eps):   &-\eta''(\rho)(\rho\hat{\ve}_2+h_1\hat{\ve}_2+h_2\hat{\ve}_1)+\eta'(\rho)\no(h_2\hat{p}_0+h_1\hat{p}_1+\rho\hat{p}_1)-\eta'(\rho)\nabla^\G h_1(\rho+h_1)\hat{p}_0\nonumber\\
    &-\p_\rho\tilde{p}_1\nabla^\G  h_1-\eta''(\rho)(\rho+h_1)\hat{\ve}_0|\nabla^\G  h_1|^2-\p_\rho\tilde{p}_0\nabla^\G  h_2+\eta'(\rho)(\rho+h_1)\hat{\ve}_0\Delta^\G  h_1+\eta(\rho)\nabla_x(d_\Gamma\hat{p}_1)\nonumber\\
    &+2\eta'(\rho)\nabla^\G  h_2\cdot\nabla_x(d_\Gamma \hat{\ve}_0)+2\eta'(\rho)\nabla^\G  h_1\cdot \nabla_x(d_\Gamma\hat{\ve}_1)-2\eta'(\rho)\p_\no (d_\Gamma\hat{\ve}_2)-2\p_\no \p_\rho\tilde{\ve}_2+\no\p_\rho\tilde{p}_2\nonumber\\
    &+\nabla_x\tilde{p}_1-\eta'(\rho)(h_1\hat{\ve}_1+h_2\hat{\ve}_0)\Delta_x d_\Gamma-\p_\rho\tilde{\ve}_2\Delta_xd_\Gamma
    -\eta'(\rho)\hat{\ve}_1\rho\Delta_xd_\Gamma-\eta(\rho)\Delta_x(d_\Gamma\hat{\ve}_1)
    -\Delta_x\tilde{\ve}_1\nonumber\\
    &\ =\(|\nabla^\G  h_1|^2\nabla^\G  h_1-2(\nabla^\G  h_1\cdot\nabla^\G  h_2)\no\)\p_\rho(\theta'(\rho))^2\nonumber\\
    &\ \ +(\theta'(\rho))^2\(\Delta^\G  h_2\no-\Delta^\G  h_1\nabla^\G  h_1+\Delta_x d_\Gamma \nabla^\G  h_2-\tfrac 12\nabla^\G |\nabla^\G  h_1|^2\)\label{train:scal-d}\\
    O(\eps^2)&: \sum_{0\leq {i'}\leq 2 ;0\leq i,j,j'\leq 1}  R^{i' j'  i j }_\eps (x,t)  (\partial_s^{j'} h_2)^i (\partial_s^{i'} h_2 )^j\nonumber\\
    &+\sum_{0\leq i,j,k,i',j',k'\leq 1}  \tilde{R}^{i' j' k'  i j k }_\eps (x,t)  (\partial_s^{j'} h_2)^i (\partial_s^{i'} h_2 )^j(\partial_s^{k'} h_2 )^k
\label{train:scal-e}
  \end{alignat}
 where
$
   R^{i' j' i j }_\eps, \tilde{R}^{i' j' k' i j k }_\eps
$  are uniformly bounded with respect to $\eps\in (0,1], (x,t)\in \Gamma(3\delta)$. 
 \end{lem}
 \begin{proof}
 Since $c_0^{in}:=\theta_0(\rho)$, it follows from  \eqref{yuningliu:formula1}  that:
 \begin{equation*}
   \begin{split}
     -\eps\Div&\(\nabla\theta_0(\tfrac{d_\Gamma}\eps - h_\eps)\otimes\nabla\theta_0(\tfrac{d_\Gamma}\eps - h_\eps)\)
     =-\eps\Delta\theta_0(\tfrac{d_\Gamma}\eps - h_\eps)\nabla\theta_0(\tfrac{d_\Gamma}\eps - h_\eps)-\tfrac{\eps}2\nabla|\nabla\theta_0(\tfrac{d_\Gamma}\eps - h_\eps)|^2\\
     =&-\(\tfrac1{\eps^2}\theta_0''(\rho) +\theta_0''(\rho)|\nabla^\G  h_\eps|^2+ \theta_0'(\rho) \tfrac{\Delta d_\Gamma}{\eps} -\theta_0'(\rho)\Delta^\G  h_\eps\)\left(\no -\eps\nabla^\G  h_\eps\right)\theta_0'(\rho)\\
     &-\tfrac{\eps}2\nabla\(\(\tfrac{1}{\eps^2}+|\nabla^\G  h_\eps|^2\)(\theta_0'(\tfrac{d_\Gamma}\eps - h_\eps))^2\)\\
     =&-\tfrac1{\eps^2}\p_\rho(\theta_0'(\rho))^2\no+\tfrac 1\eps\p_\rho(\theta_0'(\rho))^2\nabla^\G  h_1+\p_\rho(\theta_0'(\rho))^2\nabla^\G  h_2-\p_\rho(\theta_0'(\rho))^2|\nabla^\G  h_1|^2\no\\
     &-  (\theta_0'(\rho))^2\tfrac{\Delta d_\Gamma}\eps\no+(\theta_0'(\rho))^2\(\Delta d_\Gamma \nabla^\G  h_1+\Delta^\G  h_1\no\)\\
     &+{\eps\(|\nabla^\G  h_1|^2\nabla^\G  h_1-2(\nabla^\G  h_1\cdot\nabla^\G  h_2)\no\)\p_\rho(\theta_0'(\rho))^2}\nonumber\\
    &
    +{\eps(\theta_0'(\rho))^2\(\Delta^\G  h_2\no-\Delta^\G  h_1\nabla^\G  h_1+\Delta d_\Gamma \nabla^\G  h_2-\tfrac 12\nabla^\G |\nabla^\G  h_1|^2\)}\\
    &+\eps^2\sum_{ {i'} \leq 2;   i,j,j'\leq 1}     R^{ i' j' i,j }_\eps (x,t)  \(\partial_s^{j'} h_2\)^i \(\partial_s^{i'} h_2 \)^j\nonumber\\
    &+\eps^2\sum_{  i,j,k,i',j',k'\leq 1}    \tilde{R}^{ i' j' k'  i j k }_\eps (x,t)  \(\partial_s^{j'} h_2\)^i \(\partial_s^{i'} h_2 \)^j\(\partial_s^{k'} h_2 \)^k
   \end{split}
 \end{equation*}
 where
$
   R^{i' j' i j }_\eps, \tilde{R}^{i' j' k' i j k }_\eps
$  are uniformly bounded with respect to $\eps\in (0,1], (x,t)\in \Gamma(3\delta)$.
On the other hand,
 \begin{equation*}
   \begin{split}
\Delta\tilde{\ve}_j&=(\tfrac{1}{\eps^2}+|\nabla^\G  h_\eps|^2)\p_\rho^2 \tilde{\ve}_j+(\tfrac{ \Delta_x d_\Gamma}{\eps}-\Delta^\G  h_\eps)\p_\rho \tilde{\ve}_j
    \\
& +2(\tfrac{\no}\eps-\nabla^\G  h_\eps)\cdot \nabla_x\p_\rho \tilde{\ve}_j(\rho, x,t)+\Delta_x \tilde{\ve}_j,\quad0\leq j\leq 2,\\
\nabla \tilde{p}_j&=(\tfrac{\no}\eps -\nabla^\G  h_\eps)\p_\rho \tilde{p}_j+\nabla_x \tilde{p}_j,\quad-1\leq j\leq 1.
   \end{split}
 \end{equation*}
 Furthermore we have
 \begin{equation*}
   \begin{split}
\Delta \(d_\Gamma \hat{\ve}_j(x,t)\eta(\rho)\)&=d_\Gamma\hat{\ve}_j\((\tfrac1{\eps^2}+|\nabla^\G  h_\eps|^2)\eta''(\rho)+(\tfrac {\Delta d_\Gamma}\eps-\Delta^\G  h_\eps )\eta'(\rho)\)\\
     &\quad +\eta(\rho)\Delta(d_\Gamma\hat{\ve}_j)+2\eta'(\rho)\(\tfrac\no\eps-\nabla^\G  h_\eps\)\cdot\nabla_x(d_\Gamma\hat{\ve}_j),\\
\nabla\(d_\Gamma \hat{p}_j(x,t)\eta(\rho)\)&=\eta(\rho)\nabla_x(d_\Gamma \hat{p}_j)+d_\Gamma \hat{p}_j\(\tfrac{\no}\eps-\nabla^\G  h_\eps\) \eta'(\rho).
\end{split}
 \end{equation*}
If we substitute the above two sets of formulas into \eqref{yuning:1.67}, we get the desired expansion. The last statement \eqref{train:scal-e} follows from Lemma \ref{lem:ExpDecay} and the exponential decay of $\theta'_0,\theta''_0,\eta',\eta'', \p_\rho\tilde{p}_1,\p_\rho\tilde{\ve}_2$  and $\p_\rho^2 \tilde{\ve}_2$.
 \end{proof}

\noindent
\subsubsection{Stokes equation for $\ve_0$ and $p_{-1}$:}
 Taking the normal component of \eqref{train:scal-a} and using \eqref{yuning:v0nvanish} gives $\p_\rho \tilde{p}_{-1}=-\p_\rho(\theta_0'(\rho))^2$ and therefore
 \begin{equation*}
    \tilde{p}_{-1}(\rho,x,t)=-\theta_0'(\rho)^2+\tilde{\tilde{p}}_{-1}(x,t)~\text{for all}~(x,t)\in\Gamma(3\delta)
 \end{equation*}
 for some $\tilde{\tilde{p}}_{-1}\colon \Gamma(3\delta)\to \R$.
 The matching condition \eqref{yuning:matching3} implies that
 \begin{equation*}
   \begin{split}
     p^\pm_{-1}(x,t)&=0 + \tilde{\tilde{p}}_{-1}(x,t) \pm d_\Gamma \hat{p}_{-1}(x,t)\qquad \text{if } (x,t)\in\Gamma(3\delta),
   \end{split}
 \end{equation*}
where $p^\pm_{-1}(x,t)$ are constants due to the outer expansion, which  can be  chosen to be $0$ for simplicity. Hence we obtain $\hat{p}_{-1}\equiv  \tilde{\tilde{p}}_{-1}\equiv 0$.
 As a result
 \begin{equation}\label{yuning:pvanish}
   { \tilde{p}_{-1}(\rho,x,t)=  p_{-1}(\rho,x,t)=-\theta_0'(\rho)^2,~\hat{p}_{-1}=0\quad \text{for all } (x,t)\in \Gamma(3\delta).}
 \end{equation}
 Now we go back to \eqref{train:scal-a} and deduce that $-\p_\rho^2 \tilde{\ve}_0=0$ and this implies
  \begin{equation}\label{yuning:v0indepent}
     \tilde{\ve}_0(\rho,x,t)=\tilde{\ve}_0(x,t)\qquad \text{for all}~(x,t)\in \Gamma(3\delta)
  \end{equation}
  since the only bounded solution of $-\p_\rho^2 \tilde{\ve}_0=0$ is independent of $\rho$.
  This together with Lemma~\ref{yuning:compulemma1} yields
  \begin{equation}\label{yuning:vtilde0}
  {\tilde{\ve}_0 = \tfrac12 (\ve_0^+ +\ve_0^-)\text{ on }\Gamma(3\delta),\quad   \hat{\ve}_0 = \tfrac1{2d_\Gamma} (\ve_0^+ -\ve_0^-)\text{ on } \Gamma(3\delta)\setminus \Gamma}
\end{equation}
and
\begin{equation*}
  \ve_0(\rho,x,t)=\tfrac 12(\ve_0^+(x,t)+\ve_0^-(x,t))+\tfrac {\eta(\rho)}2(\ve_0^+(x,t)-\ve_0^-(x,t)),~\forall (x,t)\in\Gamma(3\delta)\backslash\Gamma, \rho \in\R.
\end{equation*}
On the interface  $\Gamma$ the  matching condition yields
\begin{equation*}
  \lim_{\rho\to\pm\infty}\tilde{\ve}_0(\rho,x,t)=\lim_{\rho\to\pm\infty}{\ve}_0(\rho,x,t)=\ve_0^\pm(x,t)\quad \text{for all } (x,t)\in \Gamma.
\end{equation*}
Hence we have
\begin{equation}\label{yuning:jumpv0}
  {[\ve^\pm_0](x,t):=\ve_0^+(x,t)-\ve_0^-(x,t)=0,\quad \text{for all } (x,t)\in \Gamma}
\end{equation}
and it follows from Lemma \ref{yuning:compulemma1} that
\begin{equation} \label{yuning:redefv0n}
 { \hat{\ve}_0=\tfrac12\p_\no (\ve^+_0- \ve^-_0)~\text{on}~\Gamma}.
\end{equation}
 \medskip

\noindent
\subsubsection{Divergence equation for $\ve_1$:}
Using \eqref{yuning:v0indepent} and \eqref{yuning:diver1}, we obtain at order   $O(1)$ of  the expansion for $\Div \ve_A^{in}$ that:
\begin{equation}\label{eq:DivEqv1}
  \begin{split}
    &  (\rho+h_1)  \eta'(\rho)\hat{\ve}_{0,\no}(x,t) -d_\Gamma\eta'(\rho)
  \nabla^\G  h_1\cdot\hat{\ve}_{0,\btau}(x,t)\\&+\Div_x\(\tilde{\ve}_0 (\rho,x,t)+ \hat{\ve}_0(x,t)d_\Gamma \eta(\rho)\)+ \partial_\rho \tilde{\ve}_{1,\no} (\rho,x,t)=0
  \end{split}
\end{equation}
  By restricting \eqref{eq:DivEqv1} on $\Gamma$ we can show that $\partial_\rho \tilde{\ve}_{1,\no}  (\rho,x,t)=0$, which implies
\begin{equation}\label{yuning:tildev1n}
  \tilde{\ve}_{1,\no}(\rho,x,t) =\tilde{\ve}_{1,\no}(x,t)\quad\text{on } \Gamma,
\end{equation}
by showing that
\begin{equation*}
 (\rho+h_1)  \eta'(\rho)\hat{\ve}_{0,\no}(x,t)+\Div_x \(\tilde{\ve}_0 (x,t)+ \hat{\ve}_0(x,t){d_\Gamma \eta(\rho)}\)=0\quad\text{on } \Gamma.
\end{equation*}
Actually, it follows from \eqref{yuning:redefv0n}, \eqref{yuning:jumpv0} and $\Div \ve_0^\pm =0$ 
that
 \begin{equation}\label{yuning:v0vanish}
 2\hat{\ve}_{0,\no}= \p_\no (\ve^+_{0,\no}- \ve^-_{0,\no})=  \Div_\btau[ \ve_{0,\btau}^{\First{\pm}}] =0\quad\text{on } \Gamma
 \end{equation}
  and from \eqref{yuning:vtilde0} that
  \begin{equation*}
   \Div_x \tilde{\ve}_0 ( x,t) = \tfrac12(\Div_x \ve^+_0+\Div_x \ve^-_0)=0\quad \text{on }\Gamma.
  \end{equation*}
Hence it follows from \eqref{yuning:tildev1n} and \eqref{yuning:matching2} that $\tilde{\ve}_{1,\no} (x,t)= \ve^\pm_{1,\no} (x,t)$ and therefore
\begin{equation*}
 [\ve_{1,\no}^\pm ]:=\ve^+_{1,\no}-\ve^-_{1,\no}=0\qquad \text{on}~\Gamma.
\end{equation*}
A similar result holds  on $\Gamma(3\delta)\backslash\Gamma$ following from \eqref{yuning:vtilde0}, \eqref{yuning:1.68} and the extension process afterwards:
\begin{equation*}
  \Div_x (\tilde{\ve}_0 + d_\Gamma \hat{\ve}_0 \eta(\rho))=\tfrac12\Div_x (\ve_0^+ +\ve_0^-)+\tfrac 12\eta(\rho)\Div_x  \(\ve_0^+-\ve_0^-\)=0.
\end{equation*}
Since the first two terms on the left hand side of  \eqref{eq:DivEqv1} vanish on $\Gamma$, we can compensate them in \eqref{yuning:diver1} by terms with higher powers of $\eps$ using $1=\eps\frac{\rho+h_\eps}{d_\Gamma}$. Therefore we simply solve
\begin{equation}\label{yuning:newdis1}
  \partial_\rho \tilde{\ve}_{1,\no}  (\rho,x,t)=0\quad \text{on}~\G(3\delta)\backslash\Gamma.
\end{equation}
Thus we obtain from Lemma \ref{yuning:compulemma1} that
\begin{equation}\label{yuning:tildev1nfinal}
 \tilde{\ve}_{1,\no}  = \tfrac12 (\ve_{1,\no} ^+ +\ve_{1,\no} ^-),\quad   \hat{\ve}_{1,\no}  = \tfrac1{2d_\Gamma} (\ve_{1,\no} ^+ -\ve_{1,\no} ^-)\quad\text{on } \G(3\delta)\backslash\Gamma
\end{equation}
and
 the $O(\eps)$ order equation of \eqref{yuning:diver1} becomes:
\begin{alignat}{1}\nonumber
     & \tfrac{(\rho+h_1)^2}{d_\Gamma}  \eta'(\rho)\hat{\ve}_{0,\no}  -(\rho+h_1)\eta'(\rho)
  \nabla^\G  h_1 \cdot\hat{\ve}_{0,\btau} - \nabla^\G  h_1\cdot\partial_\rho\tilde{\ve}_{1,\btau}+h_2\eta'(\rho)\hat{\ve}_{0,\no}+  (\rho+h_1)  \eta'(\rho)\hat{\ve}_{1,\no}\\ \label{yuning:divforv2}
&    - d_\Gamma\eta'(\rho)
  \nabla^\G  h_1 \cdot\hat{\ve}_{1,\btau}- d_\Gamma\eta'(\rho)
  \nabla^\G  h_2 \cdot\hat{\ve}_{0,\btau}  + \Div_x\(\tilde{\ve}_1  + \hat{\ve}_1 d_\Gamma \eta(\rho)\)+ \partial_\rho \tilde{\ve}_{2,\no}
  =0.
\end{alignat}

\noindent
\subsubsection{Stokes equation for $\ve_1$ and $p_{0}$:}
%
We substitute \eqref{yuning:pvanish} and \eqref{yuning:vtilde0} into \eqref{train:scal-b}. This leads to
\begin{equation}\label{yuning:newmunus1}
  O(\tfrac{1}{\eps}):-\p_\rho^2 \tilde{\ve}_1+\no\p_\rho \tilde{p}_0 -\hat{\ve}_0(\rho+h_1) \eta''(\rho)  -2  \p_\no (d_\Gamma \hat{\ve}_0) \eta'(\rho) =-(\theta_0'(\rho))^2\no \Delta d_\Gamma~\text{in}~\G(3\delta).
\end{equation}
 Restricting  \eqref{yuning:newmunus1} on $\Gamma$, integrating with respect to $\rho\in \R$, and  using \eqref{yuning:matching2}, \eqref{yuning:matching3}, \eqref{yuning:jumpv0} and the last formula in \eqref{yuning:1.26} imply
\begin{equation*}
 \no[\tilde{p}_0]-2\hat{\ve}_0= \sigma H\no\qquad \text{on }\Gamma.
\end{equation*}
To  proceed, we need the following lemma
 \begin{lem}\label{yuning:limitlem1}
   Let $j=0,1$.Under the condition that $[\ve_j^\pm]:=\ve_j^+-\ve_j^-=0$ on $\Gamma$, it holds that
    \begin{equation*}
      \quad [p_j^\pm]=[\tilde{p}_j],2\hat{\ve}_j=[\p_\no  \ve_j^\pm]=2\(D\ve_j^+-D\ve_j^-\)\quad \text{on }\Gamma.
    \end{equation*}
 \end{lem}
\begin{proof}
We shall only prove $[\p_\no  \ve_j^\pm]=2\(D\ve_j^+-D\ve_j^-\)$ since the others are consequences of the matching condition \eqref{yuning:matching3} and Lemma \ref{yuning:compulemma1} (see also \eqref{yuning:v1jumpconse}). On $\Gamma$ we have
  \begin{equation*}
    \begin{split}
      2(D\ve_j^+-D\ve_j^-)\no&=\p_\no (\ve_j^+-\ve_j^-)-(\nabla \no)\cdot (\ve_j^+-\ve_j^-)+\nabla(\ve_{j,\no} ^+-\ve_{j,\no}^-)\\
      &=\p_\no (\ve_j^+-\ve_j^-)+\no\p_\no  (\ve_{j,\no} ^+-\ve_{j,\no}^-)+\btau\p_\btau  (\ve_{j,\no}^+-\ve_{j,\no}^-)
    \end{split}
  \end{equation*}
  The last two components vanish due to $[\ve_j^\pm]=0$ and $\Div \ve_j^\pm=0$.
\end{proof}

These altogether imply
\begin{equation} \label{yuning:struggle1}
 2(D\ve_0^+-D\ve_0^-)\no-(p^+_0-p^-_0)\no =-\sigma H\no\qquad \text{on }\Gamma.
\end{equation}
This along with formal outer expansion leads to the following  first order sharp interface system for $(\ve_0^\pm, p^\pm_0)$ which has a solution, which is smooth in $\overline{\Omega^\pm(t)}$, as long as $\Gamma_t$ remains smooth:
  \begin{alignat*}{2}
    -\Delta \ve_0^\pm +\nabla p_0^\pm &= 0 &\qquad &\text{in }\Omega^\pm (t), t\in (0,T_0),\\
    \Div \ve_0^\pm &= 0 &\qquad &\text{in }\Omega^\pm (t), t\in (0,T_0),\\
    [2D  \ve_0^\pm-p_0^\pm\mathbf{I}]\no &= -\sigma H\no && \text{on }\Gamma_t, t\in (0,T_0),\\
     [\ve^\pm_0]&=0 && \text{on }\Gamma_t, t\in (0,T_0),\\
 \Second{\ve_0|_{\partial\Omega}} &\Second{=0} &&\Second{\text{on }\partial\Om\times (0,T_0).}
  \end{alignat*}
Now we can take the inner product of \eqref{yuning:newmunus1} with $\mathbf{n}$ and use the above equations together with \eqref{yuning:v0vanish}, \eqref{yuning:tildev1n}, and \eqref{yuning:vtilde0} and this implies
 \begin{equation}\label{yuning:tildep0}
   \p_\rho \tilde{p}_0 
=(\theta_0'(\rho))^2H \qquad \text{on }\Gamma.
 \end{equation}
 We integrate the above formula, use the matching condition, and obtain
 \begin{equation*}
   \tilde{p}_0(\rho,x,t)=H\int^\rho_{-\infty}(\theta_0'(z))^2{\rm d} z+p^-_0(x,t)\quad \text{ on } \Gamma.
 \end{equation*}
 Taking $\rho\to +\infty$ in the above formula and combining it with \eqref{yuning:redefv0n} and \eqref{yuning:struggle1} leads to
 \begin{equation}\label{yuning:v0hatvani}
  {2\hat{\ve}_0 = \p_\no  \ve^+_0 -\p_\no  \ve^-_0 = 0,~ p_0^+-p_0^-=\sigma H \qquad \text{on }\Gamma.}
\end{equation}
Using this, \eqref{yuning:v0indepent} and \eqref{yuning:tildep0}, we deduce from \eqref{yuning:newmunus1}  that
\begin{equation*}
  \partial_\rho^2\tilde{\ve}_1(\rho,x,t)=0 \qquad \text{on}~ \G.
\end{equation*}
Hence we can use \eqref{yuning:matching2} to deduce $\partial_\rho \tilde{\ve}_1=0$ and
\begin{equation}\label{yuning:v1jump}
 { [\ve_1]=\ve_1^+-\ve_1^-=0 \qquad \text{on } \Gamma.}
\end{equation}
Due to \eqref{yuning:v0hatvani}, the last two terms on the left hand side of \eqref{yuning:newmunus1} vanish on $\Gamma$ and we omit them 
and get
\begin{equation}\label{yuning:takenormal}
   O(\tfrac{1}{\eps}):-\p_\rho^2 \tilde{\ve}_1+\no\p_\rho \tilde{p}_0  =-(\theta_0'(\rho))^2\no \Delta d_\Gamma~\text{in}~\G(3\delta)
\end{equation}
 but we  have to compensate the omitted terms into \eqref{train:scal-c} and \eqref{train:scal-d}. Meanwhile, we use \eqref{yuning:pvanish} and \eqref{yuning:vtilde0} to simplify \eqref{train:scal-c} into:
 \begin{equation}\label{yuning:train1new1}
   \begin{split}
     O(1):&-\p_\rho^2 \tilde{\ve}_2-\eta''(\rho)(h_2\hat{\ve}_0+(\rho+h_1)\hat{\ve}_1)+\eta'(\rho)\hat{p}_0\no(\rho+h_1)+\no\p_\rho\tilde{p}_1-\p_\rho\tilde{p}_0\nabla^\G  h_1 \\
     &+\eta(\rho)\(\nabla_x(d_\Gamma\hat{p}_0)-\Delta_x(d_\Gamma\hat{\ve}_0)\)-\Delta_x\tilde{\ve}_0+\nabla_x\tilde{p}_0+2\eta'(\rho)\nabla^\G  h_1\cdot \nabla_x(d_\Gamma \hat{\ve}_0) \\
     &-2\eta'(\rho)\partial_\no(d_\Gamma\hat{\ve}_1) 
-  \eta'(\rho)(h_1\hat{\ve}_0+\hat{\ve}_0\rho)\Delta_x d_\Gamma \\
     &\underbrace{-\tfrac{\hat{\ve}_0}{d_\Gamma}\p_\rho\((\rho+h_1)^2\eta'(\rho)\)-2(\rho+h_1)\p_\no \hat{\ve}_0\eta'(\rho)}_{\text{new terms due to compensation}}\\
     =&-\p_\rho(\theta_0'(\rho))^2|\nabla^\G  h_1|^2\no+(\theta_0'(\rho))^2\(\Delta d_\Gamma \nabla^\G  h_1+\Delta^\G  h_1\no\),
   \end{split}
 \end{equation}
 while the following terms must be added to \eqref{train:scal-d}:
 \begin{equation}\label{yuning:penlty}
  - \tfrac{\hat{\ve}_0}{d_\Gamma} h_2(\rho+h_1)\eta''(\rho)-2\tfrac{\hat{\ve}_0}{d_\Gamma} h_2\eta'(\rho)-2h_2\p_\no\hat{\ve}_0\eta'(\rho).
 \end{equation}
 More precisely, the last two terms on the left-hand side of \eqref{yuning:newmunus1} can be treated as follows:
 \begin{equation*}
   \begin{split}
     &-\hat{\ve}_0(\rho+h_1) \eta''(\rho)  -2  \p_\no (d_\Gamma \hat{\ve}_0) \eta'(\rho)
     =-\tfrac{\eps(\rho+h_\eps)}{d_\G}\(\hat{\ve}_0(\rho+h_1) \eta''(\rho)  +2  \p_\no (d_\Gamma \hat{\ve}_0) \eta'(\rho)\)\\
     &=\eps\(-\tfrac{\hat{\ve}_0}{d_\Gamma}\p_\rho\((\rho+h_1)^2\eta'(\rho)\)-2(\rho+h_1)\p_\no \hat{\ve}_0\eta'(\rho)\)\\
     &\quad +\eps^2\(- \tfrac{\hat{\ve}_0}{d_\Gamma} h_2(\rho+h_1)\eta''(\rho)-2\tfrac{\hat{\ve}_0}{d_\Gamma} h_2\eta'(\rho)-2h_2\p_\no\hat{\ve}_0\eta'(\rho)\).
   \end{split}
 \end{equation*}
   Now we take the normal component of \eqref{yuning:takenormal} and employ  \eqref{yuning:newdis1}. This leads to
 \begin{equation}\label{yuning:simp1}
\partial_\rho \tilde{p}_0 +(\theta_0'(\rho))^2 \Delta d_\Gamma=0
\end{equation}
 and this together with matching conditions \eqref{yuning:matching3} implies
\begin{equation}\label{yuning:p0expli}
\hat{p}_0 = \tfrac1{2d_\Gamma}\(p_0^+-p_0^-  + \sigma\Delta d_\Gamma   \),
   \tilde{p}_0 = \tfrac 12(p_0^++p_0^-)-\tfrac \sigma 2\Delta d_\G \eta(\rho)
   ~\text{in}~\Gamma(3\delta)\setminus\Gamma
\end{equation}
where $\sigma=\int_\R (\theta'_0(s))^2\sd s$.  Here we have used $\tfrac \sigma 2\partial_\rho \eta(\rho)=\theta_0'(\rho)^2$ by the definition of $\eta$.
 Thus
 \begin{equation}\label{yuning:p0missing}
   p_0=\tfrac 12(p_0^++p_0^-)+\tfrac {\eta(\rho)}2(p_0^+-p_0^-).
 \end{equation}
 So \eqref{yuning:takenormal} reduces to $\p_\rho^2 \tilde{\ve}_1=0$ in $\G(3\delta)$. Since the only bounded solution of this equation is independent of $\rho$, we can use Lemma \ref{yuning:compulemma1} to deduce
\begin{equation}\label{yuning:explitsoluv1}
 { \hat{\ve}_1=\tfrac 1{2d_\Gamma}(\ve^+_1-\ve^-_1),~\tilde{\ve}_1=\tfrac 12(\ve^+_1+\ve^-_1),~\ve_1=\tfrac 12(\ve^+_1+\ve^-_1)+\tfrac {\eta(\rho)}2(\ve^+_1-\ve^-_1)}.
\end{equation}
This combined with \eqref{yuning:v1jump}  implies
\begin{equation}\label{yuning:v1jumpconse}
  {2\hat{\ve}_1=\p_\no (\ve_1^+-  \ve_1^-)\quad \text{on }\Gamma.}
\end{equation}

\noindent
\subsubsection{Divergence equation for $\ve_2$:}
Now we consider  the order $O(\eps)$ equation for the divergence-free equation \eqref{yuning:divforv2}.
Restricting \eqref{yuning:divforv2} on $\G$ and using \eqref{yuning:vtilde0}, \eqref{yuning:explitsoluv1}, \eqref{yuning:v0vanish}, \eqref{yuning:v1jump} we conclude that
\begin{equation}\label{yuning:diverv2}
\lim_{d_\G\to 0}\tfrac{\hat{\ve}_{0,\no}}{d_\Gamma} (\rho+h_1)^2 \eta'(\rho)  +  (\rho+h_1)  \eta'(\rho)\hat{\ve}_{1,\no}    + \partial_\rho \tilde{\ve}_{2,\no}   =0\quad \text{on}~\Gamma.
\end{equation}
Note that the second term in the above equation also vanishes due to  \eqref{yuning:v1jumpconse}, \eqref{yuning:v1jump} and the divergence-free condition of $\ve_1^\pm$:
\begin{equation}\label{yuning:v1nvanish}
  { \hat{\ve}_{1,\no} =\tfrac12\p_\no (\ve^+_{1,\no} - \ve^-_{1,\no} )=-\tfrac12\Div_\btau(\ve^+_{1,\btau}- \ve^-_{1,\btau})=0\quad \text{on}~\Gamma}.
\end{equation}
Using this in \eqref{yuning:diverv2} and  employing  \eqref{yuning:v0vanish}, we obtain 
\begin{equation}\label{yuning:v2nvanish}
 \partial_\rho \tilde{\ve}_{2,\no} (\rho,x,t)+(\rho+h_1)^2\eta'(\rho)
 \p_\no \hat{\ve}_{0,\no}=0\quad\text{on }\Gamma.
\end{equation}
In order to  determine $\ve_{2,\no} ^\pm$,  instead of  solving a certain sharp interface system as we did for $(\ve^\pm_1,p^\pm_1)$, we choose $\ve_{2,\no} ^-\equiv0$ and integrate \eqref{yuning:v2nvanish}:
\begin{equation*}
  (\ve_{2,\no} ^+-\ve_{2,\no} ^-)+\int_{\R}(\rho+h_1)^2\eta'(\rho)\, \sd \rho\,\p_\no \hat{\ve}_{0,\no}=0\quad \text{on}~\Gamma.
\end{equation*}
Altogether, we define
\begin{equation}\label{yuning:v2nplus1}
  \ve_{2,\no} ^-\equiv0~\text{in }\Omega\times [0,T_0],\quad \ve_{2,\no} ^+=-\p_\no \hat{\ve}_{0,\no}
  \int_{\R}(\rho+h_1)^2\eta'(\rho)\sd \rho\quad \text{in }\overline{\Gamma(3\delta)}
\end{equation}
and extend $\ve_{2,\no} ^+$ smoothly to $\O \times [0,T_0]$. Note that, in contrast to $\ve^\pm_1$, we do not assume $\ve^+_2$ to be divergence-free.

In  view of \eqref{yuning:v0hatvani},  we can treat \eqref{yuning:divforv2} in $\G(3\delta)$ in the same manner as before by omitting the terms that vanish on the interface and just solve
 \begin{equation}\label{yuning:v2tilde2}
   {(\rho+h_1)^2}  \eta'(\rho)\tfrac{\hat{\ve}_{0,\no}}{d_\Gamma}   + \partial_\rho \tilde{\ve}_{2,\no}   =0\quad \text{in}~ \Gamma(3\delta) .
\end{equation}
This together with \eqref{yuning:explitsoluv1}, \eqref{yuning:vtilde0}, \eqref{yuning:1.68} and the extension process after it will  change the expansion of the divergence equation   \eqref{yuning:diver1} into
\begin{alignat}{1}\nonumber
&\Div\left( \ve_A^{in}(\tfrac{d_\Gamma}\eps-h_\eps, x,t)\right)=\nonumber\\
\nonumber
  O(\eps):\quad &  -(\rho+h_1)\eta'(\rho)
  \nabla^\G  h_1 \cdot\hat{\ve}_{0,\btau}- d_\Gamma\eta'(\rho)
  \nabla^\G  h_1 \cdot\hat{\ve}_{1,\btau} + (\rho+h_1)\eta'(\rho)\hat{\ve}_{1,\no} \\
  \label{yuning:lip1'}
  &+ h_2\eta'(\rho)\hat{\ve}_{0,\no}- d_\Gamma \eta'(\rho)\nabla^\Gamma h_2\cdot \hat{\ve}_{0,\btau}
  \\\nonumber
   O(\eps^2):\quad & +(\rho+h_1)h_2 \eta'(\rho)\frac{\hat{\ve}_{0,\no}}{d_\Gamma}- h_2 \eta'(\rho)\nabla^\Gamma h_1\cdot \hat{\ve}_{0,\btau} + h_2 \eta'(\rho)\hat{\ve}_{1,\no}- d_\Gamma \eta'(\rho)\nabla^\Gamma h_2\cdot \hat{\ve}_{1,\btau}\\\nonumber
&- \nabla^\G  h_\eps\cdot\partial_\rho\tilde{\ve}_{2,\btau} +  (\rho+h_\eps)  \eta'(\rho)\hat{\ve}_{2,\no} (x,t) -  d_\Gamma\eta'(\rho)
  \nabla^\G  h_\eps\cdot\hat{\ve}_{2,\btau}(x,t)\\\label{yuning:lip1}
& + \Div_x\(\tilde{\ve}_2 (\rho,x,t)+ \hat{\ve}_2(x,t)d_\Gamma \eta(\rho)\)
\end{alignat}
as remainders terms at the corresponding orders.

The above analysis together with \eqref{yuning:matching2} allow us to solve $\tilde{\ve}_{2,\no} $ and $\hat{\ve}_{2,\no} $ uniquely:
 \begin{lem}\label{yuning:v2nplus2}
  The normal component of $\tilde{\ve}_2$ and $\hat{\ve}_2$ are given by:
  \begin{alignat*}{2}
  \tilde{\ve}_{2,\no} &=\ve^-_{2,\no} +d_\Gamma\hat{\ve}_{2,\no}
  -\frac{\hat{\ve}_{0,\no}}{d_\Gamma}\int_{-\infty}^\rho(z+h_1)^2\eta'(z)\sd z&\quad& \text{on } \Gamma(3\delta)\backslash\Gamma,\\
  \hat{\ve}_{2,\no} &=\frac 1{2d_\Gamma}\(\frac{\hat{\ve}_{0,\no}}{d_\Gamma}\int_{\mathbb{R}}(z+h_1)^2\eta'(z)\sd z+\ve_{2,\no} ^+-\ve_{2,\no} ^-\)&\quad& \text{on } \Gamma(3\delta)\backslash\Gamma,
\end{alignat*}
where $\ve_{2,\no} ^\pm$ is given by \eqref{yuning:v2nplus1} and is extended smoothly to $\O \times [0,T_0]$.
\end{lem}
\noindent
\subsubsection{Stokes equation for $\ve_2$ and $p_{1}$:}
According to \eqref{yuning:vtilde0} and \eqref{yuning:p0expli}
\begin{align*}
  \Delta_x \tilde{\ve}_0 +  \Delta_x( d_\Gamma \hat{\ve}_0)\eta(\rho)&= \Delta_x \ve_0^+ \tfrac{1+\eta(\rho)}2 + \Delta_x \ve_0^- \tfrac{1-\eta(\rho)}2,\\
  \nabla_x \tilde{p}_0 + \nabla_x (d_\Gamma \hat{p}_0)\eta(\rho) &=
  \nabla_x p_0^+ \tfrac{1+\eta(\rho)}2 + \nabla_x p_0^- \tfrac{1-\eta(\rho)}2.
\end{align*}
Thus
\begin{align*}
   & -\Delta_x \tilde{\ve}_0 -\Delta_x (d_\Gamma \hat{\ve}_0)\eta(\rho) + \nabla_x \tilde{p}_0 + \nabla_x (d_\Gamma \hat{p}_0)\eta(\rho)\\
&= (-\Delta_x \ve_0^+ +\nabla_x p_0^+)\tfrac{1+\eta(\rho)}2    +(-\Delta_x \ve_0^- +\nabla_x p_0^-)\tfrac{1-\eta(\rho)}2.
\end{align*}
This together with \eqref{yuning:vtilde0}, \eqref{yuning:simp1}, \eqref{yuning:p0expli}, and \eqref{yuning:explitsoluv1}   reduces \eqref{yuning:train1new1} to
\begin{alignat}{1} \nonumber
      O(1):&-\p_\rho^2 \tilde{\ve}_2-\eta''(\rho)(h_2\hat{\ve}_0+(\rho+h_1)\hat{\ve}_1)+\eta'(\rho)\hat{p}_0\no(\rho+h_1)+\no\p_\rho\tilde{p}_1 \\\nonumber
     &+{(-\Delta_x \ve_0^+ +\nabla_x p_0^+)\tfrac{1+\eta(\rho)}2    +(-\Delta_x \ve_0^- +\nabla_x p_0^-)\tfrac{1-\eta(\rho)}2}+2\eta'(\rho)\nabla^\G  h_1\cdot\nabla_x(d_\Gamma \hat{\ve}_0) \\\nonumber
     &-2\eta'(\rho)\partial_\no(d_\Gamma\hat{\ve}_1) -(h_1\hat{\ve}_0+\hat{\ve}_0\rho)\eta'(\rho)\Delta_x d_\Gamma \\\nonumber
     &-\tfrac{\hat{\ve}_0}{d_\Gamma}\p_\rho\((\rho+h_1)^2\eta'(\rho)\)-2(\rho+h_1)\p_\no \hat{\ve}_0\eta'(\rho)\\\label{yuning:train1new}
     &=-\p_\rho(\theta_0'(\rho))^2|\nabla^\G  h_1|^2\no+(\theta_0'(\rho))^2\Delta^\G  h_1\no.
\end{alignat}
Restricting \eqref{yuning:train1new} on $\G$ and using \eqref{yuning:v0hatvani} to eliminate the terms containing $\hat{\ve}_0$ we derive
\begin{alignat}{1}\nonumber
     &-\p_\rho^2\tilde{\ve}_2-\eta''(\rho)(\rho+h_1)\hat{\ve}_1+\eta'(\rho)(\rho+h_1)\no \hat{p}_0+\no\partial_\rho\tilde{p}_1-2\eta'(\rho)\hat{\ve}_1-\p_\no  \hat{\ve}_0\p_\rho\((\rho +h_1)^2\eta'(\rho)\)\\
     &\quad -2(\rho+h_1)\p_\no \hat{\ve}_0\eta'(\rho)\label{yuning:v2tilde1}
     =-\p_\rho(\theta_0'(\rho))^2|\nabla_\G   h_1|^2\no+(\theta_0'(\rho))^2\Delta_\G  h_1\no\quad \text{on }\Gamma
\end{alignat}
at order $O(1)$.
Integrating the above identity with respect to $\rho\in\mathbb{R}$ and then using \eqref{yuning:matching2} and \eqref{yuning:matching3}    yields
 \begin{equation}\label{yuning:secondlim1}
    -2\hat{\ve}_1+ \no[\tilde{p}_1] +2h_1\no\hat{p}_0-4h_1\p_\no \hat{\ve}_0 =\sigma \Delta _\G h_1\no\quad \text{on }\Gamma.
\end{equation}
Now Lemma~\ref{yuning:limitlem1} together with \eqref{yuning:v1jumpconse} and \eqref{yuning:secondlim1} implies
\begin{equation}\label{yuning:1.39}
 [2D \ve_1^\pm-p^\pm_1\mathbf{I}]\cdot \no=2h_1\no\hat{p}_0-4h_1\p_\no \hat{\ve}_0-\sigma\Delta_\G  h_1\no \quad \text{on } \Gamma.
\end{equation}
 These together with \eqref{yuning:v1jump} and the outer expansion \eqref{yuning:1.68} leads to the second order sharp interface limit \eqref{yuning:secondlimit}.

Now we determine $\tilde{\ve}_{2,\btau}$ and $\tilde{p}_1$ such that the tangential part of \eqref{yuning:train1new} is fulfilled on $\Gamma$ and the normal part is fulfilled on $\Gamma(3\delta)$, up to a high order term in $\eps$.
From the normal part of \eqref{yuning:train1new} we obtain
\begin{alignat}{1}\nonumber
     &\partial_\rho \tilde{p}_1(\rho,x,t)=
\p_\rho^2 \tilde{\ve}_{2,\no} +\eta''(\rho)(h_2\hat{\ve}_{0,\no}+(\rho+h_1)\hat{\ve}_{1,\no} )-\eta'(\rho)\hat{p}_0(\rho+h_1) \\\nonumber
&\quad -2\eta'(\rho)\nabla^\G  h_1\cdot\nabla_x(d_\Gamma \hat{\ve}_{0,\no})
     +2\eta'(\rho)\partial_\no(d_\Gamma\hat{\ve}_{1,\no} ) +(\rho +h_1)\hat{\ve}_{0,\no}\eta'(\rho)\Delta_x d_\Gamma \\\nonumber
     &\quad +\tfrac{\hat{\ve}_{0,\no}}{d_\Gamma}\p_\rho\((\rho+h_1)^2\eta'(\rho)\)+2(\rho+h_1)\p_\no \hat{\ve}_{0,\no}\eta'(\rho)\\\label{yuning:defa1}
     &\quad -\p_\rho(\theta_0'(\rho))^2|\nabla^\G  h_1|^2+(\theta_0'(\rho))^2\Delta^\G  h_1=:a_1(\rho,x,t)~\text{in}~\G(3\delta).
\end{alignat}
Note that we omitted the lower order term
\begin{equation*}
  \left((-\Delta_x \ve_{0}^+ +\nabla_x p_0^+)\tfrac{1+\eta(\rho)}2    +(-\Delta_x \ve_{0}^- +\nabla_x p_0^-)\tfrac{1-\eta(\rho)}2\right)\cdot\no
\end{equation*}
that  will be treated in \eqref{yuning:xiequa1} below.
 Furthermore, $\lim_{\rho\to\pm\infty} \tilde{p}_1(\rho,x,t)= p_1^\pm(x,t)\mp \hat{p}_1(x,t)$ for $x\in \Omega^\pm(t)$ because of \eqref{yuning:matching3}.
Hence we define for $(x,t)\in\Gamma(3\delta)$
 \begin{equation}\label{yuning:p0expli1}
   \begin{split}
     \tilde{p}_1(\rho,x,t) &=\frac12\left( p_1^+(x,t)+p_1^-(x,t)-\int_\R a_1(z,x,t)\sd z\right) + \int_{-\infty}^\rho a_1(z,x,t)\sd z,\\
   \hat{p}_1(x,t) &= \frac1{2d_\Gamma(x,t)}\left( p_1^+(x,t)-p_1^-(x,t) - \int_\R a_1(z,x,t)\sd z\right)
   \end{split}
 \end{equation}
 where $p_1^\pm$ is determined by \eqref{yuning:secondlimit}.
 Restricting \eqref{yuning:defa1} on the interface and integrating with respect to $\rho$ leads to the following compatibility condition:
 \begin{equation}\label{yuning:1.60}
 p^+_1( P_{\G_t}(x),t)-p^-_1(P_{\G_t}(x),t)=\int_\R a_1(z,P_{\G_t}(x),t)\sd z.
 \end{equation}


 In the rest part of this subsection, we shall determine $(\tilde{\ve}_{2,\btau},\hat{\ve}_{2,\btau})$, the tangential part of $(\tilde{\ve}_2,\hat{\ve}_2)$.
The tangential part of \eqref{yuning:secondlim1} implies $\hat{\ve}_{1,\btau}=-2h_1\p_\no \hat{\ve}_{0,\btau}$ on $\Gamma$ and we can employ it
to simplify the equation for  $\ve_{2,\btau}$ on $\Gamma$
\begin{alignat}{1}\nonumber
    -\partial_\rho^2 \tilde{\ve}_{2,\btau}&= \eta''(\rho)(\rho+h_1)\hat{\ve}_{1,\btau}+2\eta'(\rho)\hat{\ve}_{1,\btau}+\p_\no  \hat{\ve}_{0,\btau}\p_\rho\((\rho +h_1)^2\eta'(\rho)\) +2(\rho+h_1)\p_\no \hat{\ve}_{0,\btau}\eta'(\rho)\\\label{yuning:tildev21}
    &=\((\rho^2-h_1^2)\eta''(\rho)+4\rho\eta'(\rho)\)\p_\no \hat{\ve}_{0,\btau}~\text{on}~\Gamma
 \end{alignat}
To proceed, let us note that $\xi(z):=\int_{-\infty}^z \rho\eta'(\rho)d\rho$ and $\xi'(\rho)$ belongs to $\mathcal{R}_{0,\alpha}$ for some $\alpha>0$. This follows easily from the exponential decay of $\xi'(\rho)= \rho\eta'(\rho)$ and $\int_{-\infty}^\infty \rho\eta'(\rho)\sd \rho =0$ since $\eta'$ is even.
Hence we can integrate \eqref{yuning:tildev21}  to obtain $\tilde{\ve}_{2,\btau}$:
\begin{equation*}
  \partial_\rho \tilde{\ve}_{2,\btau}=- \int_{-\infty}^\rho\((z^2-h_1^2)\eta''(z)+4z\eta'(z)\)\sd z\ \p_\no \hat{\ve}_{0,\btau}\quad \text{on }\Gamma.
\end{equation*}
Therefore we define   $\tilde{\ve}_{2,\btau}$ on $\Gamma(3\delta)$ through
\begin{equation}\label{yuning:tildev22}
  \begin{split}
 \tilde{\ve}_{2,\btau}(\rho,x,t)&=- \p_\no \hat{\ve}_{0,\btau}(\rho,P_{\G_t}(x),t)\int_{-\infty}^\rho\int_{-\infty}^y\((z^2-h_1^2)\eta''(z)+4z\eta'(z)\)\sd z \sd y
  \end{split}
\end{equation}
 in $\Gamma(3\delta)$   and
\begin{equation*}
 \hat{\ve}_{2,\btau}(x,t)\equiv 0,\quad \ve_{2,\btau}^\pm(x,t):=\lim_{\rho \to \pm\infty} \tilde{\ve}_{2,\btau}(\rho,x,t)\quad \text{for all}~(x,t)\in\Gamma(3\delta), \rho\in \R.
\end{equation*}
Moreover, we extend $\ve_{2,\btau}^\pm$ smoothly to $\O \times [0,T_0]$.
It can be verified that the above definitions are compatible with \eqref{yuning:matching2} and \eqref{yuning:ansatz1}.

\subsection{Proof of Lemma \ref{yuning:app1}}\label{subsec:ProofOfLemma}

The proof will heavily rely on Lemma \ref{yuning:lip4}.

\smallskip

\noindent
{\bf Proof of \eqref{eq:ApproxStokesEq2}:} It follows from \eqref{yuning:v0nvanish} that the order $O(\eps^{-1})$ is eliminated from \eqref{yuning:diver1}. Then it follows from \eqref{yuning:newdis1} that the order $O(1)$ is eliminated from \eqref{yuning:diver1} and change the $O(\eps)$ order terms of \eqref{yuning:diver1} to be as in \eqref{yuning:lip1'}. All $O(\eps)$-terms in \eqref{yuning:lip1'} vanish on $\Gamma$ and decay exponentially in $\rho$. Hence it follows from Lemma \ref{yuning:lip4} that they will be included in the term $g_\eps\left(\tfrac{d_\Gamma}\eps-h_\eps,x,t\right)$ with $(g_\eps)_{0<\eps< 1}\in \mathcal{R}_{1,\alpha}^0$.
Moreover, the $O(\eps^2)$ order terms in \eqref{yuning:lip1} can also be included in
a term $\eps^2 \tilde{g}_\eps(x,t)$, where $\tilde{g}_\eps(x,t)$ is uniformly bounded with respect to $(x,t)\in \Gamma(3\delta), \eps\in (0,1]$.
   So we proved \eqref{eq:ApproxStokesEq2}.

\smallskip

\noindent
{\bf Proof of \eqref{eq:ApproxStokesEq1}:}
The construction \eqref{yuning:pvanish} and \eqref{yuning:vtilde0} fulfill (\ref{train:scal-a}) and thus  eliminate the $O(\eps^{-2})$ order terms in the expansion of \eqref{yuning:1.67}.
The formula \eqref{yuning:p0missing} and \eqref{yuning:explitsoluv1} balance the $O(\eps^{-1})$ order terms in (\ref{train:scal-b}) and change  the $O(1)$ order terms in (\ref{train:scal-c}) to be \eqref{yuning:train1new}.
 It can be verified that, in \eqref{yuning:train1new}, the terms that vanish on $\Gamma$, except
 \begin{equation}\label{yuning:1.69}
   (-\Delta_x \ve_0^+ +\nabla_x p_0^+)\tfrac{1+\eta(\rho)}2   + (-\Delta_x \ve_0^- +\nabla_x p_0^-)\tfrac{1-\eta(\rho)}2,
 \end{equation}
 can be included in $\mathcal{R}_{0,\alpha}^0$ and the rest terms are given in \eqref{yuning:v2tilde1}.
According to the construction, Lemma \ref{yuning:v2nplus2} together with formula \eqref{yuning:p0expli1}, \eqref{yuning:tildev22} determines $\tilde{\ve}_2, \hat{\ve}_2, \tilde{p}_1,\hat{p}_1$.
Moreover, they belong to the class $\mathcal{R}_{0,\alpha}^0$ and will be included in $r_\eps$.

Although \eqref{yuning:1.69} vanishes on $\G$ according to the outer expansion \eqref{yuning:1.68}, it does not decay in $\rho$ exponentially.
To treat it,   we replace the factor  $\tfrac{1\pm \eta(\rho)}2$ by some smooth $\zeta^\pm_N\colon \R\to \R$ such that $|\zeta^\pm_N(\rho)| \leq C e^{-\alpha |\rho|}$ and
\begin{equation*}
   \int_{-\infty}^\infty \zeta^\pm_N(\rho)\sd \rho =0, \zeta^\pm_N(\rho)= \tfrac{1\pm \eta(\rho)}2  \quad \text{for all }\rho \lessgtr \pm N.
\end{equation*}
Then for $j=0,1$,
\begin{equation}\label{yuning:xiequa1}
 \left.(-\Delta_x \ve_j^\pm +\nabla_x p_j^\pm)\tfrac{1\pm\eta(\rho)}2\Big|_{\rho = \frac{d_\Gamma}\eps -h_\eps }= (-\Delta_x \ve_j^\pm +\nabla_x p_j^\pm)\zeta^\pm_N(\rho)\right|_{\rho = \frac{d_\Gamma}\eps -h_\eps }
\end{equation}
 since $(-\Delta_x \ve_j^\pm +\nabla_x p_j^\pm)|_{\Omega^\pm(t)}=0$ and $\frac{d_\Gamma}\eps -h_\eps \lessgtr \pm N $ in $\Omega^\mp(t)$ if $N>0$ is chosen large enough.
Hence the term in \eqref{yuning:1.69} equals to
\begin{equation*}
      \sum_{\pm}(-\Delta_x \ve_j^\pm +\nabla_x p_j^\pm)\zeta^\pm_N(\rho)
\in \mathcal{R}_{0,\alpha}^0.
\end{equation*}
    Using Lemma \ref{yuning:v2nplus2} and the previous step, one can verify that the terms  in (\ref{train:scal-d}) as well as those in \eqref{yuning:penlty} belongs to $\mathcal{R}_{0,\alpha}$ and taking into account their level, we can write them in the general form $(\tilde{r}_{\eps})_{0<\eps<1}\in \mathcal{R}_{1,\alpha}$.
    Finally, the   terms $R_\eps^{i'j'ij}$ and $\tilde{R}_\eps^{i'j'k'ijk}$ in \eqref{yuning:class1} come from  \eqref{train:scal-e} by multiplication with $\eps^2$.

\subsection{Expansion of the Allen-Cahn Part}\label{subsec:1.1}

 We shall consider the inner expansion of $c$ as follows:
\begin{equation}\label{yuning:innerexpan}
\begin{split}
  c^{in} (x,t)=\tc^{in}(\rho,s,t)=&\, \theta_0(\rho) + \eps^2 \tc_2(\rho,S(x,t),t) + \eps^3 \tc_3\(\rho,S(x,t),t\)\\
  =:&\, c_0^{in}(x,t)+\eps^2c_2^{in}(x,t)+\eps^3c_3^{in}(x,t).
\end{split}
\end{equation}
 where $\rho$ is defined by
 \begin{equation}\label{yuning:stretch1}
   \rho = \frac{d_{\G_t}(x)}\eps - h_1(S(x,t),t)-\eps h_{2,\eps}(S(x,t),t).
 \end{equation}
Here and in the following $x$ and  $(\rho,s)$ will always be related by \eqref{yuning:stretch1} and $s=S(x,t)$ if both variables appear. Moreover, we will for simplicity write $h_2$ instead of $h_{2,\eps}$ in the following.

It follows from
\begin{equation*}
  S(X_0(s,t)+r\no (s,t))=s\qquad \text{for all }s\in \mathbb{T}^1,t\in [0,T_0], r\in (-3\delta,3\delta)
\end{equation*}
that, by differentiating with respect to $r$,
\begin{equation*}
  (\nabla S)(x,t)\cdot \no(S(x,t),t)=0 \qquad \text{for all }(x,t)\in \Gamma (3\delta).
\end{equation*}
 Therefore it follows from \eqref{yuning:1.46} that
\begin{equation*}
  \left|\tfrac{\no(s,t)}\eps -{(\nabla^\Gamma h_\eps)(r,s,t)} \right|^2
= \frac{1}{\eps^2} + |\nabla^\Gamma  h_\eps(r,s,t) |^2
\end{equation*}
for all $r\in(-3\delta,3\delta), s=S(x,t)\in \mathbb{T}^1, t\in [0,T_0]$.
So we can employ the formula \eqref{yuningliu:formula1} and notation \eqref{yuning:1.58}, \eqref{yuning:1.57} to get
\begin{equation}\label{yuning:expanca1}
 \begin{split}
    \partial_t c_0^{in}(x,t)&= \theta_0'(\rho)\( -\tfrac{V(s,t)}\eps
    - \p_t^\Gamma  h_\eps(r,s,t)
    \)\\
  \nabla c_0^{in}(x,t)&= \theta_0'(\rho)\left( \tfrac{\no(s,t)}\eps  - \nabla^\Gamma h_\eps(r,s,t) \right)\\
 \Delta c_0^{in}(x,t)&=\theta_0''(\rho) \(\tfrac1{\eps^2} +|\nabla^\Gamma h_\eps(r,s,t)|^2\) + \theta_0'(\rho)\left( \tfrac{\Delta d_{\Gamma_t}(x)}{\eps} -\Delta^\Gamma h_\eps(r,s,t)\right).
\end{split}
\end{equation}
Here and in the following all functions as e.g.\
$
 \theta_0',  \tc_{2}, h_1, \nabla^\Gamma \tc_2, \nabla^\Gamma h_1
$
without arguments are evaluated at $\rho, (\rho,s,t)$, $(s,t)$, $(r,\rho,s,t)$, $(r,s,t)$, respectively.
Moreover, Taylor expansion yields
\begin{align*}
  \frac1{\eps^2} f'(c^{in}(x,t))&= \frac1{\eps^2} f'(\theta_0(\rho)) +f''(\theta_0(\rho)) \tc_2(\rho,s,t)+\eps f''(\theta_0(\rho)) \tc_3(\rho,s,t)+ \mathfrak{R}_1,
\end{align*}
where
\begin{equation*}
  \mathfrak{R}_1   = \frac{\eps^2}2f'''\(\theta_0(\rho)+ \xi(\rho,s,t)(\eps^2\tc_2+\eps^3\tc_3)(\rho,s,t)\)\left(\tc_2(\rho,s,t)+\eps\tc_3(\rho,s,t)\right)^2
\end{equation*}
for some $\xi(\rho,s,t)\in [0,1]$.
This together with  \eqref{yuning:expanca1} and \eqref{yuningliu:formula1} applied to $\tc_2,\tc_3$ leads to
\begin{alignat}{1}\nonumber
    &\p_t c^{in}(x,t)-\Delta c^{in}(x,t)+ \eps^2X_0^\ast(\we_1)\cdot\nabla c_0^{in}(x,t)+\frac1{\eps^2} f'(c^{in}(x,t))\\\nonumber
  &= \underbrace{\frac{f'(\theta_0)-\theta_0''}{\eps^2}}_{=0}+\frac {-V-\Delta d_{\Gamma}(x,t)}\eps \theta_0'
  +(\Delta^\Gamma h_1-\p_t^\Gamma h_1)\theta_0'-\p_\rho^2 \tc_2+f''(\theta_0) \tc_2-|\nabla^\Gamma h_1|^2\theta_0''\\\nonumber
  &\ +\eps(\Delta^\Gamma h_2-\p_t^\Gamma h_2+ \no\cdot X_0^\ast(\we_1))\theta_0'-\eps (V+\Delta d_\Gamma(x,t))\p_\rho \tc_2-\eps\big(\p_\rho^2\tc_3+f''(\theta_0)\tc_3\big)\\\label{yuning:monst1}
  &\ -2\eps \nabla^\Gamma h_1\cdot\nabla^\Gamma h_2\theta_0''+\mathfrak{R}_1+\sum_{k=2}^5\eps^k\mathfrak{R}_k + \eps^3(D_t-\Delta_\G)\tc_3.
\end{alignat}
where $\mathfrak{R}_k$, $k=2,\ldots, 5$, are defined as in the proof of Theorem~\ref{thm:InnerExpansion}.


First we eliminate the terms of order $O(\eps^{-1})$ on the right hand side of \eqref{yuning:monst1}.
Because of \eqref{yuning:ex1} and   \eqref{eq:Limit4}, we have
  \begin{equation}\label{yuning:ex2}
    V+\Delta d_\G(x,t)-\no\cdot X_0^\ast(\ve)  = -\eps(\rho+h_\eps)\kappa_1(s,t)+ \eps^2 \kappa_2(s,t)(\rho+h_\eps)^2+\eps^3 \kappa_{3,\eps}(\rho,s,t).
  \end{equation}
 Hence, using Lemma \ref{lemma:va},  we arrive at
\begin{alignat}{1}\nonumber
&\p_t c^{in}(x,t)+\ve^{in}_A(\rho,x,t)\cdot\nabla c^{in}(x,t)+ \eps^2X_0^\ast(\we_1)\cdot\nabla c_0^{in}(x,t)-\Delta c^{in}(x,t) +\frac1{\eps^2} f'(c^{in}(x,t))\\\nonumber
  &=\theta_0'\Big((\rho+h_1)\kappa_1 + \eps h_2\kappa_1 - \eps \kappa_2(\rho+h_\eps)^2-\eps^2 \kappa_{3,\eps}+ \Delta^\Gamma h_1(r,s,t)-\p_t^\Gamma h_1(r,s,t) \Big)\\\nonumber
  &\ +\theta_0'\left(X_0^\ast(\ve_{1,\no} ^\pm)+X_0^\ast(\ve) \cdot \nabla_\Gamma h_1 + X_0^\ast(\ve)\cdot \Lgrad h_1 - X_0^\ast(\divtau \ve) (\rho+h_1)\right)\\\nonumber
 &\ -\p_\rho^2 \tc_2+f''(\theta_0) \tc_2\underbrace{-|\nabla_\Gamma h_1|^2\theta_0''-2 \nabla_\Gamma h_1\cdot \Lgrad h_1 \theta_0''- |\Lgrad h_1|^2 \theta_0''}_{=-|\nabla^\G h_1|^2\theta_0''~\text{due to}~\eqref{Prelim:1.11}}\\\nonumber
 &\ +\eps\big[(\Delta^\Gamma h_2-\p_t^\Gamma h_2+X_0^\ast(\no\cdot \we_1))\theta_0'-(V+\Delta d_\Gamma(x,t))\p_\rho \tc_2-\p_\rho^2\tc_3\\\nonumber
&\qquad +f''(\theta_0)\tc_3-2\nabla^\Gamma h_1\cdot\nabla^\Gamma h_2\theta_0''+ \left(b(\rho,s,t)+X_0^\ast(\ve)\cdot \nabla_\Gamma h_2- X_0^\ast(\divtau \ve) h_2\right)\theta_0'\big] 
 \\\label{yuning:monst1'}
&\ +  \eps X_0^\ast(\ve_{\no})\partial_\rho \tc_2+ R_\eps (\rho,s,t) +\mathfrak{R}_1 +\sum_{k=2}^5\eps^k\mathfrak{R}_k+\eps^3(D_t-\Delta_\G)\tc_3
\end{alignat}
where $\ve$ solves \eqref{eq:Limit1} and $R_\eps(\rho,s,t)$ is given in Lemma \ref{lemma:va}.
Now, we want to eliminate all terms of order $O(1)$.
To this end, we first list all the $O(1)$ terms in the right hand side of \eqref{yuning:monst1'} that are multiplied by   $\theta_0'$  and employ \eqref{Prelim:1.11}:
\begin{equation*}
  \begin{split}
    \mathfrak{A}:=&(\rho+h_1)\kappa_1 +(\Delta^\Gamma h_1-\p_t^\Gamma h_1) +X_0^\ast(\ve_{1,\no} ^\pm)+X_0^\ast(\ve) \cdot \nabla_\Gamma h_1 + X_0^\ast(\ve)\cdot \Lgrad h_1 - X_0^\ast(\divtau \ve)(\rho+ h_1)\\
    =&\underbrace{h_1\kappa_1 +\Delta_\Gamma h_1 -D_t h_1+X_0^\ast(\ve_{1,\no} ^\pm)+X_0^\ast(\ve)\cdot \nabla_\Gamma h_1  - X_0^\ast(\divtau \ve) h_1}_{\text{vanishes according to \eqref{yuning:h1equ} below}}\\
    & +\rho (\kappa_1-X_0^\ast(\divtau \ve))+\Ldelta h_1-\Lt h_1+ X_0^\ast(\ve)\cdot \Lgrad h_1.
  \end{split}
\end{equation*}
Note that the terms in the second last line  depend only on $(s,t)$.
This motivates us to define $h_1$ as the solution of the following equation on $\mathbb{T}^1\times [0,T_0]$:
  \begin{equation}\label{yuning:h1equ}
       D_t  h_1 -X_0^\ast(\ve) \cdot  \nabla_\Gamma h_1  - \Delta_\Gamma h_1  -\kappa_1 h_1 + X_0^\ast(\divtau \ve) h_1=X_0^\ast(\ve_{1,\no} ^\pm)
   \end{equation}
together with the initial condition $h_1|_{t=0}=0$. This changes \eqref{yuning:monst1'} into
\begin{alignat}{1}\nonumber
&\p_t c^{in}(x,t)+\ve^{in}_A(x,t)\cdot\nabla c^{in}(x,t)+ \eps^2X_0^\ast \we_1\cdot\nabla c_0^{in}(x,t)-\Delta c^{in}(x,t) +\frac1{\eps^2} f'(c^{in})\\\nonumber
  =&\theta_0'\Big( \eps h_2\kappa_1 - \eps \kappa_2(\rho+h_\eps)^2+\rho (\kappa_1-X_0^\ast(\divtau \ve)) \Big)\\\nonumber
 & -\p_\rho^2 \tc_2+f''(\theta_0) \tc_2-|\nabla_\Gamma h_1|^2\theta_0''\\\nonumber
 &\underbrace{-2 \nabla_\Gamma h_1\cdot \Lgrad h_1 \theta_0''- |\Lgrad h_1|^2 \theta_0''+\theta_0'(\Ldelta h_1-\Lt h_1 + X_0^\ast(\ve)\cdot \Lgrad h_1)}_{=:\mathfrak{D}_1(d_\Gamma, \rho,s,t)}\\\nonumber
 &+\eps\big[(\Delta^\Gamma h_2(r,s,t)-\p_t^\Gamma h_2(r,s,t)+X_0^\ast(\no\cdot \we_1))\theta_0'-(V+\Delta d_\Gamma(x,t))\p_\rho \tc_2-\p_\rho^2\tc_3\big]\\\nonumber
&+\eps\big[f''(\theta_0)\tc_3-2\nabla^\Gamma h_1(r,s,t)\cdot\nabla^\Gamma h_2(r,s,t)\theta_0''+\left(b+X_0^\ast(\ve)\cdot \nabla_\Gamma h_2 - X_0^\ast(\divtau \ve) h_2\right) \theta_0'\big] 
\\\label{yuning:monst2}
& +  \eps X_0^\ast(\ve_{\no})\partial_\rho \tc_2+ R_\eps (x,t)+\mathfrak{R}_1+\sum_{k=2}^5\eps^k\mathfrak{R}_k+\eps^3(D_t-\Delta_\G)\tc_3 -\eps^2 \kappa_{3,\eps}\theta'_0.
\end{alignat}
By the definition in \eqref{Prelim:1.11}, $\Ldelta h_1,\Lgrad h_1$ and $\Lt h_1$ are smooth functions that vanish on $\Gamma$.
Therefore, using a Taylor expansion with respect to $r=d_\Gamma$ in the coefficients of $\Ldelta ,\Lgrad $, and $\Lt $ und using $d_\Gamma=\eps (\rho+h_1)+\eps^2h_2$ we obtain
\begin{alignat}{1}\label{yuning:1.28}
\mathfrak{D}_1(d_\Gamma,\rho,s,t)&= d_\Gamma \mathfrak{D}(\rho,s,t)+ d_\Gamma^2 \tilde{\mathfrak{D}}(d_\Gamma,\rho,s,t)
=\eps(\rho+h_1) \mathfrak{D}(\rho,s,t)+ \eps^2 \widehat{\mathfrak{D}}_\eps(\rho,s,t),
\end{alignat}
where $\mathfrak{D}(\rho,s,t)=\partial_r \mathfrak{D}_1(0,\rho,s,t)\in \mathcal{R}_{0,\alpha}$ is independent of $h_2$, but $\widehat{\mathfrak{D}}_\eps(\rho,s,t)$ depends on $h_2$.
Moreover,  there is some $C(M)>0$ such that
\begin{equation}\label{yuning:1.32}
  |\widehat{\mathfrak{D}}_\eps(\rho,s,t)|\leq C(M)e^{-\alpha|\rho|}\qquad \text{for all }(\rho,s,t)\in \mathbb{R}\times \mathbb{T}^1\times [0,T_\eps], \eps\in (0,\eps_0]
\end{equation}
provided \eqref{eq:h2epsBound} holds true.
Hence $\mathfrak{D}_1$ can be rewritten as a term of order $O(\eps)$.
In order to eliminate the remaining $O(1)$ terms, we only need to choose $\tc_2$ such that
\begin{equation*}
  -\p_\rho^2 \tc_2+f''(\theta_0(\rho)) \tc_2=|\nabla_\Gamma h_1|^2\theta_0''(\rho)- \theta_0'(\rho)   \rho (\kappa_1-X_0^\ast(\divtau \ve))
\end{equation*}
 for all $(\rho,s,t)\in\R\times \T^1\times [0,T_0]$ and the solvability is guaranteed by Proposition \ref{yuning:solveode}.
With this choice and \eqref{Prelim:1.11}, \eqref{yuning:monst2} reduces to
\begin{alignat}{1}\nonumber
&\p_t c^{in}(x,t)+\ve^{in}_A(\rho,x,t)\cdot\nabla c^{in}(x,t)+ \eps^2X_0^\ast(\we_1)\cdot\nabla c_0^{in}(x,t)-\Delta c^{in}(x,t) +\frac1{\eps^2} f'(c^{in})\\\nonumber
 =&\theta_0'\( \eps h_2\kappa_1 - \eps \kappa_2(\rho+h_1)^2 \)+\eps(\rho+h_1)\mathfrak{D}\\\nonumber
 &+\eps\big[(\Delta_\Gamma h_2-D_t  h_2+\no\cdot X_0^\ast(\we_1))\theta_0'-(V+\Delta d_\Gamma(x,t))\p_\rho \tc_2-\p_\rho^2\tc_3+f''(\theta_0)\tc_3\big]\\\nonumber
 & -2 \eps \nabla_\Gamma h_1\cdot\nabla_\Gamma h_2\theta_0'' + \eps \left[\nabla_\Gamma h_2\cdot X_0^\ast(\ve) - X_0^\ast(\divtau \ve) h_2+b\right] \theta_0'
 +  \eps \no\cdot X_0^\ast(\ve)\partial_\rho \tc_2\\\nonumber
&-\eps^2 \kappa_2(2(\rho+h_1)h_2+\eps h^2_2)\theta'_0 -2\eps\( \nabla_\G h_1\cdot\Lgrad h_2+\Ldelta h_1\cdot \nabla_\G h_2+\Lgrad h_1\cdot\Lgrad h_2\)\theta_0''\\
&+\eps^2\widehat{\mathfrak{D}}_\eps+\eps(\Ldelta -\Lt )h_2\theta'_0\label{yuning:monst2'}
+ R_\eps (x,t)+\mathfrak{R}_1+\sum_{k=2}^5\eps^k\mathfrak{R}_k+\eps^3(D_t-\Delta_\G)\tc_3 -\eps^2 \kappa_{3,\eps}\theta'_0.
\end{alignat}
 Next we start to eliminate the terms of order $O(\eps)$. To this end, we collect all the $O(\eps)$ terms on the right hand side of \eqref{yuning:monst2'} that are multiplied by $\theta_0'$ (but not $\theta_0''$) and do not vanish on $\Gamma_t$. These terms are included in:
\begin{alignat}{1}\nonumber
    \mathfrak{D}_2&:= h_2\kappa_1- \kappa_2(\rho+h_1)^2   +  \Delta_\Gamma h_2 -D_t  h_2 +X_0^\ast(\no\cdot\we_1) 
+  X_0^\ast(\ve)\cdot \nabla_\Gamma h_2 - X_0^\ast(\divtau \ve) h_2+b\\\nonumber
&=\mathfrak{B}+ \kappa_1h_2- \kappa_2h_1^2   + \Delta_\Gamma h_2 - D_t  h_2 +    X_0^\ast(\no\cdot\we_1)
+   X_0^\ast(\ve)\cdot \nabla_\Gamma h_2 - X_0^\ast(\divtau \ve) h_2\\
 &  \quad 
-\mathfrak{B} - \kappa_2(\rho^2+2\rho h_1)+b.\label{yuning:r2}
\end{alignat}
Here all terms expect the last two ones depend only on $(s,t)$. Moreover, we added a term $\mathfrak{B}=\mathfrak{B}(s,t)$ in order to satisfy the compatibility condition for the equation for $\tc_3$ in the sequel. 
To eliminate all terms except the last three in \eqref{yuning:r2},  we choose  $h_2$ as the solution of
\begin{alignat}{1}\label{yuning:h2equ}
   & D_t  h_2    -   \Delta_\Gamma h_2 - \kappa_1h_2 -   X_0^\ast(\ve)\cdot \nabla_\Gamma h_2 + X_0^\ast(\divtau \ve) h_2  =\mathfrak{B}- \kappa_2h_1^2+X_0^\ast(\no\cdot\we_1)
\end{alignat}
on $\T^1\times [0,T_0]$
together with the initial condition $h_2|_{t=0}=0$.
Existence of a solution is proved in Lemma~\ref{lem:Existenceh2}.
Hence \eqref{yuning:r2}
reduces to
$
  \mathfrak{D}_2 =-\mathfrak{B} - \kappa_2(\rho^2+2\rho h_1)+b
$ and \eqref{yuning:monst2'} reduces to
\begin{alignat}{1}\nonumber
  &\p_t c^{in}(x,t)+\ve^{in}_A(\rho,x,t)\cdot\nabla c^{in}(x,t)+ \eps^2X_0^\ast(\we_1)\cdot\nabla c_0^{in}(x,t)-\Delta c^{in}(x,t) +\tfrac1{\eps^2} f'(c^{in})\\\nonumber
  &=\eps(\rho+h_1)\mathfrak{D}+\eps\big[\(b-\mathfrak{B}- \kappa_2(\rho^2+2\rho h_1)\)\theta_0'-(V+\Delta d_\Gamma(x,t)-\no\cdot X_0^\ast(\ve))\p_\rho \tc_2 \\\nonumber
  &\ +{\eps^2 D_t \tc_3 -\eps^2 \Delta_\G \tc_3 -\p_\rho^2\tc_3+f''(\theta_0)\tc_3}-2\nabla_\Gamma h_1\cdot\nabla_\Gamma h_2\theta_0''\big]\\\nonumber
  &\   -\eps^2 \kappa_2(2(\rho+h_1) h_2+\eps h^2_2)\theta'_0 -\eps^2 \kappa_{3,\eps}\theta'_0+\eps^2\widehat{\mathfrak{D}}_\eps+\eps(\Ldelta -\Lt )h_2\theta'_0\\
&\  -2\eps\( \nabla_\G h_1\cdot\Lgrad h_2+\Ldelta h_1\cdot \nabla_\G h_2+\Lgrad h_1\cdot\Lgrad h_2\)\theta''_0 \label{yuning:monst5}
 + R_\eps (x,t) +\mathfrak{R}_1 +\sum_{k=2}^5\eps^k\mathfrak{R}_k.
\end{alignat}
In view of \eqref{yuning:ex2}, we can further rewrite \eqref{yuning:monst5} as
\begin{alignat}{1}\nonumber
  &\p_t c^{in}(x,t)+\ve^{in}_A(\rho,x,t)\cdot\nabla c^{in}(x,t)+ \eps^2X_0^\ast(\we_1)\cdot\nabla c_0^{in}(x,t)-\Delta c^{in}(x,t) +\tfrac1{\eps^2} f'(c^{in})\\\nonumber
  &=\eps(\rho+h_1)\mathfrak{D}+\eps\big[\(b-\mathfrak{B}- \kappa_2(\rho^2+2\rho h_1)\)\theta_0'+\eps^2 D_t \tc_3 -\eps^2 \Delta_\G \tc_3 -\p_\rho^2\tc_3+f''(\theta_0)\tc_3\big]\\\nonumber
  &\ -2\eps\nabla_\Gamma h_1\cdot\nabla_\Gamma h_2\theta_0'' -\eps^2 \kappa_{3,\eps}\theta'_0-\eps^2 \kappa_2(2(\rho+h_1) h_2+\eps h^2_2)\theta'_0+\eps^2\widehat{\mathfrak{D}}_\eps+\eps(\Ldelta -\Lt )h_2\theta'_0\\\nonumber
  &\ -2\eps\( \nabla_\G h_1\cdot\Lgrad h_2+\Ldelta h_1\cdot \nabla_\G h_2+\Lgrad h_1\cdot\Lgrad h_2\)\theta''_0\\
&\  +\eps^2\left((\rho+h_\eps)\kappa_1- \eps  \kappa_2 (\rho+h_\eps )^2-\eps^2 \kappa_{3,\eps} \right)\p_\rho\tc_2\label{yuning:monst6}
   + R_\eps (x,t) +\mathfrak{R}_1 +\sum_{k=2}^5\eps^k\mathfrak{R}_k.
\end{alignat}
It remains to eliminate all the $O(\eps)$ terms by solving the following equation for $\tc_3$:
\begin{alignat}{1}\nonumber
&\eps^2 D_t \tc_3 -\eps^2 \Delta_\G \tc_3 -\p_\rho^2\tc_3+f''(\theta_0)\tc_3=2\nabla_\Gamma h_1\cdot\nabla_\Gamma h_2\theta_0''\\\label{yuning:monst4}
&\quad-\(b- \mathfrak{B}-\kappa_2(\rho^2+2\rho h_1)\)\theta_0' - (\rho+h_1)\mathfrak{D} \quad \text{on }\R\times \T^1\times [0,T_0],
\end{alignat}
which is equivalent to \eqref{yuning:monst4'}.
This is a system treated in Theorem \ref{lem:InftyEstimLinPart} and we shall choose $\mathfrak{B}$ such that the compatibility condition of Theorem  \ref{lem:InftyEstimLinPart} is valid, namely
\begin{equation}\label{yuning:deterb}
\mathfrak{B}(s,t)\int_\R(\theta_0')^2\sd \rho
  =\int_\R\left[  \theta_0'(b-\kappa_2(s,t)\rho^2) + (\rho+h_1)\mathfrak{D} \right]\theta_0'(\rho)\sd \rho
\end{equation}
since $\int_{\R} \rho (\theta_0'(\rho))^2\sd \rho=0$ and $\int_{\R} \theta_0''(\rho)\theta_0'(\rho)\sd \rho =0$.
It is crucial that the right hand side of \eqref{yuning:deterb}, especially $\mathfrak{D}$ defined in \eqref{yuning:1.28}, is uniquely determined by $h_1, \tc_2$ and  the solution of \eqref{eq:Limit1}. Hence $\mathfrak{B}$ is is uniquely determined by $h_1, \tc_2$ and  the solution of \eqref{eq:Limit1}.
So we are able to solve \eqref{yuning:h2equ} and then \eqref{yuning:monst4} without leading to a circular argument. Finally \eqref{yuning:monst5} becomes
\begin{alignat}{1}\nonumber
    &\p_t c^{in}(x,t)+\ve^{in}_A(\rho,x,t)\cdot\nabla c^{in}(x,t)+ \eps^2X_0^\ast(\we_1)\cdot\nabla c_0^{in}(x,t)-\Delta c^{in}(x,t) +\tfrac1{\eps^2} f'(c^{in})\\\nonumber
   &= R_\eps (x,t)  +\sum_{k=2}^5\eps^k\mathfrak{R}_k+\mathfrak{R}_1-\eps^2 \kappa_{3,\eps}\theta'_0 -\eps^2 \kappa_2(2(\rho+h_1) h_2+\eps h^2_2)\theta'_0 +\eps^2\widehat{\mathfrak{D}}_\eps\\\nonumber
  &\ +\eps(\Ldelta -\Lt )h_2\theta'_0  -2\eps\( \nabla_\G h_1\cdot\Lgrad h_2+\Ldelta h_1\cdot \nabla_\G h_2+\Lgrad h_1\cdot\Lgrad h_2\)\theta''_0\\\nonumber
&\ +\eps^2\((\rho+h_\eps)\kappa_1- \eps  \kappa_2 (\rho+h_\eps )^2-\eps^2 \kappa_{3,\eps}\)\p_\rho\tc_2\\\label{yuning:1.43}
&=  R_\eps (x,t) +\sum_{k=2}^5\eps^k\mathfrak{R}_k
+\mathfrak{R}_1+\mathfrak{R},
\end{alignat}
which implies \eqref{eq:ExpandAC}.
